\documentclass[reqno,11pt]{amsart}
\usepackage{geometry}
\usepackage{mathtools,amssymb,amsthm,mathrsfs,color,lineno,paralist,graphicx,float}
\usepackage[colorlinks,
linkcolor=red,
anchorcolor=green,
citecolor=blue, 
]{hyperref}
\usepackage{comment}

\usepackage[T1]{fontenc}
\usepackage[utf8]{inputenc}

\usepackage{calc}
\definecolor{bleu1}{RGB}{0,57,128}
\def\bleu1{\color{bleu1}}

\usepackage{etoolbox}
\patchcmd{\section}{\normalfont}{\normalfont \bleu1}{}{}
\patchcmd{\subsection}{\normalfont}{\normalfont \bleu1}{}{}
\patchcmd{\subsubsection}{\normalfont}{\normalfont \bleu1}{}{}

\newtheorem{PProp}{\bleu1 Proposition}

\newtheorem{Main}{\bleu1 Theorem}

 \usepackage{ulem}
\def\cT{\mathcal T}
\usepackage{graphicx,psfrag,epsfig}
\usepackage{amsmath}
\usepackage{amsthm}
\usepackage{amsfonts}
\usepackage{amssymb}
\usepackage{tikz}
\usetikzlibrary{arrows,shapes,trees}
\usepackage{graphicx}
\usepackage{hyperref}
\usepackage{amssymb}
\usepackage{amsfonts}
\usepackage{amsthm}
\usepackage{amsmath}
\usepackage{epstopdf}
\usepackage{mathrsfs}

\def \etan{{\mu}}
\def \sigman{{\sigma}}
\def \rn{{\bar r}}

\def \hA{{\widehat A}}
\def \id{{\rm Id}}
\def \hB{{\widehat  B}}

\def\tp{\tilde p}
\def\tq{\tilde q}
\def\SL{\text{SL}}

\def\cO{{\mathcal{O}}}
\def \cM1{{\mathcal M_1}}
\def\cK{{\mathcal{K}}}
\def\cM{{\mathcal{M}}}
\def\cV{{\mathcal{V}}}
\def\cN{{\mathcal{N}}}
\def \bA{{\bar A}}
\def \bB{{\bar B}}
\def \hA{{\widehat A}}
\def \hB{{\widehat  B}}

\def \bm{{\bar m}}

\def\<{{\left<}}
\def\>{{\right>}} 

\def \Id{{\text{Id} }}

\def\hA{\widehat{A}}
\def\hB{\widehat{B}}
\def\bk{\bar k}

\def\tA{\widetilde{A}}
\def\tB{\widetilde{B}}

\def\bff {{\bf f}}
\def\bfg {{\bf g}}
\def \bfh {{\bf h}}
\def\bfL {{\bf L}}

\usepackage{mathrsfs}
\usepackage{color}
\definecolor{dgreen}{rgb}{0.1,0.6,0.1}

\definecolor{bluegreen}{rgb}{0.1,0.5,0.2}
\definecolor{bpurple}{rgb}{0.74,0.2,0.64}

\def\black{\color{black}}

\def\pelta{\partial}

\definecolor{orange}{rgb}{0.8, 0.33, 0.0}

\def\cM{\mathcal M}
\def\cN{\mathcal N}
\def\cO{\mathcal O}

\definecolor{orange}{rgb}{1,0.5,0}

\definecolor{orange}{rgb}{1,0.5,0}

\definecolor{bluegray}{rgb}{0.6, 0.6, 0.8}

\definecolor{bluered}{rgb}{0.6,0.3,0.4}

\newtheorem{sublemma}{\bleu1 Sublemma}

\newtheorem{lemma}{\bleu1 Lemma}

\newtheorem*{theorem*}{\bleu1 Theorem}

\newtheorem{corollary}[lemma]{\bleu1 Corollary}
\newtheorem{proposition}[lemma]{\bleu1 Proposition}
\newtheorem*{proposition*}{\bleu1 Proposition}
\newtheorem{remark}{\bleu1 Remark}
\newtheorem{definition}{\bleu1 Definition}
\newtheorem{example}{\bleu1 Example}
\newtheorem{question}{\bleu1 Question}
\newtheorem{remarki}{\bleu1 Observation}

\newtheorem*{Atheorem*}{Theorem}

\def\<{{\langle}}
\def\>{{\rangle}}

\def\bfp{{\bf p}}
\def\bfq{{\bf q}}
\def\bfV{{\bf V}}
\def\bfW{{\bf W}}

\def\al{{\alpha}}
\def\a{{\alpha}}
\def\la{{\lambda}}
\def\be{{\beta}}

\def\ga{{\gamma}}
\def\gax{{\xi}}
\def\de{{\delta}}

\def\eps{{\varepsilon}}

\def\Si{{\Sigma}}

\def\ett{{\delta}}

\def\({{\langle}}
\def\){{\rangle}}

\def\NN{{\mathbb{N}}}
\def\RR{{\mathbb{R}}}

\def\ZZ{{\mathbb{Z}}}
\def\TT{{\mathbb{T}}}

\def\N{{\mathbb{N}}}
\def\R{{\mathbb{R}}}

\def\Z{{\mathbb{Z}}}
\def\T{{\mathbb{T}}}

\def\cO{{\mathcal{O}}}
\def\cC{{\mathcal{C}}}

\def \cM1{{\mathcal M_1}}
\def\cK{{\mathcal{K}}}
\def\cM{{\mathcal{M}}}
\def\cV{{\mathcal{V}}}
\def\cN{{\mathcal{N}}}
\def \bA{{\bar A}}
\def \bB{{\bar B}}
\def \hA{{\widehat A}}
\def \hB{{\widehat  B}}
\def \tA{{\widetilde A}}
\def \tB{{\widetilde  B}}

\def \res{{\mathcal Q}}

\def \bm{{\bar m}}

\def \ve{{\varepsilon}}

\def \tf{{\tilde f}}
\def \tg{{\tilde g}}

\def\<{{\left<}}
\def\>{{\right>}} 

\def \Id{{\text{Id} }}
\def \akl{{\a_{k,l}}}

\title[]{KAM-rigidity for parabolic affine abelian actions}
\author{Danijela Damjanovi\'c}
\address[Danijela Damjanovi\'c]{Department of Mathematics, Kungliga Tekniska H\"{o}gskolan, Lindstedtsv\"{a}gen 25, SE-100 44  Stockholm, Sweden}
\email{ddam@kth.se }
\author{Bassam Fayad}
\address[Bassam Fayad]{Department of Mathematics, University of Maryland, 2307 Kirwan Hall, College Park, USA}
\email{bassam@umd.edu}
\author{Maria Saprykina}
\address[Maria Saprykina]{Department of Mathematics, Kungliga Tekniska H\"{o}gskolan, Lindstedtsv\"{a}gen 25, SE-100 44  Stockholm, Sweden}
\email{masha@kth.se}


\usepackage{ulem} 
\begin{document} 
\begin{abstract} We show the following dichotomy for a linear parabolic $\Z^2$-action $\rho_L$ on the torus with at least one step-2 generator:
\begin{itemize}
\item[$(i)$] Any affine $\Z^2$-action with  linear part $\rho_L$ has  a $\Z$-factor that is either identity or genuinely parabolic, and is thus not KAM-rigid, or
\item[$(ii)$] Almost every affine $\Z^2$-action with linear part $\rho_L$ is KAM-rigid  under volume preserving perturbations. 
\end{itemize} 
\end{abstract}
\maketitle

\setcounter{tocdepth}{2}

\color{black} 
{
  \hypersetup{hidelinks}
  \tableofcontents
}

\newpage 

\section{Introduction, statements and overview of the main proofs}

\subsection{Background and context}

A smooth (by which we mean $C^\infty$)  $\Z^k$-action $\rho$ on a smooth manifold $M$ is said to be {\it locally rigid} if there exists a neighborhood $\mathcal U$ of $\rho$ in the space of smooth  $\Z^k$-actions on $M$, such that for every $\eta\in \mathcal U$ there is a smooth diffeomorphism $h$ of $M$ such that $h\circ \rho(g)\circ h^{-1}=\eta(g)$, for all $g\in \Z^k$. 

When the rank of the acting group is $k=1$,  we are in the realm of classical dynamics ($\Z$-actions), where local rigidity in this strong form is not known to occur. Moreover, it is known that  for affine maps on the torus local rigidity does not occur.

The only known situation
in classical dynamics where a weaker form of local rigidity is proved,
is the case of  toral translations  $T_\a$ on $\T^d$ with  Diophantine frequency vectors $\a$ (we exclude rigidity modulo infinite moduli from this discussion). Indeed, it follows from Arnold's normal form for perturbations of toral translations, that a volume preserving  perturbation of $T_\a$  with a  Diophantine average translation vector $\a$ is smoothly conjugated to $T_\a$ \cite{A}.  We will call this phenomenon {\it KAM-rigidity}.  

The situation is dramatically different for $\Z^k$-actions with $k\geq 2$, where local rigidity is more common. 

For Anosov (hyperbolic) actions, an important breakthrough was the proof of  local rigidity by Katok and Spatzier \cite{KS}. The main tool in this context is the use  of the action's invariant geometric structures \cite{GK, KS}, which after that proved useful in obtaining local rigidity for more general classes of partially hyperbolic actions with such geometric structures  \cite{NT, DK3, W1, W2, VW}. 
{We note that there are many other local and global rigidity results for abelian partially hyperbolic actions than the ones mentioned above; we refrain from citing them all as our focus in this paper  will be on local rigidity in the absence of any form of hyperbolicity and where there are no robust invariant geometric structures. }


In this context two famous manifestations of rigidity for $\Z^k$-actions, are: KAM-rigidity of {\it simultaneously} Diophantine torus translations  (see Definition \ref{definition sim})
 \cite{M,DF,WX,P}, and local rigidity for higher rank linear (and affine) partially hyperbolic actions on the torus \cite{DK}.  

Simultaneously Diophantine torus translations generate $\Z^k$-actions which may have no Diophantine elements at all, so the result for single Diophantine translations does not apply,  and one is forced to use commutativity of different action generators in a crucial way in order to obtain KAM rigidity. 

For linear partially hyperbolic actions considered in \cite{DK}, the crucial assumption which leads to smooth rigidity is that such $\Z^k$-actions are of {\it higher rank}. A $\Z^k$-action by toral automorphisms is {\it higher rank} if there is a $\Z^2$ subgroup such that all of its non-zero elements act by ergodic automorphisms. This condition is equivalent to the absence of $\Z$-factors of the action, namely: a higher rank $\Z^k$-action does not factor (possibly up to a finite index subgroup) to an action generated by a single automorphism of a (possibly different) torus. The condition is equivalent also to the exponential mixing for the action, which plays a crucial role in the proof of local rigidity in \cite{DK}. Recently, it was announced in \cite{W3} that exponential mixing leads to local rigidity for large classes of partially hyperbolic affine actions.

The above mentioned two classes of actions on the torus, elliptic ones generated by translations on one hand, and partially hyperbolic ones on the other hand, lie in the general class of {\it affine} $\Z^k$-actions on the torus. Affine actions are actions generated by affine maps, and an affine map is a composition of a linear map and a translation. Such actions can be dynamically very different: they can be elliptic (the linear part of the action is the identity), or partially hyperbolic (the linear part contains a partially hyperbolic map, i.e., a map with some eigenvalues outside unit circle), or parabolic (the linear part acts by parabolic maps i.e. maps which have all eigenvalues 1), or can combine all these features. 
If the linear part contains a root of the identity, we take a finite index subgroup in the acting group which brings us to the  general description above. In what follows we always assume that roots of the identity have been eliminated by passing to a finite index subgroup.  

In this paper we focus on the most intricate and most surprising case of {\it parabolic actions}.  In this case, single elements of the action (even under Diophantine conditions) are {\it not} KAM rigid, nor is there any mixing for the linear part of the action. 



First we define a property of linear actions which will distinguish between actions that can be linear parts of  KAM rigid affine actions,  and those which cannot. 

A linear $\Z^2$-action $\rho_L$ on $\mathbb T^d$ is {\it unlocked} if there is an affine $\Z^2$-action $\rho$ with linear part $\rho_L$ such that every $\mathbb Z$-factor of $\rho$, if it exists, is generated by a non-trivial translation. 
Thereby, we say a linear $\Z^2$-action $\rho_L$ on $\mathbb T^d$ is {\it locked}  if for any affine action $\rho$ with linear part $\rho_L$, there exists a $\mathbb Z$-factor which is {\it not} a translation $T_\alpha$, $\alpha\neq 0$,  see Definition \ref{def_stiff}.

Our analysis leads us to ask the following classification question for general affine $\Z^2$-actions on the torus: 

\begin{question} \label{qqq}  Is it true that for any linear $\Z^2$-action $\rho_L$ on $\mathbb T^d$ there is the following dichotomy: 
\begin{itemize}
\item[(i)] The action $\rho_L$ is locked, or

\item[(ii)] The action $\rho_L$ is unlocked, and almost every affine action $\rho$ with the linear part $\rho_L$ is KAM-rigid. 

\end{itemize}
\end{question}





The previously mentioned classification results provide the positive  answer to the above question for large classes of actions. For simultaneously Diophantine translation actions the linear part is the identity and thus is unlocked, so it fits into the case $(ii)$ of the question. Since simultaneously Diophantine condition is  a full measure condition, the works \cite{M,DF,WX,P} 
give the positive answer to the question in this case.  
For affine actions with higher rank partially hyperbolic linear part the linear part is unlocked, since it has no $\mathbb Z$-factors at all. 
Hence, these actions also fall in category $(ii)$. By \cite{DK}, such actions have the local rigidity property (and, therefore, KAM rigidity property as well; in fact, KAM rigidity holds for all affine actions of this kind, not just for almost all). 
Moreover, these two classes combined also give the positive answer to the question:  \cite{DF} proves KAM rigidity for a full measure set of affine $\Z^2$-actions on the torus whose linear part is  a direct product of a higher rank action 
and the identity.

The main goal of this paper is to develop new tools to study this question for  {\it parabolic} actions. Our main result states that the answer to the question is affirmative for parabolic actions containing a step-2 element (see Definition \ref{def_step}). We hope that the tools developed in this paper will allow to address the classification question \ref{qqq} in full generality. 

In this work the step-2 assumption on at least one element of the action is important in the proof, as we will discuss in detail in \S \ref{s_overview}. In a nutshell, the reason is that if the cohomological equation above a step-2 affine map with a Diophantine translation part has a solution, then the solution is tame. This is not the case for higher step affine maps \cite{DFS_nontame}. 

\color{black}



\subsection{Main results and representative examples}


 We will  be interested in the problem of local rigidity of volume preserving perturbations  of affine parabolic $\Z^2$-actions on the torus $\T^d$, where $d\in \N$. 
By $\lambda$ we will denote the Haar measure on the torus $\T^d$. 


We begin by defining some basic notions.

\begin{definition}\label{def_step} We say that  $A\in \SL(d,\Z)$ is step-$S$ parabolic if 
$
(A-\Id)^S=0,
$ and $
(A-\Id)^{S-1}\ne 0.
$
An affine map of $\mathbb T^d$ with the linear part $A$ is defined via 
$$
a(x)=Ax +\al \mod 1, \quad x\in \mathbb T^d,
$$ \color{black}
where $\al\in \mathbb R^d$. An affine map $a$ is said to be  step-$S$ if its linear part $A$ is step-$S$.  We denote by  $\rm Aff_S(\T^d)$  the space of all parabolic affine maps of step at most $S$.

A step-$S$ affine parabolic $\Z^2$-action on $\T^d$ is a homomorphism $\rho:\Z^2\to \rm Aff_S(\T^d)$. We denote by $\rho_L$ the action generated by the linear part of $\rho$. 
\end{definition}

Note that every affine map whose linear part is a  unipotent matrix $A\in \SL(d,\Z)$ (i.e., a matrix with all eigenvalues 1) is step-$S$ parabolic for some $S\le d$.


Let $\rho$ be an affine parabolic $\Z^2$-action on the torus $\T^d$. Then $\rho$ is generated by  two commuting affine maps $a$ and $b$, and we will denote $\rho$ simply by its generators as $\(a,b\)$. The  linear parts $A$ and $B$ of  $a$ and $b$, respectively, also commute. But the commutativity of the linear parts is not enough to guarantee the commutativity of  $a$ and $b$.  If $a(x)=Ax+ \alpha \mod 1$ and $b(x)=Bx+ \beta \mod 1$, then the commutator $[a, b]$ is a translation on $\mathbb T^d$ by the vector $(A-\Id)\beta-(B-\Id)\a$, which needs to be an integer in order to have  $[a, b]=\Id$. We may choose the lifts of $a$ and $b$ to $\mathbb R^d$ so that the integer vector $(A-\Id)\beta-(B-\Id)\a$ is trivial, and so without loss of generality we will always consider the lifts  $Ax+ \alpha$ and $Bx+ \beta$ of $a$ and $b$, respectively,  such that $(A-\Id)\beta=(B-\Id)\a$. We will work with these lifts the whole time, and we denote them by the same letters $a$ and $b$, respectively. 

We will also use the shorthand notation $A+\alpha$ to denote the affine map on $\T^d$ with the linear part $A$ and the translation part $\alpha$. 

\color{black}
\begin{definition} We denote by $\cT(A,B)$ the set of possible translation parts $(\a,\beta)$ in the affine actions with linear part $\(A, B\)$, that is
$$
\cT(A,B):=\{\alpha, \beta\in \mathbb R^d \mid (A-\Id)\beta=(B-\Id)\a\}.
$$
\end{definition}
Given an affine map $a(x)=Ax +\al \mod 1$, a perturbation of $a$ is a diffeomorphism of the torus which can be lifted to $\mathbb R^d$, where it has the form 
$$
F(x)= A(x)+\a+f(x)
$$ 
for some small $\Z^d$-periodic vector-valued function $f$. Therefore we will simply write $a+f$ for a small perturbation of $a$. We will be interested in the set of smooth volume preserving perturbations, which we denote by  $ \text{Diff}^\infty_\lambda (\T^d)$.

We define now the notion of {\it KAM rigidity} which is central in this work. Note that an affine  parabolic action always has a translation factor. Therefore, one cannot hope that  any form of local rigidity, stronger than the one available for translations, can hold. Only local rigidity of KAM type can be expected for these actions. 

\begin{definition} \label{KAM local rigidity}
We say that an affine $\Z^2$-action $\(a,b\)$ is {\it KAM-rigid under $\lambda$-preserving perturbations}, if there exists  $\sigma \in \N$, $r_0 \in \N$, $r_0\geq \sigma$, and  $\eps>0$ satisfying the following:

 If $r\geq r_0$ and $\(F,G\)=\(a+f,b+g\)$ is a smooth $\la$-preserving $\Z^2$-action such that  
 \begin{equation}\label{eq_System0}
 \|f\|_r\leq \eps, \quad \|g\|_r\leq \eps, 
 \quad \widehat f:= \int_{\T^d} f d\lambda=0, \quad \widehat g:=\int_{\T^d} g d\lambda=0,
\end{equation}    
 then there exists $H=\Id +h \in \text{Diff}^\infty_\lambda (\T^d)$ such that $\|h\|_{r-\sigma}\leq C(a,b)\, \eps$ and 
\begin{equation*}\label{eq_System} 
 H \circ (a+f) \circ H^{-1} = a, \quad H \circ (b+g) \circ H^{-1} = b,
\end{equation*} 
 where $C(a,b)$ is a constant depending only on the action $\(a,b\)$.
\end{definition} 

In this paper,  we will use the term {\it KAM-rigid} for short reference to  {\it KAM-rigid under $\lambda$-preserving perturbations} since this will be the only context in which we place ourselves. 

It is not difficult to find examples of parabolic commuting actions $\(A,B\)$ such that all the affine actions with this linear part are not KAM-rigid. Indeed, the commutation condition  may force the affine action to have, for any choice of $(\a,\beta)\in \cT(A,B)$, a rank-one factor, to which Arnold's KAM-rigidity cannot be applied. As we will see in the following examples, this happens if the 
rank-one factor of the affine action is either identity or genuinely parabolic (i.e., has a non-trivial linear part). We call such pairs $\(A,B\)$ locked (Definition \ref{def_stiff} below). 

\color{black}

Let $E_{ij}$ denote the integer matrix which has $1$ in the position $(i,j)$, the rest of the elements being 0. 

\begin{example}[Affine actions with identity as a rank-one factor] \label{ex.id} 
Consider the linear action on $\T^3$ generated by $A=\Id+E_{21}$, $B=\Id+E_{31}$. The commutation condition  implies that for all $(\a,\beta)\in \cT(A,B)$ we have $\a_1=\beta_1=0$. Hence, any affine action $\(A+\a,B+\beta\)$, restricted to a sub-torus corresponding to the variable $x_1$, equals identity (we say that the action projects to identity on the torus $\mathbb T_1$ spanned by variable $x_1$, i.e., the action has the identity factor).  Consequently, such affine action is not KAM-rigid, as explained below. 
\end{example} 

To see why the action in Example \ref{ex.id} is not KAM-rigid, we can use the following general fact.

\begin{proposition} \label{propid} If a parabolic commuting action $\(A,B\)$ is lower triangular, and $\(A+\a,B+\beta\)$ is such that $\a_1=\beta_1=0$ then  $\(A+\a,B+\beta\)$ is not KAM-rigid. \end{proposition}

\begin{proof} Keep $b=B+\beta$ unchanged and perturb $a=A+\a$ to $F(x_1,\ldots,x_d)=Ax+\a+(0,\ldots,0,\eps \sin(2\pi x_1))$. The maps $b$ and $F$ commute, and 
$F$ satisfies \eqref{eq_System0}. To see that $A+\a$ is not conjugated to $F$ by a volume preserving conjugacy, consider the special two-dimensional case: $a(x_1,x_2)=A(x_1,x_2)=(x_1,x_2+x_1)$ and $F(x_1,x_2)=(x_1,x_2+x_1+\eps \sin(2\pi x_1))$
(the general case is not different from it). 
Define the circle diffeomorphism $g: x_1\mapsto x_1+\eps \sin(2\pi x_1)$ and the conjugacy $H(x_1,x_2)=(g(x_1),x_2)$. Clearly, 
$H$ does not preserve area, and it is easy to see that, up to translation, $H$ is the only conjugacy between $a$ and $F$. 
\end{proof}

\begin{remark} Note that if $\langle a,b\rangle$ has identity as a rank-one factor of dimension at least $2$, then it is straightforward to perturb the action as in \eqref{eq_System0} so that there is no 
conjugacy at all with the affine action. To do this, it is enough to perturb the identity factor itself in a volume preserving way such that \eqref{eq_System0} holds. \end{remark} 

\begin{example}[Affine actions with a genuinely parabolic rank-one factor] \label{ex.rankone}
Let us assume that $$B=\Id+E_{21}+E_{32}.$$ If 
$A$ is lower triangular and $(A-\Id)$ does not contain neither $E_{2j}$ nor $E_{3j}$ for any $j$, 
then for any $(\a,\beta)\in \cT(A,B)$ we have that $\a_1=\a_2=0$, and $A+\a$ acts as identity on the two-torus obtained by the projection on $(x_1,x_2)$.
 Then the affine action $\(A+\a,B+\beta\)$ has the skew shift $(x_1,x_2)\mapsto (x_1+\beta_1,x_2+x_1+\beta_2)$ of the two-torus as a rank-one factor, and is thus not KAM-rigid, as explained below. 
\end{example} 

To see why the action in Example \ref{ex.rankone} is not KAM-rigid, we can use the following general fact. 

\begin{proposition} \label{prop.rankone} For any $r\in \N$, for any parabolic affine map $A+\a$ on $\T^d$ with $A\neq \Id$, for any $\eps>0$ there exists $f$ such that 
$$
 \|f\|_r\leq \eps,  
 \quad \widehat f:= \int_{\T^d} f d\lambda=0,
 $$   
and $A+\a+f \in \text{Diff}^\infty_\lambda(\T^d)$ is  not conjugated to $A+\a$. 
\end{proposition}

\begin{proof} Without loss of generality, we can consider the case of the two-dimensional skew shift $(x_1,x_2)\mapsto (x_1+\beta_1,x_2+x_1+\beta_2)$. This map can be perturbed into $(x_1,x_2)\mapsto (x_1+\beta_1+\eps \sin(2\pi(x_1+x_2)),x_2+x_1+\beta_2)$, which is a shifted classical standard map that is not conjugate to a skew shift.    
\end{proof}

The phenomena in Examples \ref{ex.id} and \ref{ex.rankone} can be subsumed under the existence, for any choice of $(\a,\beta)\in \cT(A,B)$, of a rank-one factor for the affine action $\(A+\a,B+\beta\)$ that is either identity or a genuinely parabolic action, which overrules KAM-rigidity. This motivates the following definition.

\begin{definition}[Locked actions] \label{def_stiff}
When the commuting linear action $\(A,B\)$ is such that 
 for any choice of $(\a,\beta)\in \cT(A,B)$,  the affine action $\(A+\a,B+\beta\)$ has  a rank-one factor that is either identity or a genuinely parabolic action (i.e., has a non-trivial linear part), we say that $\(A,B\)$ is locked. We call the action unlocked if it is not locked.
\end{definition}

An immediate corollary of Propositions \ref{propid} and \ref{prop.rankone} is the following. 
\begin{corollary} \label{cor.stiff} 
If $\(A,B\)$ is locked, then for any choice of $(\a,\beta) \in \cT(A,B)$, the action of $\(a,b\)$ is not KAM-rigid.
\end{corollary}

The main result of this paper is to show that besides the locked actions, for actions having (at least) one step-2 generator,  KAM-rigidity under $\lambda$-preserving perturbations holds almost surely in the choice of the translation part. We formulate this dichotomy as follows.

\begin{Main} \label{th.main} 
Given a commuting action $\(A,B\)$ of parabolic matrices, where $A$ is step-2, we have the following dichotomy. 
\begin{itemize}
\item[(i)] Action $\(A,B\)$ is locked, thus for any choice of $(\a,\beta) \in \cT(A,B)$, the action of $\(a,b\)$ is not KAM-rigid.\item[(ii)] Action $\(A,B\)$ is unlocked, and for almost every choice of $(\a,\beta)\in \cT(A,B)$, the action of $\(a,b\)$ is ergodic and KAM-rigid under volume preserving perturbations. 
\end{itemize}
\end{Main}

In the case of step-2 actions, we have a more stringent alternative. 
\begin{Main} \label{almostsure} 
Given a commuting pair $\(A,B\)$ of step-2 parabolic matrices, we have the following dichotomy. 
\begin{itemize}
\item[(i)] For any choice of $(\a,\beta)\in \cT(A,B)$, the action of $\(a,b\)$ has a rank-one factor that is identity, and is therefore not ergodic and not locally rigid.
\item[(ii)] For almost every choice of $(\a,\beta)\in \cT(A,B)$, the action of $\(a,b\)$ is ergodic and KAM-rigid under volume preserving perturbations. 
\end{itemize}
\end{Main} 

Corollary \ref{cor.stiff} states that $(i)$ impedes KAM-rigidity.  The proof of the dichotomy between $(i)$ and $(ii)$ is the main result of this paper that we formulate more precisely {\it via} Proposition \ref{prop dioph} and  Theorem \ref{mainC} of the next section.

We note that the statements remain true if we replace the preservation of the volume  $\lambda$ (and also averages with respect to the volume)
by preservation of any common invariant measure. 

\color{black}

In the following  observations we discuss the relevance of the assumptions of the main theorems.

\begin{remarki}[KAM-rigidity {\it vs.} local rigidity: why do we need Diophantine conditions?] {\rm When a linear $\Z^2$-action is not higher rank, it has a rank-one factor that we can represent by a pair $\({\rm Id},C\)$. The absence of rigidity of a single linear map $C$ implies that the local rigidity in this case can only be considered for affine actions. This work 
 treats rigidity of unlocked affine parabolic actions under Diophantine conditions on the translation vectors of the action. Roughly speaking, we put the Diophantine conditions on all the frequency vectors associated to sub-tori on which some element of the action acts as a translation. As we will see in the next section, there may be a finite or infinite number of such conditions. However, there will unavoidably be some Diophantine conditions that must be satisfied, thus only KAM type rigidity can be considered. The necessary set of Diophantine conditions comes from the fact that the affine action always has a translation part (due to the existence of a common eigenspace of eigenvalue 1 for the commuting linear pair, see Definition \ref{def_translation} and \S \ref{s_case1}). The other Diophantine conditions (cf. Definition \ref{definition resonance})  that are used are natural conditions that play a crucial role in the proof (cf. \S \ref{s-case2}), although we do not see for the moment how to show that they are necessary for the result to hold (see Question \ref{qq1}).}
\end{remarki}

\begin{remarki}[On higher step parabolic actions: why do we assume that one generator is step-2?]
{\rm The assumption that one element of the action is step-$2$ plays a heavy role in the proof.  In particular, it is crucial in defining a tame candidate for a conjugacy at each step of the inductive KAM conjugacy scheme.  
More details about the use of this assumption will be given in \S \ref{s_overview}.  We note that working under this, probably restrictive, assumption 
requires introducing some new ideas and techniques. Moreover, some phenomena, like resonances, appear only along sub-tori where both generators are step-2 (see \S \ref{Diophantine conditions}). This is why we prefer to focus on the special case described by Theorem \ref{th.main} and keep the study of the general case for a future work. }



\end{remarki}

\begin{remarki}[From parabolic actions to general affine actions: why do we focus on parabolic actions?] {\rm  If a non higher rank linear action $\(A,B\)$ has a rank-one factor $\(\Id,C\)$ where $C$ does not have  $1$ as an eigenvalue, then, due to the commutation constraint, $\Id$ is still a factor for any affine action with the linear part $\(A,B\)$. As a consequence, local rigidity would fail for the affine actions the same way it fails for the linear one. Based on this argument and on the fact that higher rank actions (i.e., those for which all the elements of the linear part are ergodic, thus partially hyperbolic) are locally rigid \cite{DK}, the case of parabolic actions naturally appears as the main problem to settle in order to  give a general classification of affine abelian actions on the torus in terms of KAM local rigidity.  }
\end{remarki}

\begin{remarki}[The case of $\mathbb Z^k$ actions for  $k\ge 3$] {\rm The same methods we use here provide the KAM-rigidity result for certain classes of affine  $\mathbb Z^k$-actions with $k\ge 3$, as well, see  Remark \ref{z-k} for more details.
It is useful to note that the higher the rank $k$ of the acting group is, the "more locked" the action can become. An example of maximal rank parabolic abelian linear action on $\mathbb T^4$ is the $\mathbb Z^4$-action  generated by $\Id+E_{12}, \Id+E_{14},\Id+E_{32}$ and $\Id+E_{34}$. This action is completely locked in the strongest possible sense: there are no affine non-linear   $\mathbb Z^4$-actions which have this action as a linear part, at all.  We expect that the same holds for any maximal rank parabolic abelian linear action on any $\mathbb T^d$, they are locked. In these cases there is no KAM-rigidity. This points to the fact that  parabolic   $\mathbb Z^2$-actions are the most common situation in which we could expect to have KAM-rigidity.  }
 
 \color{black}

\end{remarki}

\begin{remarki}[The role of commutativity in the KAM-rigidity. The parabolic higher rank trick] 

{\rm As in \cite{M} and \cite{DK}, our proof of local rigidity relies on a KAM inductive conjugacy scheme (see \S \ref{s_overview} for an outline of the scheme). At each step of the scheme, a system of cohomological equations must be solved up to a quadratic error. 

In \cite{M}, each equation of the system is a cohomological equation above a circle rotation. Hence each individual equation has a formal solution provided vanishing of averages, but this solution may not be tame because each individual angle is not necessarily  Diophantine. 
The main observation by Moser is that a cocycle relation forces the formal solutions to coincide and to be tame. This implies that the commutation relation allows one to find a tame solution to the system up to a quadratic error.

In \cite{DK}, the system of cohomological equations consists of individual equations that have tame solutions modulo a countable set of obstrcutions. The commutation, or the higher rank trick, is used to show that these obstructions can be removed up to a  quadratically small error. 

As it will be explained in detail in \S \ref{s_overview}, our approach to proving KAM-rigidity for parabolic actions combines these two mechanisms of local rigidity. 

The main challenge in our work is to replace the partially hyperbolic higher rank trick by a parabolic one. To explain this a little better, we risk a technical description that may look obscure now but that will become much clearer from the detailed overview of the proof in  \S \ref{s_overview} as well as from the introduction of Section \ref{plan_of_proof}. 

The higher rank trick usually relies on the exponential growth of integer vectors (Fourier frequencies) under the dual action of partially hyperbolic matrices. Also, the mechanism that lies at the heart of the higher rank trick of \cite{DK} is the following consequence of the partial hyperbolicity of the action: for an integer vector $\bar{m}$ that is lowest (with smallest norm) on its orbit under the dual action $\bA$ of the first generator of the action, it is possible to iterate by $\bB$ 
in one of the two directions (future or past) so that $\|\bA^k\bB^lm\|$ be always larger than $c\|m\|$.

A main difficulty in our work is to replace the above argument by the fact that parabolic actions only grow at a polynomial rate. Much more annoying is the fact that for some integer vectors $\bar m$ that are lowest on their $\bA$ orbit, it is possible that the iterates by $\bB^l$ be decreasing in norm during a long time in both directions of $l$, before starting to increase. The challenge is to make sure that by choosing one direction of iteration for $\bB$, the double iterates $\|\bA^k\bB^lm\|$ remain always larger than $\|m\|^{\delta}$ for some $\delta>0$ independent of $m$ ($\delta$ is comparable to $1/S$ where $S$ is the step of the action). This parabolic version of the higher rank trick is done in \S \ref{double-buble}), where "being unlocked" property of the action again plays a major role. 


 }

\end{remarki}

\medskip

There has been very few local rigidity results for parabolic actions. One example are actions by left multiplication on nilmanifolds. These are parabolic, and a form of local rigidity for such $\mathbb R^2$-actions on 2-step nilmanifolds was obtained in \cite{D}, under Diophantine conditions. Results of similar type  were obtained for $\mathbb Z^2$-actions on Heisenberg nilmanifolds in \cite{DT}. More recently, in \cite{ZW}, it is proved that certain large abelian parabolic actions on homogeneous spaces of semisimple Lie groups have strong local rigidity properties.  

In the remainder of the introduction we give precise definition and precise formulation of the main rigidity result, as well as the overview of the proofs, examples and comments on possible applications. 


\subsection{Diophantine affine parabolic actions}\label{Diophantine conditions}  In this section we define the full measure Diophantine conditions required on the pair $(\a,\beta)\in \cT(A,B)$ in order to guarantee KAM-rigidity. It will be a combination of two types of conditions: simultaneously Diophantine condition for the maximal translation factor of the action, and Diophantine conditions for the translation parts of special elements of the action that we refer to as resonances.

\subsubsection{The maximal translation factor.}


\begin{definition}[Maximal translation factor] \label{def_translation} We say that the action $\(A,B\)$ has a maximal identity factor if there is a torus $\T_1$ of dimension $d_1$ such that $\(A,B\)$, restricted to this torus, equals identity. The action of $\(a,b\)$ restricted to this factor is called the maximal translation factor of $\(a,b\)$.
\end{definition}
 
\begin{definition}[Simultaneously Diophantine vectors]\label{definition sim}
We say that a pair of vectors $(\a,\beta) \in \T^{d} \times \T^{d}$ is  simultaneously Diophantine 
if there exists $\gamma, \tau>0$ such that 
$$
\max\{ |1 -e(k,\a) |, |1-e (k,\beta) |\}> \frac{\gamma}{|k|^\tau},
$$
where $e(m,x)=e^{2\pi i (m,x)}$. We denote this property by $(\a,\beta)\in \text{SDC}(\gamma,\tau)$. 
\end{definition}

  Observe that SDC-pairs of vectors  form a set of full Haar measure in   $\T^{d} \times \T^{d}$.

\subsubsection{Resonant vectors.}

In what follows we will use the dual action corresponding to the linear part $\(A, B\)$, induced on $\Z^d$. For a matrix $A$, the dual action on $\Z^d$ is denoted by  
$$
\bA:=(A^{tr})^{-1}, \quad \bA=\Id+\hA.
$$
For a general $m$ and $(k,l)\in \Z\times \Z$,  $\bar A^k\bar B^l m$ has a polynomial expression (see Lemma \ref{l_form_of_AkBl}).
However, 
if $m$ is such that there exists a pair $(k,l)\in \Z\times \Z\setminus \{(0,0)\}$ satisfying $\bar A^k\bar B^l m=m$, then, since this implies that  $\bar A^{ik}\bar B^{il} m=m$ for all $i \in \Z$, we necessarily have
(even if $A$ and $B$ are higher step):
$$
\bar A^k\bar B^l m-m=k\hA m+l\hB m=0.
$$
Hence, if $\bar A m \neq m$ or $\bar B m\neq m$, we can associate to such an $m$ a unique  pair  
$(k,l)\in \N\times \Z$ such that either $(k,l)=(1,0)$, or $(k,l)=(0,1)$, or  $k$ and $l$ are mutually prime and $k>0$. For all these cases, we use the same notation $k \wedge l=1$ and say that $m$ is resonant and that $(k,l)$ is its associated resonance pair.  Notice that, due to commutativity,  if $m$ is resonant, then any other integer vector on the $\(\bA, \bB\)$-orbit of $m$ is also resonant with {\it the same resonance pair}. So resonance pairs are attached to orbits, rather than individual vectors.   
We summarise the above discussion in the following 

\begin{definition}[Resonant vectors and resonance pairs]  Any vector $m \in \Z^d\setminus\{0\}$  such that  $\bar A^k\bar B^l m=m$ for some $(k, l)$, while either $\bar A m \neq m$ or $\bar B m\neq m$,  is called a resonant vector.  We will use the following notations:

$\cC_2(k,l)$ denotes the set of all resonant $m$ associated to the resonance pair $(k,l)$,  

$\cC_2=\cC_2(A, B)=\bigcup_{k\wedge l=1} \cC_2(k,l)$ denotes the set of all resonant vectors, 

$\res(A,B)$ denotes the set of all resonance pairs $(k,l)\in \Z^2$. 
\end{definition}

The following lemma shows that the norm of the resonant pair is bounded by the norm of any of the corresponding resonances. It is therefore bounded by the smallest one of them on the    $\(\bA, \bB\)$-orbit.

\begin{lemma}\label{l_tech_reson}
Let $a=A+\a$ and $b=B+\be$  be commuting affine parabolic maps. If $(k,l)\in \Z^2$ is the (unique) pair associated to the resonance $m$ as in Definition \ref{definition resonance}, then there exists $C=C(A,B)>0$ such that 
$$
 C (|k|+|l|) \leq |m| .
$$
\end{lemma}

\begin{proof} 
Let $m$ be a resonant vector, i.e., 
let $\bA^k\bB^lm = m$. As explained earlier, this implies that $k \hA m = - l \hB m$. Consider now the two integer vectors: $x=\hA m$ and $y=\hB m$. By assumption, $k$ and $l$ are mutually prime. 
This implies, in particular, that each component $x_j$ of the vector $x$ is divisible by $l$. Hence, $x_j\geq |l|$. Therefore, there exists a constant $C(A)$ (depending on $\hA $) such that  $|m| \geq C(A)  |l|$. 
In the same way,  $|m | \geq C(B) | k| $, which implies the statement. 
\end{proof}




\begin{definition}[Diophantine resonances] \label{definition resonance} Let $\(a, b\)$ be an affine  parabolic $\Z^2$-action. The number 
$\akl=  a^kb^l - A^k  B^l$ will be called the translation part of the element $(k, l)$ of the action.  



We say that a resonance $m\in \cC_2(k, l)$   is $(\gamma,\tau)$-{Diophantine}, if \begin{equation}\label{admissible} |1-e(m,\akl)|>\frac{\gamma}{|m|^{\tau}}.\end{equation}
\end{definition}

\begin{remark} The set of resonant vectors and resonance pairs for a given action may be empty, finite non empty, or infinite, as we will see in the examples at the end of this section. \end{remark}


\subsubsection{Diophantine property for actions.}
We are ready to define the {\it Diophantine} parabolic affine actions, for which the main local rigidity result holds.

\begin{definition}[Diophantine actions] \label{rotation vectors} 
 Given $\gamma,\tau>0$ and a   parabolic  affine $\Z^2$-action $\(a,b\)$, where $a$ is step-2, we say that $\(a,b\)$ is  $(\gamma,\tau)$-Diophantine  if: 
 \begin{enumerate}
 \item  the maximal translation factor  
 of $\(a,b\)$ is $(\gamma,\tau)$-simultaneously Diophantine (as in Definition \ref{definition sim}), and
 \item every resonance $m \in \cC_2(A,B)$ is $(\gamma,\tau)$-Diophantine (as in Definition \ref{definition resonance}).
 \end{enumerate} 
\end{definition}

\begin{example} \label{ex3} Let $\(A,B\)$ be the action on $\T^2$ generated by $A=\Id+E_{21}$ and $B=\Id$. The affine action $\(A+(\a,0),B+(0,\beta)\)$ is $(\gamma,\tau)$-Diophantine if and only if both $\a$ and $\beta$ are $(\gamma,\tau)$-Diophantine. 
\end{example}
\begin{proof} Indeed, in this case $\a_1=\a, \beta_1=0$ is the translation 
part of the affine action, and the SDC condition reduces to the Diophantine  condition on $\a$. 

In this example all the vectors $(m_1,m_2)$ with $m_2\neq 0$ are resonant with the same resonance pair $(0,1)$, and the Diophantine condition on the resonance reduces to the Diophantine condition on $\beta$. 
\end{proof}

The following simple observation is an important step in establishing the dichotomies in Theorems \ref{th.main} and \ref{almostsure}. 
\begin{proposition}\label{prop dioph} Fix $\tau>d$. Let $\(A, B\)$ be a linear  parabolic $\Z^2$-action.
 We have the following alternative: 
\begin{itemize}
\item[(i)] $\(A,B\)$ is locked as in Definition \ref{def_stiff}.
\item[(ii)] $\(A,B\)$ is unlocked, and  for almost every $(\a,\beta)\in \cT(A,B)$, $\(A+\a,B+\beta\)$ is $(\gamma,\tau)$-Diophantine for some $\gamma>0$.  
\end{itemize}
In case $A$ and $B$ are step-2, alternative $(i)$ can be reduced to the existence of a rank-one factor that is identity. 
\end{proposition}

\begin{remark} Notice that no step-2 assumption is made on any generator in the first part of the Proposition.
\end{remark} 
\begin{proof}[Proof of Proposition \ref{prop dioph}] 
Consider the maximal translation factor of $\(a,b\)$ generated by the pair of translation vectors  denoted by  $(\a^{(1)},\beta^{(1)})$. The condition  $(\a,\beta)\in \cT(A,B)$ imposes some relations over $\Z$ between the coordinates of the vectors $\a^{(1)}$ and $\beta^{(1)}$. Then we have two possible scenarios.  

The first one is that  there exists a vector $\bk\in \Z^{d_1}\setminus \{0\} $ such that $(\bk,\a^{(1)})=(\bk,\beta^{(1)})=0$ for all $(\a,\beta)\in \cT(A,B)$, in which case a change of coordinates with $X_1:=(\bk,x^{(1)})$ will exhibit a one-dimensional rank-one factor on which the action is identity. 

If the first scenario does not hold, then for all $k\in \Z^{d_1}\setminus \{0\} $ either $(k,\a^{(1)})$ or $(k,\beta^{(1)})$ is not identically zero on $ \cT(A,B)$. In this case we can split the set of integers, $\Z^{d_1}\setminus \{0\} = \mathcal Z_1+ \mathcal Z_2$, in  such a way that for $k\in \mathcal Z_1$,  $(k,\a^{(1)})$  is not identically zero on $ \cT(A,B)$, and for  $k\in \mathcal Z_2$,  $(k,\beta^{(1)})$  is not identically zero on $ \cT(A,B)$. 
Now, for any $k \in \mathcal Z_1$, for any $\de >0$, we have that 
$\lambda\left\{\a^{(1)} \in \T^{d_1} : \|(k,\a^{(1)})\|\leq \de |k|^{-d_1-1}\right\}\leq  c \de |k|^{-d_1-1}$ for a constant $c=c(d)$. 
Summing over all  $k\in \mathcal Z_1$ and then over all $k\in \mathcal Z_2$, and using Arcela-Ascoli theorem, we get that for almost every $(\a,\beta)\in \cT(A,B)$, there exists $\gamma>0$ such that for each $k\in \mathcal Z_1$, it holds that 
$\|(k,\a^{(1)})\|\geq \gamma |k|^{-d_1-1}$, and 
for each $k\in \mathcal Z_2$ it holds that $\|(k,\beta^{(1)})\|\geq \gamma |k|^{-d_1-1}$.
This implies that  for almost every 
$(\a,\beta)\in \cT(A,B)$,  the maximal translation factor  
 of $\(a,b\)$ is $(\gamma,\tau)$-simultaneously Diophantine.

Next, we consider a resonance $m \in \cC_2(A,B)$ and let $k\wedge l=1$ be its unique corresponding vector such that 
$k\hA m+l\hB m=0$. We then have two possible cases.

\medskip 

\noindent {\sc Case $1$.} There exists a resonance $m$ such that for every $(\a,\beta)\in \cT(A,B)$ it holds that $e(m,\akl)=1$. Then we prove the following.
\begin{lemma} \label{lemma.stiff} In the assumptions of Case 1, the action $\langle a, b\rangle$ has a rank-one factor that is genuinely parabolic. 
\end{lemma}
\begin{proof} 
Let $m \in \cC_2(k,l)$ be a resonance, in which case for $A':=A^kB^l$ we have $\bar{A'}m=m$, while either 
$\bA m \neq m$ or $\bB m\neq m$. For definiteness, assume that $\bB m\neq m$. We also have that for $\a':=\akl$, $e(m,\a')=1$.

After a change of variables we can assume that $m$ is one of the basis vectors, that is,  $m_{i}=0$ for $i\neq i_1$ and $m_{i_1}=0$. We also assume that both matrices $B$ and $A'$ have 1-s on the main diagonal. 

Since $\bB m\neq m$ and $\bar{A'} m=m$,  we have that  $B$ contains some $E_{i_1i_2}$ while 
$A'$ does not contain any $E_{i_1*}$ (where $*$ ranges through possible indices). 
By another change of coordinates, we can assume that $B$ does not contain any other $E_{i_1*}$ besides $E_{i_1i_2}$. (Indeed, if $B$ contains $\sum_{j=2}^s k_jE_{i_1i_j}$, we can use the coordinate change 
$x_{i_2}\mapsto \sum_{j=2}^s k_jx_{i_j}, x_k\mapsto x_k$ for $k\neq i_2$). 

By the commutativity of $a'$ and $b$ 
we get that $\a'_{i_2}=0$, 
and that $A'$ contains no $E_{i_2*}$ (otherwise $\hA'\hB$ would contain no $E_{i_1*}$ while $\hB\hA'$ would contain some).  Also, the hypothesis $e(m,\akl)=1$ translates into $\a'_{i_1}=0$. 

If $B$ has no element of the type $E_{i_2*}$, we conclude that the action $\langle a', b\rangle$ where $a'=a^kb^l$ factors on the torus 
$\TT_{i_2,i_1}$ on which 
$a'$ acts as identity, while $b$ is genuinely parabolic.

If $B$ had an element $E_{i_2i_3}$, then again after a change of coordinates, we can assume that 
 $B$ does not contain any other $E_{i_2*}$ besides $E_{i_2i_3}$.

As before, we have two consequences: 1) $\a'_{i_3}=0$, and 2) $A'$ contains no $E_{i_3*}$ (otherwise $\hA'\hB$ would contain no $E_{i_2*}$ while $\hB\hA'$ would contain some).  

If $B$ has no element of the type $E_{i_3*}$ we conclude that the torus 
$\T_{i_3,i_2,i_1}$ is a factor of the action 
$\langle a', b\rangle$ on which 
$a'=a^kb^l$ acts as identity, while $b$ is genuinely parabolic.

Arguing inductively, we obtain the proof of the lemma. \end{proof}

\medskip 

If Case $1$ does not hold, then we must be in the following case:

\noindent {\sc Case $2$.}  For every resonance $m$, there exists $(\a,\beta)\in \cT(A,B)$, such that $e(m,\akl)\neq 1$, then by linearity of $\akl$ in the variables of $(\a,\beta)$, we see that the measure of $(\a,\beta)\in \cT(A,B)$, such that 
$$
|1-e(m,\akl)|\leq \frac{\gamma}{|m|^{d+1}}
$$
is less than $c \frac{\gamma}{|m|^{d+1}}$ for a constant $c=c(d)$.
Summing up over all possible resonances and using Arcela-Ascoli theorem, we get that for almost every  $(\a,\beta)\in \cT(A,B)$, there exists $\gamma>0$ such that every resonance is $(\gamma,d+1)$ Diophantine. 


From the proof of Lemma \ref{lemma.stiff}, we see that Case 1 cannot happen if the action is step-2.

\end{proof}

\color{black}

\subsection{KAM-rigidity}

Now we are ready to formulate precisely part $(ii)$ of Theorems \ref{th.main} and \ref{almostsure}. The following is our main rigidity result. 

\begin{Main} \label{mainC} Let $\(A,B\)$ be an unlocked linear parabolic $\Z^2$-action with (at least) one step-2 generator. If $(\a,\beta)\in \cT(A,B)$ are such that  $\(a,b\)$ is $(\gamma,\tau)$-Diophantine for some $\gamma>0$ and $\tau>0$, then $\(a,b\)$ is KAM-rigid. 
\end{Main}

Theorems \ref{th.main} and \ref{almostsure} follow directly from Theorem \ref{mainC} and Proposition \ref{prop dioph}.

\begin{remark}\label{z-k} If an affine $\mathbb Z^k$-action,  $k\ge 3$, contains a Diophantine affine $\mathbb Z^2$-action with at least one step-2 generator, then Theorem \ref{mainC}  directly implies KAM-rigidity for the  $\mathbb Z^k$-action. This is because the smooth conjugacy provided by Theorem \ref{mainC} for the $\mathbb Z^2$-action would then conjugate the whole $\mathbb Z^k$-action perturbation. This simple observation is a consequence of the commutation and  the ergodicity of the Diophantine $\mathbb Z^2$-action.  This argument has been already used in  \cite{DK} (see Lemma 3.2  in \cite{DK}) to draw the conclusion about local rigidity for a $\mathbb Z^k$-action from that of its $\mathbb Z^2$-subaction. 
\end{remark}
\color{black}

\subsection{Examples of KAM-rigid actions} $ \ $ 

To begin with, let us return to the simple Example \ref{ex3}.

\begin{proposition} \label{ex33} Let $\(A,B\)$ be the action on $\T^2$ given by $A=\Id+E_{21}$ and $B=\Id$. If $\a$ and $\beta$ are Diophantine numbers, then the affine action $\(A+(\a,0),B+(0,\beta)\)$ is KAM-rigid. 
\end{proposition}

\begin{proof} As explained earlier, when $\a_1$ is Diophantine, the translation factor of the action is SDC. On the other hand, all the resonances are of the form $(m_1,m_2)=(0,m_2)$, $m_2\neq 0$, with the corresponding resonance pair $(0,1)$. Since $\a_{0,1}=(0,\beta)$, we have the following. When $\beta$ is $(\gamma,\tau)$-Diophantine, condition \eqref{admissible} holds with the constants $(\gamma,\tau)$. Hence, Theorem \ref{mainC}  implies the KAM-rigidity of $\(A+(\a,0),B+(0,\beta)\)$.
\end{proof}

It is clear that if $\a$ is Liouville, the corresponding action will not be KAM-rigid. To see this, just perturb 
$A+(\a,0)= (x_1,x_2)\mapsto (x_1 + \a, x_2+x_1)$ to $(x_1,x_2)\mapsto (x_1+\a,x_2+x_1+\eps \varphi(x))$, where $\varphi(x)$ is a smooth function with the zero mean that is not a coboundary above the rotation of angle $\a$. However, although our proof of KAM-rigidity heavily uses the Diophantine property of resonances, we are not able to settle whether $\beta$ Diophantine is a necessary condition for KAM-rigidity. 

\begin{question} \label{qq1} Is the action $\(A+(\a,0),B+(0,\beta)\)$ KAM-rigid when $\a$ is Diophantine and $\beta$ is Liouville? 
\end{question}



The following example provides a KAM-rigid action having infinitely many resonances with infinitely many resonance pairs. 

\begin{proposition} \label{ex5} Let $\(A,B\)$ be the action on $\T^7$ given by $A=\Id+E_{52}+E_{61}+E_{73}$ and $B=\Id+E_{42}+E_{43}+E_{64}+E_{73}$. For almost every $(\a,\beta)\in \cT(A,B)$, the affine action $\(A+\a,B+\beta\)$ is KAM-rigid. 
\end{proposition}

\begin{proof} The commutation condition $(\a,\beta)\in \cT(A,B)$ is satisfied if and only if $\beta_1=\alpha_4$, $\beta_2=0$, and $\beta_3=\alpha_3=-\alpha_2$. The translation factor of the action is the three-torus corresponding to the first three coordinates, and the translations are $\a^{(1)}=(\alpha_1,\alpha_2,-\alpha_2)$ and $\beta^{(1)}=(\alpha_4,0,-\alpha_2)$. It is easy to see that if the vectors $(\alpha_1,\alpha_2)$ and $(\alpha_4,-\alpha_2)$ are Diophantine, then the pair $(\a^{(1)},\beta^{(1)})$ is SDC. Indeed, denoting  
by $\|x\|$ the closest distance from $x\in \RR$ to the integers, for any $m=(m_1,m_2,m_3)$ we have:
\begin{align*}
\max(\|(m,\a^{(1)})\|,\|(m,\beta^{(1)})\|)&=\max(\|m_1\a_1+(m_2-m_3)\alpha_2\|,\|m_1\alpha_4-m_3\alpha_2\|)\\ &\geq \gamma (|m_1|+|m_2-m_3|)^{-\tau}+ \gamma (|m_1|+|m_3|)^{-\tau}\\ &\geq \gamma' |m|^{-\tau'}.
\end{align*}

Let us turn to the resonances. They are the set of $m$ such that $m_6=0$, and the two vectors, $v_m=(m_5,m_7)$ and $w_m=(m_4,m_4+m_7)$, are collinear and not both zero at the same time. Hence, at least one of $m_5$ or $m_7$ does not vanish.  The resonance pairs are the pairs 
$k_m\wedge l_m=1$ such that $k_mv_m+l_mw_m=0$. For example, 
$$
m=(m_1,m_2,m_3,n+1,n,0,n(n+1)), \quad (m_1,m_2,m_3,n)\in \Z^3\times \Z^*
$$
is a resonant vector with the resonance pair $(n+1,-n)$. Finally, fix any resonant $m$ and observe that, since there are no constraints on $\a_5,\a_7$ and since at least one of $m_5$ or $m_7$ does not vanish,  the Diophantine condition on $\a_{k_m,l_m}$ is satisfied for  almost every $(\a,\beta)\in \cT(A,B)$. 
\end{proof}

\subsection{Overview of the proof of Theorem \ref{mainC}}  
\label{s_overview}
The proof is based on an inductive scheme of successive conjugations of the perturbed action $\(F, G\)$, where $F=a+f$ and $G=b+g$ to the affine action  $\(a, b\)$.  
As usually in the KAM approach, the linearized conjugacy equations are solved at each step of the induction 
with a loss of derivatives, which can be caused, for example, by small divisors or by other reasons. The {\it a priori} 
damaging effect of this loss is  tamed out by the quadratic speed of convergence of the scheme. 

In our context, the linearized conjugacy equations, often called the cohomological equations, are essentially of the following form:   
  \begin{equation}\label{linearisation22}
\begin{aligned}
h\circ a - Ah &=f,\\
h\circ b - Bh &=g.
\end{aligned}
\end{equation}

Two main differences with the classical KAM schemes that appear in our context are the following:

\begin{itemize} \item[$(i)$] Cohomological equations \eqref{linearisation22} above each individual generator of the action are, in general,  not solvable because of the existence of an infinite countable set of obstructions. These were first evidenced in the step-2 example in the  work of Katok and Robinson \cite{KatokRobinson}. (Large set of distributional obstructions was likewise found for any step nilflows in \cite{FF}. Also see \cite{CF} for a study of certain cases of cohomological equations above abelian actions). 

\item[$(ii)$] In the case of parabolic affine maps of step 3 and higher, and with Diophantine translation part, if the solution to 
one of the equations of \eqref{linearisation22} exists, then it is smooth if the right-hand side is smooth. However, the loss of the number of derivatives is not fixed (in other words, the linearized cohomological equation is stable, but the solutions are not tame).

In a separate work we show that, for the simplest  $C^r$ step-3 map 
$(x_1,x_2,x_3)\mapsto (x_1+\a_1,x_2+x_1,x_3+x_2)$, the loss of derivatives is roughly $r/2$, 
even for the nicest Diophantine angles $\a_1$. This constitutes a notable  difference with the step-2 case, for which \cite{KatokRobinson} showed tameness of the solutions when they did exist.  
\end{itemize} 

To address $(i)$, the usual path is to exploit the  commutation relation to find approximate solutions to the cohomological equations. This was done in two related problems in the past. First, by Moser \cite{M}, who showed that SDC commuting circle rotations are locally rigid under the condition of preserving the rotation number. This was extended to higher dimension in \cite{DF,WX,P}. Second, by Damjanovi\'c and Katok who proved in  \cite{DK} the local rigidity of higher rank partially hyperbolic affine abelian actions on the torus (i.e., actions, all of whose elements are ergodic automorphisms or affine maps with such linear parts). 

In Moser's case the objective is to linearize a commuting pair $R_{\a_i}+f_i$, $i=1,2$. The cohomological equations take the form $h(x+\a_i)-h(x)=f_i(x)-\int f_i$. They have formal solutions above the generators $R_{\a_i}$, and the commutation relation allows to upgrade the formal solutions into the approximate tame solutions. 
In fact, Moser's trick is to define, for each Fourier mode $n$, the corresponding coefficient $h_n$ of the conjugacy map, 
using either one or the other of the linearized conjugacy equations, according to which $\|n\a_i\|$, for $i=1$ or for $i=2$, is "not too small" as granted by the  SDC-condition. As a result, one gets a candidate conjugacy $h$ that is {\it tame}, i.e., of the same order as the nonlinearities $f$ and $g$ with a fixed loss of the number of derivatives. Moreover, the commutation relation plus the SDC-condition insure that the constructed $h$ solves the cohomological equations with a {\it quadratic} error (with a small abuse of notations, by {\it quadratic} we will mean  that the error  is  of order of a power $k>1$ in the nonlinearities $f$ and $g$, with a fixed loss of the number of derivatives).

The above procedure allows to implement the classical KAM quadratic scheme, 
with the issue of the constant terms $\int f_i$ being resolved due to the condition of the preservation of the rotation numbers. 

In \cite{DK} the individual equations as in \eqref{linearisation22} have a tame solution provided a countable set of obstructions vanish, each one being formally computed as weighted sums \color{black} along the dual orbit of Fourier coefficients of the nonlinearities $f$ and $g$. 
The commutation relation in this case allows to get quadratic approximations of the nonlinearities by functions whose obstructions vanish. This was labelled "highr-rank trick". Here again, the approximation is quadratic with a finite loss of the number of derivatives.

\medskip 

Our proof of KAM-rigidity for parabolic actions combines  two mechanisms of local rigidity: "Moser's trick" and "higher rank trick". 
The translation part of the action and the resonances (Fourier modes that are invariant along some element of the dual action) are treated using a mechanism, similar to Moser's trick. This is where  the Diophantine conditions of Definition \ref{rotation vectors} on the action play a crucial role. 
 It has to be noted that the use of Moser's trick for the resonances brings some technical challenges that affect the whole proof. Indeed, for a resonant Fourier mode $m$, we need to use the element $F^kG^l$ of the action, where $(k,l)$ is the resonance pair associated to $m$. This forces us to work out the linearization KAM scheme  at each step for {\it a large number of elements} of the action, and not only for the two generators. Of course, we cannot control {\it all} of the nonlinearities in $F^kG^l$ for all resonance pairs $(k,l)$ at each step, because $k$ and $l$ can be arbitrarily large. Fortunately, the resonance pairs associated to a  
resonant mode $m$ are of the order of $m$ (see Lemma \ref{l_tech_reson}). This means that if, at a given step of the KAM scheme, we truncate the nonlinearities up to order $N$ before finding an approximative solution of the  linearized conjugacy equation, we will only need to control $F^kG^l$ for $k$ and $l$ of order $N$. This can easily be included in the induction due to the parabolic nature
of $a$ and $b$. 
\color{black} 

For "non-resonant" Fourier modes, it is a higher rank trick approach similar to \cite{DK} that is invoked. 
Indeed, for a non-resonant mode $m$ we can define $h_m$ {\it via} the sum of the Fourier coefficients of the nonlinearity $f$ along the dual orbit of the step-2 generator of the affine action, taken in the "good direction": either in the future or in the past (in a similar way to what is done in Livschits theory). The fact that the generator is step-2 implies that the Fourier modes, involved in these partial sums, grow either for the past or the future sum, which allows us to define a tame candidate conjugacy $h$, as observed in \cite{KatokRobinson}. {Observe that difficulty $(ii)$ mentioned above shows that the mere definition of a candidate tame conjugacy when no element of the action is step-2 is already a challenge for the general higher step case. Other difficulties appear in relation with the applicability of the parabolic higher rank trick that will be explained in the next paragraph. }

Once $h$ is constructed, we see that it is only at special modes $\bar m$ that are lowest (in norm) on their dual orbit along $\bA$ that the constructed $h$ does not solve the cohomological equation above $a$ (at $\bar m$, the good direction switches from past to future). The error in solving the equation at $\bar m$ is indeed the full sum along $\bar m$ 
of the Fourier coefficients of the nonlinearity along $\bA$. These sums, having the form $\Si_m^{A} (f)=
\sum_{k\in \Z} f_{\bar A^{k} m}\la^{(k)}_{ m}$ (where $\lambda^{(k)}_m$ are "innocuous" multipliers of 
modulus one related to the translation part of the action, see \S \ref{s_notations} for the exact definitions), 
are the obstructions to solving the cohomological equations above $\bA$. 

 The higher rank trick uses commutativity to show that this full sum is equal to a double sum of a quadratic function $\phi$ measuring the error of the pair $(f,g)$ in \eqref{linearisation22} from forming a cocycle above the action $\(a,b\)$ \color{black}
  (see  Section  \ref{s_linearisation} and  Section  \ref{s_case3} for more explanations). It is appears to be fruitful to express the obstructions as the following double sums:
 \begin{equation} \label{eqdouble}
\Sigma_m^A (f) =
\sum_{l\geq 0}  \sum_{k\in \Z} \phi_{\bar A^{k}\bar B^l m}\la^{(k)}_{m}\mu^{(l)}_{ m} = 
-\sum_{l\leq -1}  \sum_{k\in \Z}  \phi_{\bar A^{k}\bar B^l m}\la^{(k)}_{ m}\mu^{(l)}_{ m} .
\end{equation}

 There is an important difference between the phenomenon that lies behind the control of the double sums in our case, compared to the partially hyperbolic case.  
In the {\it partially hyperbolic} higher rank case treated in \cite{DK}, the Fourier modes that appear in the double sums in one of the two directions (future or past for $\bB$) are essentially increasing due to the partial hyperbolicity of the action, and this immediately leads to approximate solutions of \eqref{linearisation22} with quadratic errors with finite loss of the number of derivatives.

In our case, due to the presence of a higher step generator in the action, there may be no growth in either direction along the dual orbits that appear in the double sums. In fact, it always happens for some modes $m$ that the double orbits 
appearing in \eqref{eqdouble} decay in both directions from $|m|$ to $|m|^{1/(S-1)}$, where $S$ is the step of the action. This is the difficulty $(ii)$ mentioned above.

One of the  key ingredients of our argument is the proof of the fact that
for a {\it unlocked} parabolic linear action with at least one step-2 generator, the fall from $|m|$ to $|m|^{1/(S-1)}$ is the worst that can happen. {Our proof uses the presence of a step-2 element, and its extension to higher step actions is another challenge in the study of the general case.}

This means that the error in solving the first equation in \eqref{linearisation22} with the conjugating transformation $h$ we constructed is quadratic, but with a loss of a certain {\it proportion} of the number of the derivatives that are considered (a proportion $(S-2)/(S-1)$ for step-$S$ maps), even under the nicest Diophantine conditions. 

The good news is that this loss of derivatives appears only in the quadratic error and not in the estimate of the conjugating map 
(for this, the step-2 assumption on one generator is crucial). As a consequence, this important loss of derivatives does not affect the convergence of the KAM scheme, for which it suffices to have a quadratic control of $C^0$ norms of the error (in fact $L^2$ would be sufficient).

Once it is shown that $h$ solves the first equation of \eqref{linearisation22} up to a quadratic error (in $C^0$ norm), the commutation relation and the fact that $A$ is step-2 can be used again to show that $h$ also solves the second equation 
with a quadratic error (see  Section  \ref{Fourier}).

Finally, we point out to the fact that equations \eqref{linearisation22} can be solved as usual up to a set of $2d$ constant terms that account for the averages of $f$ and $g$. Unlike in Moser's case of commuting circle diffeomorphisms, these  constants  are not all related to some dynamical invariants. However, we can use the volume preservation of the perturbed action and the zero average of the nonlinearities to fix the averages of the conjugating diffeomorphisms at each step of the KAM scheme, so that the constant terms become absorbed in the quadratic error. This third difference with the usual KAM scheme is explained in detail at the end of \S \ref{iteration set-up}.  

As remarked before, our arguments remain true if we replace the preservation of the volume $\lambda$
by that of any common invariant measure for $F$ and $G$. 
It suffices to  replace $\lambda$  by an arbitrary common invariant measure in all the text. 
Indeed, we do not use that $\lambda$ is invariant by $a$ and $b$ in the proof of the linearization. 
Moreover, since $a^kb^l$ is uniquely ergodic for some $k$ and $l$, 
the linearisation implies, in fact, that there is a unique invariant measure for the action $\(F, G\)$, 
and that this measure is the pullback of the Haar measure by the conjugacy.


\subsection{Comments on extensions and applications}  

There are natural questions raised by our result as to what extent the method developed here is applicable to more general situations. We comment on this below.


\medskip




$\diamond$ {\bf On applications to non-abelian actions.} We note that there are classes of solvable affine actions to which our result in Theorem \ref{mainC} can be directly applied.
An {\it abelian-by-cyclic} group $G$ is a finitely presented torsion free group  admitting a short exact sequence $0\to \mathbb Z^k\to G \to \mathbb Z\to 0$ (see \cite{WX} for detailed discussion on ABC groups). In this context, we call the subgroup $ \mathbb Z^k$ the abelian part of $G$.  Let $\rho: G\to \rm Aff (\mathbb T^d)$ be an affine action of $G$ such that $\rho( \mathbb Z^k)$ is parabolic. Then from the KAM rigidity result for the action  $\rho( \mathbb Z^k)$ one may derive KAM rigidity for the $G$ action. This way of obtaining KAM rigidity for an ABC action from KAM rigidity of its abelian part has been used before in \cite{WX} but in the special case where the abelian part  $\rho( \mathbb Z^k)$ is generated by translations. More recently, in \cite{P2}, actions on $\mathbb T^3$ of the following particular ABC group: $\Gamma=\(U, V, F: UV=VU, FU=U^2VF, FV=UVF\)$ have been studied. An example of a $\Gamma$ action on $\mathbb T^3$ is when $U=\id +E_{12}$ and $V=\id +E_{13}$ and  $$F= \begin{pmatrix}
1&0&0\\
0& 1&-1\\
0&-1&2
\end{pmatrix}. $$ 
The abelian action $\(U, V\)$ is unlocked and thus Diophantine affine actions with such linear part are KAM rigid by Theorem \ref{mainC}, which in turn implies KAM rigidity for an affine $\Gamma$ action with such abelian part.


Similar to our dichotomy result, we expect to use the method we developed in this paper to obtain a classification result for linear ABC actions $\rho_L: G\to \rm Aut (\mathbb T^d)$ having a parabolic abelian part.  There are roughly 3 main cases: 

 (i) $\rho_L$ is locked: for every affine $G$ action $\rho$ with linear part $\rho_L$, $\rho$ has a rank-one factor that is either identity or a genuinely parabolic action.
 
 (ii)  $\rho_L$  is unlocked but {\it has a locked abelian part}: for every affine $G$ action $\rho$ with linear part $\rho_L$, $\rho (\mathbb Z^k)$ has a rank-one factor that is either identity or a genuinely parabolic action, but $\rho_L$ is unlocked.  
 
 (iii) $\rho_L$ {\it has an unlocked abelian part}. 
 
 It is in the case (iii) where Theorem \ref{mainC} applies. In the case (ii),  even though Theorem \ref{mainC} does not apply directly, we expect our method and even the constructions of solutions from our proofs, to apply.




It is a curious algebraic question to determine which solvable groups acting on the torus by automorphisms can have unlocked abelian part.  The group $\Gamma$  described above allows on $\mathbb T^3$ both a locked action (case (i)) and an action with unlocked abelian part (case (iii)), but does not allow case (ii)  \cite{P2}. We remark that the 3 dimensional discrete Heisenberg group $H_3$ generated by three elementary matrices (i.e. matrices of the form $\id+E_{ij}$) on any $\mathbb T^d$ is locked (in particular it has a locked abelian part) and it is not clear if $H_3$ linear actions by toral automorphisms are always locked.\color{black}

\medskip  

$\diamond$  {\bf On connection to nilflows.} Given a nilpotent Lie group $N$ of step $k$, and a lattice $\Gamma$ in $N$, the quotient $N/\Gamma$ is a nilmanifold of step $k$. Any one-parameter subgroup of $N$ defines, via left-multiplication on $N/\Gamma$, a smooth nilflow. Similarily, a subgroup $A$ of $N$ isomorphic to $\mathbb R^k$ defines an $\mathbb R^k$ {\it nilaction} on  $N/\Gamma$.  While it was proved by Flaminio and Forni \cite{FF} that nilflows have infinite dimensional cohomology, nilactions can have finite dimensional cohomology as in \cite{CF} or in \cite{D}. In \cite{D} this was used for proving a KAM type of local rigidity result for a class of nilactions with strong Diophantine properties, on 2-step nilmanifolds. There is a close connection between the actions which we consider in this paper and nilactions. Namely, one gets a parabolic affine $\mathbb Z^k$ action if one considers return maps of an $\mathbb R^k$ nilaction to a certain section, and $\mathbb R^k$ nilactions can be viewed as suspensions over such $\mathbb Z^k$ actions on the torus.   We hope that some of the ideas developed here to study the KAM-rigidity of parabolic actions on the torus could be useful in the local rigidity study of nilactions, that is a more general problem where, besides the works cited above,  there has been yet no progress.\color{black}


\medskip

\color{black}

\subsection{Plan of the paper} The rest of the paper is devoted to the proof of Theorem  \ref{mainC}. The proof is divided into three parts. In \S \ref{sec.proof.step2} we state the main inductive KAM conjugacy step, Proposition \ref{Main iteration step}, and then show how to deduce Theorem  \ref{mainC} from it. In \S \ref{sec3} we give some necessary estimates on 
sums and double sums along the dual orbits of $A$ and $B$ that serve for constructing the approximate solutions to the cohomological equations that appear in the linearized conjugacy equations. In \S \ref{plan_of_proof} we use the latter estimates to prove Proposition \ref{Main iteration step}.
Each part will start with a detailed introduction of its content and of the ideas that are involved in the proofs. 


\section{Proof of Theorem \ref{mainC}- the iteration part} \label{sec.proof.step2}

The proof is based on a KAM scheme, with three peculiarities which distinguish it from the usual way KAM schemes are applied to proving local rigidity.

The usual KAM iteration goes as follows: we start with an $\eps$-perturbation $\(F, G\)$ of $\(a,b\)$. By linearising the conjugacy problem and by solving the linear equation approximately, we produce a conjugacy $\mathcal H_1= (\Id+\bfh_1)$ which conjugates $\(F, G\)$ to an action $\(F_1, G_1\)$ which is an $\ve^k$-perturbation of $\(a, b\)$, where $k>1$. Then we say that  $\(F_1, G_1\)$ is a {\it quadratically small} perturbation of $\(a, b\)$, with respect to how far $\(F, G\)$ was from $\(a, b\)$.  This process is repeated, and at the $n$-th step  of iteration we build conjugacies $\mathcal H_{n}=(\Id+\bfh_1)\circ \ldots \circ (\Id+\bfh_n)$ that satisfy 
  \begin{equation}\label{hh1'}
 \begin{cases} 
\mathcal H_{n}^{-1} \circ F \circ  \mathcal H_{n}&=a+\bff_{n+1},\\
\mathcal H_{n}^{-1} \circ {G}\circ  \mathcal H_{n}^{-1}&=b+\bfg_{n+1}, 
\end{cases}
\end{equation}
where $\bff_{n+1}$ and $\bfg_{n+1}$ are of order $\eps_{n}^k= \eps_{n+1}$, while $\bfh_n$ is of order $\eps_{n}$.

Truncation (or more generally, applying smoothing operators) is typically used only to remedy a fixed loss of regularity at each step of iteration while solving the linearized problem.  In our case here, due to the (possible) presence of infinitely many resonances, without truncation we might not have any quadratic estimates for the error. 
This is the first peculiarity of the proof,  the corresponding details are contained in \S \ref{Fourier}.

The other one is that at every step of the KAM procedure we solve  the linearised equations  approximately {\it only} up to a constant term. This constant term can be large, it makes the error at the $n$-th step  of order $\eps_{n}$ instead of $\eps_{n+1}$ (as we would like), so {\it a priori} there need not be any convergence of the sequence $\mathcal H_n$. This is where we use the volume preservation assumption. Namely, the volume preservation assumption allows us  to adjust the average of  ${\bfh}_n$ at step $n$, so that  the total new error $(\bff_{n+1}, \bfg_{n+1})$ becomes of order $\eps_{n+1}$. The same approach was used by Herman for Diophantine torus translations \cite{Herman}. Application of this approach in the context of group actions meets certain difficulties. 
This is explained in \S \ref{iteration set-up}. 

The third feature of the proof, which has not appeared much in similar problems, is that, even though the estimates for ${\bfh}_n$ at each step are tame, the estimates for the error at each step are not tame. Namely, the loss of the number of derivatives is not a fixed constant as usually, but a proportion (that goes to $1$ when the second generator's step goes to infinity) of the number of derivatives. However, this does not affect the convergence of the scheme. Similar observation was used recently in \cite{ZW}. 

\subsection{Linearisation of the problem and the main iterative step: Proposition \ref{Main iteration step}}\label{s_linearisation}

Given small perturbations $a+\bff$ and $b+\bfg$ of the two action generators $a$ and $b$, and the commutativity condition among them:   
\begin{equation}\label{commuting}
(a+\bff)\circ(b+\bfg)=(b+\bfg)\circ(a+\bff), 
\end{equation}
we wish to solve for $H=id+\bfh$ the conjugacy problem 
\begin{equation}\label{conjugacy}
H\circ (a+\bff) = a\circ H, \quad H\circ (b+\bfg) = b\circ H.
\end{equation}
The commutativity condition \eqref{commuting} can be
rewritten as
$$
\bff(b+\bfg) -B\bff - (\bfg(a+\bff)- A \bfg)=0,
$$
which permits to see condition \eqref{commuting}  as a sum of a linear operator applied to $\bff, \bfg$ plus a non-linear part which is quadratic in $\bff, \bfg$:
\begin{equation}\label{linearisation}
[\bff\circ b -B\bff - (\bfg\circ a- A \bfg)] + [\bff(b+\bfg)-\bff\circ b-(\bfg(a+\bff)- \bfg \circ a)]=0.
\end{equation}
Now we introduce some notations.
For any given $\bfh$, let 
\begin{equation*}\label{definition of D's}
\begin{aligned}
D_{1,0}\bfh&:=\bfh\circ a-A\bfh,\\
D_{0,1}\bfh&:=\bfh\circ b-B\bfh.
\end{aligned}
\end{equation*}

With this notations, equation \eqref{linearisation} gets the form 
\begin{equation*}\label{L} 
D_{0,1}\bff-D_{1,0}\bfg= -\bff(b+\bfg)+\bff\circ b+\bfg(a+\bff)- \bfg \circ a.
\end{equation*}

Similarily, since we are looking for the conjugating map $H$ in a neighborhood of the identity, i.e., in the form $H=\id+\bfh$ with $\bfh$ small,  our conjugacy problem \eqref{conjugacy} is linearised as
\begin{equation*}\label{linearisation2}
\begin{aligned}
D_{1,0}\bfh&=\bff+[\bfh(a+\bff)-\bfh\circ a] ,\\
D_{0,1}\bfh&=\bfg+ [\bfh(b+\bfg)-\bfh\circ b].
\end{aligned}
\end{equation*}


At the $n$-th step  of the iteration process, for given $\bff_n, \bfg_n$ we show that we can find $\tilde \bff_n, \tilde \bfg_n$, $\bfh_n$ and vectors ${\bf V}_n$ and  ${\bf W}_n$ such that  
 \begin{equation*}\label{approximate linearisation}
\begin{aligned}
D_{1,0}\bfh_n&=\bff_n + \tilde \bff_n + \bf V_n ,\\
D_{0,1}\bfh_n&=\bfg_n +\tilde \bfg_n+ \bf W_n,
\end{aligned}
\end{equation*}
where $\bf V_n$ and $\bf W_n$ are of the same order as $\bff_n, \bfg_n$,  and the new functions $\tilde \bff_n$, $\tilde \bfg_n$ are  quadratic. In later sections, when we set up the iteration process, we will see that the volume preservation assumption will force the constant terms $\bf V_n$ and $\bf W_n$ to be of quadratic order as well. 

Let us formulate the main iterative step as a proposition. In later sections Proposition \ref{Main iteration step} will be used to  perform iterations,  show their convergence and prove the main result Theorem \ref{mainC}. 

In what follows, we say that an affine action $\(a, b\)$ is {\it unlocked} if its linear part $\(A, B\)$ is unlocked. If $\(a, b\)$ is a Diophantine affine action, then its linear part is automatically unlocked, 
but we stress this in the statements since the property of $\(A, B\)$ being unlocked will play a crucial role in the proofs. 
\color{black}

\begin{PProp}\label{Main iteration step}
Let $\(a, b\)$ be an unlocked $(\ga,\tau)$-Diophantine parabolic affine $\mathbb Z^2$ action, where $a$ is step-2.  
Let $F=a+\bff $ and $G=b+\bfg $ be $C^\infty$ commuting diffeomorphisms generating a perturbation $\(F, G\)$ of $\(a, b\)$.  For $r\ge 0$, let $\Delta_r=\max\{\|\bff\|_r, \|\bfg\|_r\}$. 

There exist  constants $C$, $C_r$, $C_{r'}$  and $D=D(a,b, \gamma,\tau, d)$ such that   for any $N\in \mathbb N$ there exist vector fields ${\tilde \bff}_N$, ${\tilde \bfg}_N$, ${\bfh}_N$,  and vectors $\bf V$ and  and $\bf W$  
  such that
\begin{align*}
D_{1,0}{\bfh}_N+{\bf \tf}_N&= \bff+ \bfV,\\
D_{0,1}{\bfh}_N+{\bf \tg}_N&= \bfg+\bfW,\\
\end{align*}
 and the following estimates hold whenever $0\le r$, $D< r'$:
 \begin{equation}\label{main-est}
 \begin{aligned}
\|{\bf h}_N\|_{r}&\leq  C_r  \, N^{D}\Delta_{r},\\
\|{\bf \tf}_N\|_{0}, \,  \|{\bf \tg}_N\|_{0}
&\le CN^{D} \Delta_0\Delta_1+ C_{r'}N^{-r'+D}\Delta_{r'},\\
\|{\bf \tf}_N\|_{r}, \,  \|{\bf \tg}_N\|_{r}
&\le C_r N^{D}\Delta_{r}, \\
|{\bf V}|, |{\bf W}| &\leq  C \Delta_{0}.\\
\end{aligned}
\end{equation}
\end{PProp}

The proof of Proposition \ref{Main iteration step} is postponed to \S \ref{sec3}.

\subsection{Iteration set-up}\label{iteration set-up}

In this section we set up the iteration which we use to prove Theorem \ref{mainC}. The iterative step consists of three sub-steps:  linearization,  application of Proposition \ref{Main iteration step} and adjusting the average of the conjugating diffeomorphism by using the volume preservation of the perturbation.

\begin{proposition}\label{iterative_step}
Let $\(a, b\)$ be an unlocked $(\gamma,\tau)$-Diophantine  parabolic affine $\mathbb Z^2$ action, where $a$ is step-2.  
There exists a constant  $D>0$  only depending on the action $\(a, b\)$, for which the following holds.

Let $\(a+\bff ,b+\bfg \)$ be a $C^\infty$  volume preserving perturbation such that 
$$ave(\bff)=ave(\bfg)=0.
$$ 

Assume that we have constructed a conjugation up to the $n$-th step, $\mathcal H_{n-1}= (\id + \bfh_{n-1}) \circ\dots\circ (\id + \bfh_1)$, such that $ave((\mathcal H_{n-1})-\id)=0$ and
$$
\mathcal H_{n-1}\circ \(a+\bff, b+\bfg \) \circ \mathcal H_{n-1}^{-1}=\(a+\bff_n, b+\bfg_n\).
$$
Denote $\Delta_{r, n}:= \max\{\|\bff_n\|_r, \|\bfg_n\|_r\}$.

Then for any $N\in \mathbb N$ there exists $\bfh_n$  (which depends on $N$) such that for $0\le r$, $D\le r'$  and certain constants $C$, $C_r$, $C_{r'}$,  if $N^D\Delta_{1,n}<1$, then we have:   

\begin{enumerate}

\item[(i)] $\|\bfh_{n}\|_r \le C_rN^D \Delta_{r,n}$;

\item[(ii)] For  
$$
\bff_{n+1}:= (\id+\bfh_{n})\circ (a+\bff_n)\circ (\id+\bfh_{n})^{-1}-a,
$$ 
$$\bfg_{n+1}:= (\id+\bfh_{n})\circ (b+\bfg_n)\circ (\id+\bfh_{n})^{-1}-b,
$$  
$$\Delta_{r,n+1}:= \max\{\|\bff_{n+1}\|_r, \|\bfg_{n+1}\|_r\},$$
the following estimates hold:
\begin{equation}\label{new_error}
\begin{aligned}
\Delta_{0,n+1}&\le C_rN^D  \Delta_{0, n}\Delta_{1, n}+ C_{r'} N^{-r'+D}\Delta_{r', n}, \\
\Delta_{r,n+1}&\le C_{r}N^D \Delta_{r, n}; \\
\end{aligned}
\end{equation}

\item[(iii)] For $\mathcal H_{n}:=  (\id+ \bfh_{n}) \circ\mathcal H_{n-1}$ we have: $ave(\mathcal H_{n}-\id)=0$. 

\end{enumerate}
\end{proposition}
\begin{proof}
The non-linear problem is to find $\bfh_{n}$ such that 
\begin{equation}\label{conj}  (\id+\bfh_{n})\circ\(a+\bff_n, b+\bfg_n\)=   \(a+\bff_{n+1}, b+\bfg_{n+1}\)\circ(\id+\bfh_{n})\end{equation}
with $\bfh_{n}$ and $\bff_{n+1}, \bfg_{n+1}$ satisfying the estimates of the proposition. 

Here is a brief outline of the proof. After linearizing the above non-linear problem, 
we will first apply Proposition \ref{Main iteration step} to determine $\bfh_n$ and vectors $\bfV_{n}$, $\bfW_{n}$ such that $|\bfV_n|+|\bfW_n| \leq C\Delta_{0,n}$ and
\begin{equation}\label{nonlinear}
(\id+\bfh_{n}) \circ \(a+\bff_n, b+\bfg_n\) =   \(a+\bff_{n+1}+\bfV_{n}, b+\bfg_{n+1}+\bfW_{n}\) \circ (\id+\bfh_{n}),
\end{equation}
where  $\bfh_n$ and $\bff_{n+1}, \bfg_{n+1}$ satisfy the estimates in $(i)$ and $(ii)$. We observe that if we change $\bfh_n$ by adding to it a translation vector of order $\Delta_{0,n}$, then equation \eqref{nonlinear} will still hold with some new $\bff_{n+1}, \bfg_{n+1}$,  $\bfV_{n},\bfW_{n}$ that satisfy the same estimates. By adequately choosing the translation vector, based on the volume preservation condition and the zero average condition on the initial perturbation (and the inductive condition $ave((\mathcal H_{n-1})-\id)=0$), we will be able to absorb the constants $\bfV_{n}$ and $\bfW_{n}$ into $\bff_{n+1}$ and $\bfg_{n+1}$.

\medskip

\noindent{\bf Linearization. \ } We begin by linearizing the non-linear conjugation problem. Equation \eqref{nonlinear} is rewritten in a way that expresses the error in the new perturbation $(\bff_{n+1}, \bfg_{n+1})$ in terms of the linearization of the non-linear conjugation problem above and additional errors:
\begin{equation}\label{linearization}
\begin{aligned}
\bff_{n+1}(\id+\bfh_{n})&=  (\bfh_n\circ a- A\bfh_n) +\bff_n -\bfV_{n} +  ( \bfh_n\circ(a+\bff_{n})-\bfh_n\circ a),\\
 \bfg_{n+1}(\id+\bfh_{n})&= (\bfh_n\circ b-B\bfh_n)+ \bfg_n-\bfW_{n} + ( \bfh_n\circ(b+\bfg_{n})-\bfh_n\circ b). \\
\end{aligned}
\end{equation}

To estimate the left-hand side in the equations above,  we need to estimate the following two terms: 
\begin{equation}\label{E1E2}
\begin{aligned}
E_1:=& \bfh_n\circ a-A\bfh_n+\bff_n-\bfV_{n},\, \bfh_n\circ b-B\bfh_n+\bfg_n-\bfW_{n},\\
E_2:=& \bfh_n\circ(a+\bff_n)-\bfh_n\circ a,  \, \bfh_n\circ(b+\bfg_n)-\bfh_n\circ b.\\
\end{aligned}
\end{equation}
 Bellow we estimate both terms, $E_1$ and $E_2$, in $C^0$ norm for the transformation  $\bfh_n$ provided by in Proposition \ref{Main iteration step}. This will imply the estimate for the $C^0$ norm of $\bff_{n+1}$ and  $\bfg_{n+1}$.
\medskip

\medskip

\noindent{\bf Applying Proposition \ref{Main iteration step}.}

Apply now Proposition \ref{Main iteration step} to $\bff_n, \bfg_n$.  Fix $N\in \mathbb N$. For  the fixed  $N$, from Proposition \ref{Main iteration step} we obtain  ${\bfh_n}:=({\bfh_n})_N$, $\widetilde{(\bff_n)}_N$, $\widetilde{(\bfg_n)}_N$  and vectors $\bfV_{n}$ and $\bfW_{n}$.  The first estimate in Proposition \ref{Main iteration step} gives directly:
\begin{equation}\label{hhhh}
\|\bfh_{n}\|_r \le C_rN^D \Delta_{r,n}.
\end{equation}

Observe that the assumption that $N^D \Delta_{1,n}$ is bounded by a constant implies that $\|\bfh_{n}\|_1$ is bounded by a constant. It is a common fact (see for example  \cite[Lemma AII.26]{L}) that the inverse map  $(\Id+\bfh_{n})^{-1}=\Id+ \bfh'_{n}$ is such that  $\bfh'_{n}$ also satisfies the estimate:
\begin{equation}\label{h'}
\|\bfh'_{n}\|_r\le C_r \|\bfh_{n}\|_r.
\end{equation}

The error $E_1$ is precisely $(\widetilde{(\bff_n)}_N,\widetilde{(\bfg_n)}_N)$, so from the second estimate in Proposition \ref{Main iteration step} we get for any $r'>0$:

\begin{equation}\label{E1}
\begin{aligned}
\|E_1\|_0&\le   CN^{D} \Delta_{0,n}\Delta_{1,n}+ C_{r'}N^{-r'+D}\Delta_{r',n} .\\
\end{aligned}
\end{equation}

The estimate for $E_2$ follows by using the standard estimates  (see for example Appendix in \cite{DF}) and estimate \eqref{hhhh}:
\begin{equation}\label{E2}
\|E_2\|_0\le C\|\bfh_n\|_{1}\Delta_{0, n}\\
\le  CN^D\Delta_{1,n}\Delta_{0, n}.\\
\end{equation}

Putting the  two errors together, we have that the new error satisfies:

\begin{equation}\label{E}
\Delta_{0,n+1}\le   CN^{D} \Delta_{0,n}\Delta_{1,n}+ C_{r'}N^{-r'+D}\Delta_{r',n}.
\end{equation}

The estimate for $C^r$ norms of $\bff_{n+1}$ and $\bfg_{n+1}$ for any $r$ (the second estimate in \eqref{new_error}) follows from the definition \eqref{linearization} of these maps. We show how the estimate follows for $\bff_{n+1}$. For $\bfg_{n+1}$ the proof is the same.

From \eqref{linearization} we can write:
$$
\bff_{n+1}= \widetilde{\bff_{n}}(\id+\bfh'_{n}) +  ( \bfh_n\circ(a+\bff_{n})-\bfh_n\circ a)\circ (\id+\bfh'_{n}),
$$
where $\Id+\bfh'_{n}=(\id+\bfh_{n})^{-1}$, and $\bfh'_{n}$ satisfies estimate \eqref{h'}. 
Then by applying standard estimate for the composition of maps (see for example \cite[Theorem A.8]{Hormander}) 
and the bound for $\|\widetilde{\bff_{n}}\|_r$ which we have from Proposition \ref{Main iteration step},  we get: 
$$
\begin{aligned}
\|\bff_{n+1}\|_r \le& \|\widetilde{\bff_{n}}(\id+\bfh'_{n})\|_r +  \| \bfh_n\circ(a+\bff_{n})-\bfh_n\circ a)\circ (\id+\bfh'_{n})\|_r\\
\le &C_r (\|\widetilde{\bff_{n}}\|_r+ \|\bfh'_{n}\|_r)\le C_rN^D\Delta_{r, n.}
\end{aligned}
$$

\medskip

\noindent{\bf Adjusting the average of the conjugating diffeomorphism}. Now we will adjust the average of $\bfh_n$ in such a way that the constant terms $\bfV_{n}$ and $\bfW_{n}$ in \eqref{linearization} are forced to be as small as $\|\bff_{n+1}\|_0$ and $\|\bfg_{n+1}\|_0$, respectively. The crucial role here is played by the assumptions on the volume preservation and zero averages of the initial errors. The adjustment of the average of $\bfh_n$ will not depend on the action elements, as will be seen in Lemma \ref{c}.  We will check that it works on one action generator, the other generator can be treated in the same way.

\begin{lemma}\label{Cchange}
Suppose that $ \bfh$ is such that 
\begin{equation} \label{inv}
\begin{aligned}
a+ \bff_{n+1}+\bfV_{n} &= (\Id+{ \bfh})\circ (a+\bff_n)  \circ (\Id+ { \bfh})^{-1},\\
\end{aligned}
\end{equation}
where $|\bfV_n|= O(\Delta_{0,n})$ and $\bff_{n+1}$ satisfies estimate \eqref{new_error}. 

For any vector $C$ such that $|C|= O(\Delta_{0,n})$, the function $\hat\bfh=\bfh+ C$ satisfies 
\begin{equation} 
a+\hat \bff_{n+1}+\hat \bfV_{n} = (\Id+{ \hat\bfh})\circ (a+\bff_n)  \circ (\Id+ {\hat \bfh})^{-1},
\end{equation}
where $|\hat \bfV_n |= O(\Delta_{0,n})$ and $\hat \bff_{n+1}$ satisfies  \eqref{new_error}. 
\end{lemma}
\begin{proof} Using \eqref{inv}, we can write:
\begin{equation*}
\begin{aligned}
(\id+\bfh+C)\circ(a+\bff_n)&= (\id+\bfh)\circ(a+\bff_n)+C= (a+\bff_{n+1}+\bfV_n+C)\circ (\id +\bfh)\\
&= (a+\bff_{n+1}+\bfV_n+C)\circ (\id-C)\circ  (\id +\bfh+C)\\
&=( a+ \bff_{n+1}\circ  (\id-C)+ (\bfV_n-(A-\id)C))\circ   (\id +\bfh+C)\\
&=( a+ \hat \bff_{n+1}+ \hat \bfV_n)\circ   (\id +\bfh+C),
\end{aligned}
\end{equation*}
where $\hat \bff_{n+1}:= \bff_{n+1}\circ  (\id-C)$ and $\hat \bfV_n:= \bfV_n-(A-\id)C$. Estimate \eqref{new_error} holds then for $\hat \bff_{n+1}$ since it holds for $\bff_{n+1}$ and $|C|= O(\Delta_{0,n})$.  Obviously, $|\hat \bfV_n |$ is of the same order of magnitude as $|C|= O(\Delta_{0,n})$. 
\end{proof}

In the previous part of the proof we constructed $\bfh_n$ such that 
$$a+ \bff_{n+1}-\bfV_{n} = (\Id+{ \bfh_n})^{-1}\circ (a+\bff_n)  \circ (\Id+ { \bfh_n}).
$$
If $\bfh_n$ satisfies equation \eqref{inv}, we can apply Lemma \ref{Cchange} to adjust the average of $ \bfh_n$. The following Lemma explains how the the constant vector $C$ is chosen at the $n$-th step of the iteration. Let $ {H}_{n-1}= \mathcal H_{n-1}-\id$, and recall that, by assumption, $ave ( {H}_{n-1})=0$.

\begin{lemma}\label{c}
Let $C= -\int_{\T^d} {\bfh}_n \circ \mathcal{H}_{n-1}$, and let $\hat \bfh_{n}= \bfh_n+C$. Let $\hat{\mathcal H}_n=(id +\hat  \bfh_{n}) \circ \mathcal H_{n-1} = \id +\hat { H}_n$. Then $ave(\hat {H}_n)=0$.

\end{lemma}

\begin{proof}
$\hat{\mathcal H}_{n}=(\Id+\hat\bfh_{n})\circ \mathcal H_{n-1}$ implies
$$\hat H_{n}= H_{n-1}+  \hat \bfh_{n}\circ \mathcal H_{n-1}.$$
By the inductive assumption, $ave(H_{n-1})=0$. 
Then, by taking averages of both sides of the equation above, we get that $ave ({\hat H_{n}})=0$.
\end{proof}



After choosing $C$ as in Lemma \ref{c}, by applying Lemma \ref{Cchange}, we get the equation: 
$$a+ \hat\bff_{n+1}+\hat\bfV_{n} = (\Id+\hat\bfh_n)\circ (a+\bff_n)  \circ (\Id+ {\hat \bfh_n})^{-1},$$
which implies 
$$a+ \hat\bff_{n+1}+\hat\bfV_{n} = (\Id+\hat H_n)\circ (a+\bff)  \circ (\Id+ {\hat H_n})^{-1}.$$
From this, by composing on the right with $\Id+\hat H_n$ we get
$$A\hat H_n + \hat\bff_{n+1}(\Id+\hat H_n)+\hat\bfV_{n} = \bff+ {\hat H_n}\circ (a+\bff).$$
By taking averages with respect to the volume of both sides of the equation above and using the assumptions that $\bff$ has zero average and that $a+f$ is volume preserving,  it follows that 
$\hat\bfV_{n}=- \hat\bff_{n+1}(\Id+\hat H_n)$. 
This implies that $\hat\bfV_{n}=O( \|\hat\bff_{n+1}\|_0)$, which means that the constant $\hat\bfV_{n}$ can be  absorbed by  $\hat\bff_{n+1}$. 

From Lemma \ref{Cchange} we have that estimates \eqref{new_error} hold for  $\hat\bff_{n+1}$. Finally, we proclaim the new $\bff_{n+1}$ to be $\hat\bff_{n+1}$.

 \end{proof}

\subsection{Convergence of the iterative scheme} 

Once we have the result of Proposition \ref{iterative_step}, the set-up of the KAM scheme and its convergence is essentially the same as in usual applications of KAM method (see for example Section 5.4 in \cite{DK}). 

Assume that $\(a,b\)$ is a $(\gamma, \tau)$-Diophantine action. Let $\(a+\bff, b+\bfg\)$ be a small smooth perturbation of $\(a,b\)$. Since we are proving KAM rigidity, we also assume that $\(a+\bff, b+\bfg\)$ is volume preserving and that $\bff$ and $\bfg$ have zero average. 

Given the initial perturbation above, we let:
$$
\bff_1=\bff;\,\,\, \bfg_1=\bfg.
$$
Recall that we use the notation $\Delta_{r,1}:=\max\{\|\bff_1\|_r, \|\bfg_1\|_r\}.$

Let $D$ be the constant from Proposition  \ref{iterative_step} which depends only on $\(a,b\)$.

Fix $k=\frac{4}{3},$ and let $l= 8D+16$.

At the first step we assume that:
$$
\Delta_{0,1}<\varepsilon,\,\,\, \Delta_{l,1}<\varepsilon^{-1}
$$
for a small $\varepsilon>0$. We will show that $\varepsilon$ can be chosen so small that the iterative process converges.

 We describe now the iterative process. By Proposition  \ref{iterative_step}, there exists $\bfh_1$ with the estimates claimed in the proposition.
Then the transformation $\Id+ \bfh_1$ conjugates $\(a+\bff_1, b+\bfg_1\)$ to a new perturbation, which we call $\(a+\bff_2, b+\bfg_2\)$. 
This procedure is iterated. 

At this point we still have the freedom to choose the truncation at level $N$ when applying Proposition   \ref{iterative_step} at the $n$-th step of the iteration. If at step $n$ we choose the truncation   to be $$N_n=\varepsilon_n^{-\frac{1}{3(D+2)}},$$ where $\varepsilon _n=\varepsilon^{(k^n)}$,  then for sufficiently small $\varepsilon$ we can show inductively that the following estimates hold for all $n$: 
\begin{equation}\label{convergence_set_up}
\begin{aligned}
\Delta_{0,n}&<\varepsilon _n=\varepsilon^{(k^n)},\\
\Delta_{l, n}&<\varepsilon_n^{-1},\\
\end{aligned}
\end{equation}
and 
\begin{equation}\label{haha}
\|\bfh_{n}\|_1<\varepsilon_n^{\frac{1}{2}}.\\
\end{equation}
Here we use the letter C to denote any constant which depends only on the fixed $l$, $D$ and the unperturbed action $\(a,b\)$. 
  
Suppose that we are at the $n$-th step of iteration and that estimates \eqref{convergence_set_up} hold for $n$. First, we check that the condition $N^D\Delta_{1,n}<1$ holds by using the standard interpolation inequality 
 \begin{equation}\label{interpolation}
  \Delta_{1,n}\le C \Delta_{0, n} ^{1-\frac{1}{l}}\Delta_{l, n}^{\frac{1}{l}},
 \end{equation}
  and assumptions \eqref{convergence_set_up}:
  $$
  N_n^D\Delta_{1,n}\le C \ve_n^{\frac{-D}{3(D+2)}} \ve_n^{1-\frac{1}{l}}\ve_n^{-\frac{1}{l}}= C\ve^{\frac{-D}{3(D+2)}+1-\frac{2}{l}}.
  $$
  Since for the chosen value of $l$ the term $\frac{-D}{3(D+2)}+1-\frac{2}{l}$ is positive, by choosing initial $\ve$ sufficiently small, 
  we get that  the above expression is smaller than 1.

  Then the maps $\bfh_n$, $\bff_{n+1}$ and  $\bfg_{n+1}$ are constructed by applying Proposition  \ref{iterative_step}. 
 The same proposition, combined with  \eqref{convergence_set_up} and 
 the interpolation inequalities \eqref{interpolation}, together with  our choice of $l$, imply that \eqref{haha} holds for a sufficiently small $\varepsilon$: 
$$
\|h_{n}\|_1\le CN_n^{D}\Delta_{1, n}\le C N_n^D \Delta_{0, n}^{1-\frac{1}{l}} \Delta_{l, n}^{\frac{1}{l}}\le C \ve_n^{-\frac{D}{3(D+2)}+1-\frac{2}{l}}\le \ve_n^{\frac{1}{2}}.
$$

 Now we check that   \eqref{convergence_set_up} holds for $n$ replaced by  $n+1$. 
 
 First we compute the bounds for the $l$ norms by using the estimates of Proposition  \ref{iterative_step}: 
 $$
  \Delta_{l, n+1}\le CN_n^D\Delta_{l,n}\le C\ve_n^{\frac{-D}{3(D+2)}}(1+\ve_n^{-1})\le 2C\ve_n^{\frac{-D}{3(D+2)}-1}< \ve_n^{-\frac{1}{3}-1}=\ve_n^{-\frac{4}{3}}=\ve_{n+1}^{-1}.
  $$

 Finally, we estimate the 0-norms (by using again estimates in Proposition  \ref{iterative_step} and the interpolation inequality): 
  \begin{equation*}
  \begin{aligned}
   \Delta_{0, n+1}&\le CN_n^D \Delta_{0, n}^{1-\frac{1}{l}} \Delta_{l, n}^{\frac{1}{l}}\Delta_{0, n}+ CN_n^{-l+D}\Delta_{l, n}\\
   &\le C(\ve_n^{-\frac{D}{3(D+2)}+2-\frac{2}{l}}+\ve_n^{\frac{l-D}{3(D+2)}-1})\le \ve_n^{\frac{4}{3}}= \ve_{n+1},
  \end{aligned}
  \end{equation*}
 since, given our choice of $l$, both expressions $-\frac{D}{3(D+2)}+2-\frac{2}{l}$ and $\frac{l-D}{3(D+2)}-1$ are strictly larger than $\frac{4}{3}$.

 Therefore the estimates in \eqref{convergence_set_up} hold for all $n$. This implies the convergence of $\mathcal H_n$ in the $C^1$ norm to some $\mathcal H_\infty$, which conjugates the initial perturbation to $\(a, b\)$. The fact that the conjugation $\mathcal H_\infty$  is $C^m$ for every $m>0$ (i.e., that the process converges in any norm) is proved in a standard way by using interpolation estimates (see for example the end of Section 5.4 in \cite{DK}).

\color{black}

\subsection{Volume preservation of the conjugacy}\label{s_volume_pres} Now we have that $\mathcal H_\infty$ conjugates the perturbation $\(F, G\)$ to $\(a,b\)$. Since $\(F, G\)$ is assumed to be volume preserving, the conjugation relation implies that the pushforward of the volume by $\mathcal H_\infty$ is invariant under $\(a,b\)$. Since  $a^kb^l$ is uniquely ergodic for some $k$ and $l$, the map $\mathcal H_\infty$ is volume preserving.

{The proof of Theorem \ref{mainC} is now completed {\it modulo} the proof of Proposition \ref{Main iteration step}. The rest of the paper is dedicated to the proof of  Proposition \ref{Main iteration step}. \hfill $\Box$}
\color{black}

\section{ Estimates of sums and double sums along the dual orbits} \label{sec3}

In this subsection we give the necessary estimates on 
sums and double sums along the dual orbits of $A$ and $B$ that will be crucial in solving the cohomological equations and proving Proposition \ref{Main iteration step}. The main results of this section are Propositions \ref{c_est_easy} and \ref{lemma.double.sums.case3}.

In Proposition \ref{c_est_easy} we deal with partial 
sums along the step-2 dual orbits.  These sums will be used in the proof of Proposition \ref{Main iteration step} for estimating the norms of the conjugacies. We also give a first estimation of full sums along the step-2 dual orbits that will be the key for estimating the error in solving the cohomological equations at the resonant Fourier modes. 

As explained  in  
Section  \ref{s_case3},  a full sum along the step-2 dual orbit can be reinterpreted,  {\it via} the higher rank trick, as a double sum of a quadratically small function $\phi$ measuring the error of the pair $(f,g)$ in \eqref{linearisation22} from forming a cocycle above the action $\(a,b\)$.  In Proposition \ref{lemma.double.sums.case3}, we deal with these  double sums. The estimates we obtain in this proposition will serve for estimating  the  error in solving the cohomological equations at non-resonant Fourier modes. 


\subsection{Notations}\label{s_notations} 
In this subsection we summarise the notations used in the rest of the paper. 

\begin{itemize} 
\item Assume that $\(A,B\)$ is unlocked commuting linear parabolic action.  We consider  two commuting affine maps $a(x)=Ax+\alpha$ and  $b(x)=Bx+\alpha$  on $\mathbb T^d$, where 
$a$ is step-2 and $b$ is  step-$S$ (see Definition \ref{def_step}). Elements of the step-$S$ action $\(a,b\):\ZZ^2 \to \text{Diff\,}^\infty_\lambda(\T^d)$ are denoted by $a^kb^l$, $(k, l)\in \ZZ^2$. 

\item  Let 
$$A= \id +\tA;  \quad  B =\id +\tB.$$ 
In these notations, 
$a$ being step-2 and $b$ being  step-$S$  implies: $\tA^2=\tB^S=0$.

\item Let  $\bA=(A^{tr})^{-1}$, $\bB=(B^{tr})^{-1}$. The linear action of $\(\bA, \bB\)$ on $\ZZ^d$ is called the dual action of $\(A, B\)$.
Let 
$$\bA=\id +\hA; \quad  \bB =\id  +\hB.$$ 
Clearly, $\bA$ and $\bB$ are also step-2 and step-$S$, respectively, which implies $\hA^2=\hB^S=0$.

\item
For $S$ being the step of the action, let 
$$
\eta =0.99\frac1S;
$$

\item  
 To each $m\in \Z^d$ we associate $s=s(m)$, called the step of $m$, such that 
$$
\hB^{s} m = 0, \quad \hB^{s-1} m \neq 0.
$$
Denote 
\begin{equation}\label{not_ett}
\ett = \ett (m) = 0.99 \frac1{s} .
\end{equation}
 Clearly,  we have $s(m)\leq S$ for any $m\in \ZZ^d$, and hence, $\ett (m)\geq \eta$.

\item For each $(k,l)\in\ZZ^2$, let $\alpha_{k,l} $ stand for the translation part of $a^kb^l$:
$$
\alpha_{k,l} := a^kb^l -A^kB^l .
$$


\item  
Given a continuous function  $h: \T^d\to\R$, denote  its Fourier coefficients by $h_m$:
$$
h(x)=\sum_{m\in \Z^d} h_m e(m,x),    \quad  e(m,x):=e^{2\pi i (m,x)} . 
$$
In these notations, for $a(x)=Ax+\alpha$ we have: 
$$
h\circ a =\sum_{m\in \Z^d} h_{\bA m} e(\bA m,\alpha ) e(m,x), \quad (h\circ a)_m =h_{\bA m} e(\bA m,\alpha ) .
$$

\item For $(k,l)\in \Z^2$  denote by  $\partial_{k,l}$  the coboundary operator: for  $h\in  C^\infty(\mathbb T^d)$ let
$$
\partial_ {k,l}(h): = h(a^kb^l ) - h. 
$$
In particular, the expression for the $m$-th Fourier coefficient of a coboundary is
$$
(\partial_ {k,l}(h))_m =  h_{\bA^k\bB^l m} e(\bA^k\bB^l  m,\akl ) -h_m .
$$


\item We will work with the maps of the type  $p: \ZZ^2 \to  C^\infty(\mathbb T^d)$, 
the usual notation being: $p(k,l)\in C^\infty$ for $(k,l)\in \ZZ^2$. 
For such maps we define the operator
$Lp: \mathbb Z^2\times \mathbb Z^2 \to  C^\infty(\mathbb T^d)$, by the following.  
For any $(k,l),(s,t) \in \ZZ^2\times \ZZ^2 $, denote
$$
Lp((k,l), (s,t)):= \partial_ {k,l} p(s,t)- \partial_ {s,t} p(k,l) .
$$ 


\item  Let $m$ be such that $\bA m\neq m$ (hence, $\hA m\neq 0$). Define 
$$
\begin{aligned}
\cM(A)=\{ m\in \Z^d \mid \langle m, \hA m\rangle >  0\}, \\ 
\cN(A)=\{ m\in \Z^d \mid \langle m, \hA m\rangle< 0\}.
\end{aligned}
$$

\item Suppose that $A$ is step-2, and $m$ is such that  $\bA m\neq m$. Then we have $\bA^k m=m+k\hA m$. 
We say that  $\bm$ {\it is the lowest point on the $\bA$-orbit of $m$}  if 
$$
| \bar m | \leq  |\bm+k \hA \bm | \quad \text{for all }k\in \Z.
$$
Then we have a "switch": $\bm \in \cM (A)$ but $\bA^{-1} \bm \in \cN (A)$, or vise versa. 
Note that $\bm$ is the only point on the corresponding $\bA$-orbit in which the "switch"
between $\cN(A)$ and $\cM(A)$  happens.

We write $m= \bm$ to say that $m$ is the lowest point on its own $\bA$-orbit, and $m\neq \bm$ otherwise.

\item Let
$$
\la_m^{(-1)}:=\la_m:= e(m,\a), 
\quad 
\mu_m^{(-1)}:= \mu_m:= e(m,\be),
$$ 
$$
\begin{aligned}
&\la_m^{(k)}= 
\la_{\bar A m}\la_{\bar A^2 m} \dots \la_{\bar A^{k} m},
\quad k=1,2,\dots, \quad 
\la_m^{(0)}=1, \\
&\la_m^{(k)}=(\la_{m}\la_{\bar A^{-1} m} \dots \la_{\bar
 A^{k+1} m})^{-1}, \quad k=-2,-3,\dots , \\ 
&\mu_m^{(k)}= 
\mu_{\bar B m}\mu_{\bar B^2 m} \dots \mu_{\bar B^{k} m}, \quad k=1,2,\dots, \quad 
\mu_m^{(0)}=1, \\
&\mu_m^{(k)}=(\mu_{m}\mu_{\bar B^{-1} m} \dots 
\mu_{\bar  B^{k+1} m})^{-1}, \quad k=-2,-3,\dots.
\end{aligned}
$$ 

\item For $A$ of step-2 and $m\in \Z$, consider the following partial sums over the dual orbit of $A$: 
$$
\begin{aligned}
&\Si_m^{+ , A} (f):=
\sum_{k=0}^{\infty} f_{\bar A^{k} m}\la^{(k)}_{ m} ,\quad \Si_m^{- , A} (f):=
\sum_{k=-\infty}^{-1} f_{\bar A^{k} m}\la^{(k)}_{ m}, \\
& \Si_m^A (f):= \Si_m^{+ , A} (f)+\Si_m^{- , A} (f).
\end{aligned}
$$
The last two-sided sum defines the so-called  "obstruction operator".
\item Let $U\subset \Z^d$ be a set  that is {\it invariant under the action of} $\(\bA,\bB\)$, 
i.e., for any $m\in U$ we have $\bA^s\bB^t m\in U$. Consider a set of real numbers, indexed by $U$: $\xi=(\xi_m)=\{\xi_m\mid m\in U\}$.
With a little abuse of notation, we let the operator $\partial_{s,t}$ act on $\xi$.
Namely,  
$$
(\partial_{s,t} \xi)_m =   \xi_{\bA^s\bB^tm} e(\bA^s\bB^tm,\a_{s,t})-h_m.
$$
Since the set of indices is invariant under the action, this expression is well-defined.

\noindent Comment: This notation is needed because we will define the conjugating functions $h$ by Fourier coefficients in different ways for different (invariant) sets of indices: $\cC_1$, $\cC_2$ or $\cC_3$ (see below), and will need to solve equations in terms of Fourier coefficients before we have defined $h$ as a function. 

%

\item 
The following splitting of $\Z^d\setminus \{0\}$ will be used in our analysis.
$$
\Z^d\setminus \{0\}=\cC_{1} \cup \cC_{2} \cup \cC_{3} ,
$$
where the sets $\cC_j$ are defined as follows.

  \begin{itemize}

\item[$\cC_{1} $.]   {\bf (Degenerate case).} $\cC_{1} $ is the set of $m$ for which  $\bA m =\bB m =m$. 

\item[$\cC_{2} $.]   {\bf (Resonant non-degenerate case).}   For $(k,l)\in \Z^2\setminus \{0\}$ we say that 
$m\in \cC_{2} (k,l)$ if the following holds:

- $m\notin \cC_{1} $; 

 
- $\bA^k\bB^lm = m$.  

We define $\cC_{2} =\bigcup_{(k,l)\in \Z^2\setminus \{0\}} \cC_{2}(k,l)$.

\medskip


\medskip

\item[$\cC_{3} $.]  {\bf (Non-resonant case).} $\cC_{3} = \Z^2\setminus ( \{0\} \cup \cC_{1} \cup \cC_{2} )$.

  \end{itemize}

\end{itemize}

\subsection{Estimates of the sums along the dual orbits of a step-2 matrix}

In the constructions that follow we will work with vectors $m$ lying in certain subsets of $\Z^d$ that are invariant under the action of $\bA$ and $\bB$. 
Recall the notations $\cM(A)$ and $\cN(A)$ from Sec.~\ref{s_notations}.
\begin{PProp}\label{c_est_easy} Let $A$ be step-2 and suppose that $\bA m \neq m$.  
\begin{enumerate}
\item Consider a set $U\subset \ZZ^d $ that  is invariant under the action of $\bA$. 
Let $(\xi_m)=\{ \xi_m \in \RR \mid m\in U \}$,  and suppose that  for all $m\in U$ we have 
 $|\xi_m |\leq  c_0 | m |^{-r} $. Then there exists $c>0$ such that 

- If  $m\in \cM(A)\cap U$, then $\sum_{k=0}^{\infty} | \xi_{\bA^k m} | \leq  c | m |^{-r+1}$,

- If  $m\in \cN(A)\cap U$, then $\sum_{k=-\infty}^{-1}  | \xi_{\bA^k m} | \leq 
c | m |^{-r+1}$,

- If, moreover, $m=\bm$ (i.e.,  $m$ is the lowest point on its $\bA$-orbit), then 
$$
\sum_{k=-\infty}^{\infty} | \xi_{\bA^k m} | \leq  c | m |^{-r+1} .  
$$

\item Consider the sets $U$ and $(\xi_m)$ as in (1) and $\la_m^{(1)}$ as in Sec.~\ref{s_notations}. 
Suppose that for each $m\in U$, the set of numbers $(\zeta_m)$ satisfies
$$
{\zeta_{\bA m}} \la_m^{(1)}-\zeta_m = \xi_m.
$$
Then there exists $c>0$ such that for each $m\in U$ we have:
$$
|\zeta_m | \leq  c | m |^{-r+1}.
$$

\item There exists $c>0$ such that for any function $\xi \in C^{r}$, we have:

- If  $m\in \cM(A)$, then $| \Sigma^{+,A}_m (\xi )|\leq c \|\xi \|_{r} | m |^{-r+1},$

- If  $m\in \cN(A)$, then 
 $|\Sigma^{-,A}_m (\xi )|\leq c\|\xi  \|_{r}| m |^{-r+1} .$
 
- If, moreover, $m=\bm$ is the lowest point on its $\bA$-orbit, then 
$$
|\Sigma^{A}_m (\xi )|\leq c  | m |^{-r+1}  .  
$$

\end{enumerate}

\end{PProp}
\begin{proof}
Item (1) follows directly from the   estimate below, that will be used several times in the paper.
\begin{sublemma}\label{l_est_easy} 
Let $A$ be step-2 and  $\bA m \neq m$.  Then there exists  $c=c(r,A)>0$ such that  we have:

- If  $m\in \cM(A)$, then $
\sum_{k=0}^{\infty}  |\bar A^k m |^{-r}  \leq c | m |^{-r+1} 
$;

- If  $m\in \cN(A)$, then $
\sum_{k=-\infty}^{-1}  |\bar A^k m |^{-r}  \leq c | m |^{-r+1} 
$.

- If, moreover, $m=\bm$ (i.e.,  $m$ is the lowest point on its $\bA$-orbit), then 
$$
\sum_{k=-\infty}^{\infty}  |\bar A^k m |^{-r}\leq c  | m |^{-r+1}  .  
$$
\end{sublemma}

\begin{proof} Since $\bA$ is step-2 and  $\bA m \neq m$,  for any $m\in \Z^d$ and 
$k\in \Z$, we have 
$\bA^k m=m+k \hA m$. 

Consider the case $m\in \cM(A)$, i.e., $\langle m, \hA m \rangle\geq 0$. Then we have: $|m+ \hA m|>|m|$ and $\langle m+ \hA m, \hA m \rangle> 0$ (notice the strict inequality). Let $p$ denote the projection of the vector  $\hA m$ onto the vector $m+ \hA m$. Clearly, $p$ is non-zero and has the same direction as $m+ \hA m$.
Therefore, $|m +k\hA m|=|(m +\hA m)+ (k-1)\hA m | \geq |m+ \hA m| + (k-1)|p|>|m|+ (k-1)|p|$. Hence, 
for an apprpriate constant $c=c(r,A)$ we have:
$$
 \sum_{k=0}^{\infty} | \bar A^k m |^{-r} \leq 
\sum_{k=0}^{\infty} | (m +k\hA m)|^{-r} \leq
c | m |^{-r+1} .
$$
To justify the last inequality, note that for any $x, y>0$,  and $r>1$ we have:
$$
\sum_{l=1}^{\infty} (x+ly)^{-r}\leq   \frac{c_1}{y(r-1)}\,(x+y)^{-r+1},
$$
which can be proved by comparison with the integral $ y^{-r} \int_{t=1}^{\infty} (\frac{x}{y}+t )^{-r} dt $. 

The case of $m\in \cN(A)$, i.e., $\langle m, \hA m \rangle< 0$, is similar. Indeed, the projection of the vector $(-\hA m)$ onto $m$ has the same direction as $m$, and the above calculation holds.

Now let $m=\bm$ be the lowest point on its $\bA$-orbit.  Then, in particular, $|m+ \hA m|> |m|$ and $|m- \hA m|> |m|$. This implies, for example by studying the triangle with two sides formed by vectors  $m+ \hA m$ and $m- \hA m$,  
that we have both $\langle m+ \hA m, \hA m \rangle >  0$ and $\langle m-  \hA m, \hA m \rangle >  0$. Then we can use the two estimates above to conclude that
$$
\sum_{k=-\infty}^{\infty} | \bA^k m  |^{-r} = \sum_{k=-\infty}^{-1} | \bA^k m |^{-r}  +\sum_{k=0}^{\infty} | \bA^k m |^{-r} \leq  c' | m |^{-r+1} .  
$$
\end{proof}

 To prove (2), for a fixed $m$ write the given equation at the points $\bA^k m$ either  for $k\geq 0$ or for $k\leq -1$, multiply by appropriate constants and add up, obtaining a telescopic sum on the left-hand side. Then we get that 
$|\zeta_m|\leq \sum_{k=0}^{\infty} | \xi_{\bA^k m} |^{-r} $ or $|\zeta_m|\leq  \sum_{k=-\infty}^{-1}  | \xi_{\bA^k m} |^{-r}$. The estimate follows from (1).

To prove prove (3), recall that 
the Fourier coefficients of any $\xi \in C^{r}$ satisfy for all $m\in \Z^d$:
$$
| {\xi_m}| \leq  \| \xi \|_{r}  | m |^{-r}.
$$
Then for each $k\in\Z $ we have: 
$| {\xi_{\bar A^{k}m}} |\leq     \| \xi \|_{r}  | {\bar A^{k} m} |^{-r}  $. The result reduces to that of item (1).
\end{proof}


\subsection{Estimates of the double sums. The parabolic higher rank trick}\label{double-buble}
The double sums will be used  for the case $m \in \cC_{3}$ (non-resonant case).
For each $m$, one of the double sums  is easier to estimate than the other.
The corresponding sign of $l$ will be called the "good sign" of $l$ for the given $m$.

\begin{PProp}[Estimate of the double sums] 
\label{lemma.double.sums.case3} 
Assume that $\(a,b\)$ is unlocked parabolic affine step-$S$ action, where $a$ is step-2. Suppose  that      
$m \in \cC_{3}$ is the lowest point on its $\bA$-orbit. 

For $r$ sufficiently large
there exists a constant $c=c(r, A,B)>0$ such that for $\eta=0.99/S$, at least one of the following holds:
$$
\sum_{k\in \Z} \sum_{l\geq 0} |\bA^k \bB^l m |^{-r} \leq c |m|^{-\eta r +  8},
\quad
\text{ } 
\quad
\sum_{k\in \Z} \sum_{l< 0} |\bA^k \bB^l m |^{-r} \leq c |m |^{-\eta r +  8}.
$$
\end{PProp} 

The proof of Proposition \ref{lemma.double.sums.case3} is crucial for our analysis, it is rather technical and takes up the rest of this section.

\subsubsection{Implications of being unlocked}\label{sec_C1_prelim}
Recall the notion of being unlocked from Definition \ref{def_stiff}.
Let us make two observations.
 
\begin{lemma} \label{lemma.hA2} 
Suppose that the action $\(a,b\)$ is  unlocked.   If $\hB^{2} m \neq 0$, then  
$\hA m \neq 0$. 
\end{lemma} 

\begin{proof} Let $\hB^{2} m \neq 0$, and suppose by contradiction that 
$\hA m = 0$. 
Consider the function $g(x)=e(m ,\tB x)$. 
Observe that 
$$
g(ax)=g(Ax+\alpha) = e(m , \tB( x +\tA x+\alpha))= e(m ,\tB x) e(m, \tB \tA x)e(m ,\tB \a).
$$ 
Since 
$m \tA =(\hA m)^t=0$, and since (by commutativity $ab=ba$) we have $\tB \a = \tA \beta=0$, we conclude that $g(ax)=g(x)$. 

On the other hand, 
$$
g(bx)=e(m ,\tB x)e(m , \tB^2 x)e(m ,\tB \beta) = g(x)e(m ,\tB^2 x)e(m, \tB \beta).
$$ 
Since $m \tB^2=\hB^2 m \neq 0$,  we conclude that the action $(a,b)$ has a rank-one factor that is not a translation. 
\end{proof}

The second observation is
\begin{lemma} \label{lemma.AB2}
Suppose that the action $\(a,b\)$ is unlocked and $\hA^2 m = 0$. 
If $\hB^{s} m=0$ for some $s\geq 2$, then $\hA \hB^{s-1} m = 0$.
\end{lemma}

\begin{proof} Let $\hB^{s} m=0$,  and suppose by contradiction that $\hA \hB^{s-1} m \neq 0$. Define $f(x)=e(m ,\tA \tB^{s-1} x)$. One easily verifies that 
$f(ax)=f(x)$ and $f(bx)=f(x)$ (using relations $\tA \tB^{s-1}\a = \tA^2 \tB^{s-2}\beta$, $\tA \tB^{s-1}\beta =  \tB^{s}\a$). Hence the action $\(A+\a,B+\beta\)$ has a rank-one factor equal to identity, contradicting the assumption.
\end{proof}







\subsubsection{Polynomial expansion of $\bA^k\bB^lm $}\label{sec_C3_prelim}

Recall the notations from \S \ref{s_notations}.

\begin{lemma}\label{l_form_of_AkBl}
Assume that $\(a,b\)$ is unlocked parabolic affine action, and $a$ is step-2.  For any $m\in\cC_3 $, $k,l\in\Z$, $s=s(m)$, there exists
$t=t(m,A,B)$, $1\leq t\leq s-1$, such that
 \begin{equation} \label{AkBl_s>2}
 \bA^k\bB^lm = m+k\hA m + \sum_{j=1}^{t-1} c_j  l^j\hB^j (m+k \hA m)+ \sum_{j=t}^{s-1} c_{j}   l^{j}  \hB^{j} m ,
 \end{equation}
where $c_1=1$, and all $c_j$ are positive constants that can be computed explicitly.
\end{lemma}

\noindent {\it Proof. }
\noindent {\bf Case $s=s(m)=2$.}  Here we have   $\hA^2 m=\hB^2 m=0$ (since $s(m)=2$), and $\hA m\neq 0$,  $\hB m\neq 0$ (since $m\in\cC_3 $). Therefore, 
$\bA^km=m+k \hA m$ and $\bB^lm=m+l \hB m$. By Lemma \ref{lemma.AB2},  $\hB^2m=0$
implies that $\hA\hB m=0$, which gives the result.

\noindent {\bf Case $s=s(m)\geq 3$.}  Here we have $\hB^2 m\neq 0$, so, by Lemma \ref{lemma.hA2}, $\hA m\neq 0$.
Since  $\hA^2m=\hB^sm=0$, we have
$\bA^km=m+k \hA m$ and $\bB^lm= m + \sum_{j=1}^{s-1} c_j  l^j\hB^j m$. Note that, by Lemma \ref{lemma.AB2},
 $\hA \hB^{s-1} m = 0$.
Composing the two expressions above and using the commutativity gives the result.
\qed


\bigskip

Recall that for $m\in\cC_3$ we have the following two possibilities:
\begin{itemize}
\item $s(m)= 2$,  in which case  $\hB^2m= 0$, and for all $(k,l)\in \Z^2\setminus  \{0\}$ we have $k \hA m + l\hB m \neq 0$;

\item $s(m)\geq 3$, in which case $\hB^2m\neq 0$. 
\end{itemize}
We will use different ways of controlling the double sums for the two cases above.

\medskip
\subsubsection{Proof of Proposition \ref{lemma.double.sums.case3}, case $s(m)=2$} 
In this section we fix $m\in \cC_3$, $s(m)=2$,
and study the growth properties of 
\begin{equation}\label{v_kl_s=2}
v_{k,l}:=\bA^k\bB^lm = m+k\hA m +k\hB m .
\end{equation}
The following two lemmas  prove that there exists a constant $c=c(A)>0$ such that if
$m$ is the lowest point on its $\bA$-orbit, 
then for all $k\in \ZZ$ and either for all $l\in \NN$ or for all $l\in (-\NN)$ we have:
$$
| v_{k,l} |> c |m|.
$$

\begin{lemma}\label{l_growth_1}
Let $m\in \cC_3$, $s(m)=2$.
If $|\hA m|> |m|$, then for all $k,l \in \ZZ$
we have:
$$
| v_{k,l} | > \|\hA \|^{-1} |m|.
$$
\end{lemma}
\begin{proof} 
Assume the contrary: $| v_{k,l} | \leq \|\hA \|^{-1} |m|$.
Apply $\hA$ to equality \eqref{v_kl_s=2}. Since $\hA^2m=\hA\hB m=0$, we have: 
$$
|\hA m| =|\hA v_{k,l}| \leq \|\hA \| \,  |v_{k,l}| \leq  |m|,
$$
contradicting the assumption of the lemma.
\end{proof} 

\begin{lemma}\label{l_growth_2}
Let $m\in \cC_3$ be the lowest point on its $\bA$-orbit, and $s(m)=2$. 

If $|\hA m|\leq  |m|$, then for all $k\in \ZZ$ and either for all $l\in \NN$ or for all $l\in (-\NN)$
we have:
$$
|v_{k,l}| \geq   |m|/2 .
$$ 
\end{lemma}
\begin{proof}  
Since $m\in \cC_3$, we have $\hA m \neq 0$. 
Denote $\cV_m=\text{span\,}\{ \hA m   \}$,   
and let  $m^\bot$ and $(\hB m)^\bot$ stand for the projections of $m$ and $\hB m$, respectively, onto the orthogonal complement of
$\cV_m$. 
Then
$$
| \bA^k \bB^l m| = |m + k\hA m + l\hB m| \geq | m^\bot + l (\hB m)^\bot | .
$$
To prove the lemma, it is enough to show that 
$| m^\bot | \geq |m|/2 $. When this is done, we choose the "good sign" of $l$ to be positive if the angle between the 
vectors $m^\bot $ and $ \hB m^\bot$ is acute, and negative otherwise.
If we choose $l$ of good sign, then $| \bA^k \bB^l m|  \geq   | m^\bot|  \geq |m|/2 $.

Let us estimate  $| m^\bot |$.  
Suppose that the angle $\theta$ between the vectors $m$ and $\hA m$ satisfies $0<\theta \leq \pi/2$ 
(otherwise, use $(-\hA m)$ instead of $\hA m$). 
Since $m$ is the lowest point on its $\hA$-orbit, we have: $|m-\hA m|\geq |m|$.  
We see in this case that  the projection of $m$ onto $\cV_m$ 
satisfies
$$
\|\text{proj}_{\cV_m}m\| \leq \|\hA m\|/2.
$$ 
Hence, 
$$ 
|m^\bot|  \geq |m| - |\text{proj}_{\cV_m}m| \geq |m| - |\hA m|/2 \geq |m| - |m|/2 =  |m|/2.
$$
\end{proof}

\bigskip


\begin{lemma}[Estimate of the double sums,  $s(m)=2$]
\label{lemma.double.sums.case3_s=2} 
Assume that $\(a,b\)$ is a unlocked parabolic affine action, and $a$ is step-2. Suppose  that      
$m \in \cC_{3}$, $s(m)=2$, and $m$ is the lowest point on its $\bA$-orbit. 

For $r$ sufficiently large
there exists a constant $c=c(r,A,B)>0$ such that at least one of the following holds:
$$
\sum_{k\in \Z} \sum_{l\geq 0} |\bA^k \bB^l m |^{-r} \leq C |m|^{- r + 8},
\quad
\quad
\sum_{k\in \Z} \sum_{l< 0} |\bA^k \bB^l m |^{-r} \leq C |m |^{- r + 8}.
$$
\end{lemma} 

\begin{proof} 
Assume without loss of generality that the good sign of $l$  is positive.

Denote $u:=\hA m$, $v:=\hB m$ and let $C_{k,l}= k\hA m +l \hB m$. 
Since $u$ and $v$  are non-parallel integer vectors whose  sizes satisfy, for some $c_0>0$,
$$
1\leq |u|, \, |v| \leq c_0|m|,
$$ 
one can show that the angle $\theta$ between $u$ and $v$  satisfies  
$\sin \theta \geq c_1|m|^{-2}$, and hence $\gamma:= |\cos \theta| \leq 1-c_2/|m|^{-4}$. 
Note that for all $k,l\in \ZZ$  we have the following two inequalities:
$$
\begin{aligned}
|C_{k,l}|^2= |k u +  l v|^2 = &k^2|u|^2+ l^2 |v|^2 +  2 kl |u||v| \gamma , \\
(k |u| +  l |v|)^2 = &k^2 |u|^2+ l^2 |v|^2 +2 kl |u||v| \geq 0.
\end{aligned}
$$
Hence,
$$
\begin{aligned}
|C_{k,l}|^2 &\geq |C_{k,l}|^2 - \gamma (k |u| +  l |v|)^2 =  (1-\gamma) (k^2 |u|^2+ l^2 |v|^2) \\
&\geq c_2|m|^{-4} (k^2 |u|^2 +  l^2 |v|^2)\geq  c_2 |m|^{-4} (k^2  +  l^2 ).
\end{aligned}
$$  
Now, if $k^2  +  l^2 \geq c_3 |m|^8$, then for  $|m|>1$ we have 
$|m|\leq |C_{k,l}|/2$, and 
$$
|\bA^k \bB^l m | \geq |C_{k,l}| - |m|   \geq 
|C_{k,l}|/2 \geq  |m| .
$$

Finally, we split the desired sum: 
$$
\sum_{k\in \Z} \sum_{l\geq 0} |\bA^k \bB^l m |^{-r} = \Sigma_1 +\Sigma_2,
$$
where $\Sigma_1$ contains the terms $|\bA^k \bB^l m |^{-r}$ corresponding to $k^2+l^2\leq c_3 |m|^8$, and $\Sigma_2$ contains those with
 $k^2+l^2> c_3 |m|^8$.
 
 The sum $\Sigma_1$ contains $\leq 4 c_3 |m|^8$ terms. 
 By Lemmas \ref{l_growth_1} and \ref{l_growth_2}, for a certain $c=c(r,A,B)$ we have for all $m$:  
$$
|v_{k,l} |^{-r} \leq c |m|^{- r}, 
$$  so 
 $$
 \Sigma_1\leq c_4 |m|^{- r + 8}  . 
 $$
 We estimate $ \Sigma_2$ by comparison with an integral:
 $$
\begin{aligned}
 \Sigma_2 = &\sum_{k^2+l^2\geq |m|^8 } |v_{k,l} |^{-r}
  \leq c_5 |m|^{2r} \sum_{k^2+l^2\geq |m|^8 } (k^2  +  l^2 )^{-r/2} \\
 \leq &c_5 |m|^{2r} \int_{x^2+y^2\geq |m|^8 } (x^2  +  y^2 )^{-r/2}\, dxdy \leq c_6 |m|^{-2r+8}  \leq c_6  |m|^{- r + 8}. 
 \end{aligned}
 $$ 
 The combination of the estimates for $ \Sigma_1$ and $ \Sigma_2$ provides the desired result.
\end{proof}


\bigskip

The following subsections contain the proof of Proposition \ref{lemma.double.sums.case3} for the case $s(m)\geq 3$.
We assume that $m\in \cC_3$, $s=s(m)\geq 3$,
and study the growth properties of $|\bA^k\bB^lm|$, given by formula \eqref{AkBl_s>2}:
$$
v_{k,l}:=\bA^k\bB^lm = m+k\hA m + \sum_{j=1}^{t} c_j  l^j\hB^j (m+k \hA m)+ \sum_{j=t+1}^{s-1} c_j  l^j  \hB^j m .
$$

\medskip


\subsubsection{The case $s(m)\geq 3$: Estimate for small $l$.} Recall the notation $\ett = \ett (m) = 0.99 \frac1{s} $ from Sec.~\ref{s_notations}.

\begin{lemma}
\label{lemma.small.l} 
Assume that $\(a,b\)$ is unlocked parabolic affine action, and $a$ is step-2.  
For any $\gax >0$, there exists $c=c(A,B,\gax)>0$ such that for any  
$m\in \cC_3$ with $s(m)\geq 3$,  being  the lowest point on its $\bA$-orbit,  we have for any $k\in \ZZ$ and for any $ | l |<|\gax m|^{\ett}$:
\begin{equation} \label{eq.double}
 |\bA^k \bB^l m|\geq c |m|^{\ett}. 
\end{equation}

\end{lemma} 

\begin{proof} 
Assume that $\gax=1$; the same proof holds for any for any $\gax >0$. Denote $v_{k,l}:=\bA^k\bB^lm$ for brevity. 
First consider ''large'' $m$, such that  $|m|>C_0$ for an appropriate constant $C_0$.
For this $m$, suppose by contradiction that $|v_{k,l}|< |m|^{\ett}$ for some $|l|\leq |m|^{\delta}$.
By assumption, $m$ is the lowest point on its $\bA$-orbit, so $|m+k\hA m|\geq |m|$ for any $k\in \ZZ$.
Applying inductively $\hB^{s-j}$, $j=1,\dots s-1$ to equation \eqref{AkBl_s>2}, we get for a certain constant $C=C(A,B)$:
$$
|v_{k,l}-(m+k\hA m)|\leq C |m|^{s\ett}\leq  C |m|^{0.99}.   
$$ 
If $m$ satisfies $|m|\geq (2C)^{100}:=C_0$, then the latter implies 
$$
|v_{k,l}-(m+k\hA m)| \leq |m|/2.
$$ 
Since $|m+k\hA m|\geq |m|$, we conclude that 
$|v_{k,l}|\geq |m|/2$ which is in contradiction with our assumption that $|v_{k,l}|\leq |m|^\ett$.  
Thus, we have proved the desired estimate for all $|m|\geq (2C)^{100}$. 
If $|m|< (2C)^{100}$, the estimate is achieved by the choice of a sufficiently small constant $c(A,B)$.
\end{proof}

\bigskip

\subsubsection{The case $s(m)\geq 3$: Linear Drift in $l$} This is the section where the "good sign of $l$" for the given $m\in\Z^d$
plays the crucial role.

\begin{lemma}[Linear Drift in $l$, $s(m)\geq 3$] \label{lemma.drift.l} Assume that $\(a,b\)$ is unlocked parabolic affine action, and $a$ is step-2.   
There exists a constant $c=c(A,B)>0$ such that for any  $m \in \cC_3$, $s(m)\geq 3$, the following holds:
for all $k\in \Z$ and either for all $l\geq 0$  or for all $l<0$ we have: 
$$
|\bA^k \bB^l m|\geq c|l|.
$$
\end{lemma}

\begin{proof}
   Recall that $s=s(m)\geq 3$, which means that $\hB^s m = 0$, $\hB^{s-1}m\neq 0$. In particular, $\hB^{2}m\neq 0$.
By Lemma \ref{lemma.AB2}, the assumption on being unlocked implies $\hA \hB^{s-1}m=0$.  
Denote 
$$
\cV_m=\text{span\,}\{ \hB^{l} m,\  \hA \hB^{l'} m \mid l\in [2,s-1], \ l'\in [0,s-2]  \},
$$
where the terms $\hA \hB^{l'} m$ may vanish starting from some $l'=t$, $t\geq 1$. 
Let $m^\bot$ and $(\hB m)^\bot$ denote the orthogonal  projections of $m$ and $\hB m$, respectively, onto the orthogonal complement of $\cV_m$. 
Let us show that for some constant $c=c (A,B)>0$ we have
$$|(\hB m)^\bot|\geq c >0.
$$
We consider two subcases. 

\noindent {\bf Case $\hA\hB^{s-2}m=0$. \ } 
Here we have: 
$$
0\neq \hB^{s-1} m = \hB^{s-2}  (\hB m)^\bot.
$$ 
Since $\hB^{s-1}m\neq 0$ is an integer, we have $|\hB^{s-1}m |\geq 1$. Since the norm of $\hB^{s-2} $ is bounded away from zero, we have 
$|(\hB m)^\bot | \geq c_0 (A,B)>0$. 

\noindent {\bf Case $\hA\hB^{s-2} m \neq 0$. \ }  
In this case we have
$$
0\neq \hA \hB^{s-2} m = \hA \hB^{s-3} (\hB m) = \hA  \hB^{s-3} (\hB m)^\bot.
$$
Since $\hA\hB^{s-2} m \neq 0$ is an integer, we have $|\hA \hB^{s-2} m |\geq 1$. Since the norm of $\hA  \hB^{s-3}$ is bounded away from zero, this implies $| (\hB m)^\bot|\geq c_1 (A,B) > 0$. 

To complete the proof, recall that, by \eqref{AkBl_s>2},
$$
\begin{aligned}
| v_{k,l}|=&|\bA^k\bB^lm| = |m+k\hA m + \sum_{j=1}^{t-1} c_j  l^j\hB^j (m+k \hA m)+ \sum_{j=t}^{s-1} c_{j}   l^{t}  \hB^{j} m|  \\
\geq &|m^\bot+l(\hB m)^\bot| .
\end{aligned}
$$
Choose the "good sign of $l$" to be positive if the vectors $m^\bot$ and $l(\hB m)^\bot$ form an acute angle, and negative otherwise.
For this sign of $l$ we get the desired result.
\end{proof}

\subsubsection{The case $s(m)\geq 3$: Drift in $k$}

Recall the notations $s=s(m)$ and  $\ett=\ett (m)$ from Section~\ref{s_notations}.

\begin{lemma}[Drift in $k$] \label{lemma.drift.k} 
Assume that $\(a,b\)$ is unlocked parabolic affine action, and $a$ is step-2. 
There exist positive constants $\gax=\gax(A,B)$ and $C=C(A,B)$ such that  for any  $m \in \cC_3$,
 for any $k,l$  satisfying $|k|\geq \gax |m| $, $|l|\leq |k|^{\ett }$ with $\ett=\ett (m)$ defined in Section~\ref{s_notations}, we have:
$$
|\bA^k \bB^{l} m| \geq C |k|^{\ett }.
$$ 
\end{lemma} 

\begin{proof}
Let $s=s(m)$ be the step of $m$, defined in Section~\ref{s_notations}. Since $\ett=\ett(m)=0.99/s$, condition  $|l|\leq |k|^{\ett }$ implies $|k|\geq |l|^s$. Let $p\in [0,s-1]$ be the largest integer such that $\hA \hB^p m\neq 0$, and observe that from \eqref{AkBl_s>2} 
and 
\begin{equation} 
\hB^p \bA^k \bB^l m=\hB^p m+k\hA \hB^p m+\cO(l^{s-1}), \label{aabb} 
\end{equation} 
where $\cO(l^{s-1})$ denotes the terms free from $k$ with the maximal power of $l$ being $s-1$. If we assume that $|k|\geq \gax (A,B)|m|$ with $\gax(A,B)$ sufficiently large, then the linear term in $k$ is dominant in \eqref{aabb} so that $|\hB^p \bA^k \bB^l m|\geq |k|/2$, thus  $|\bA^k \bB^l m|\geq C|k|$ for a certain positive constant $C$.
\end{proof}

\subsubsection{Proof of Proposition \ref{lemma.double.sums.case3}, case $s(m)\geq 3$}\label{sec_C3_proof}
We now turn to the effective control of the double sums.

\begin{proof}
Assume without loss of generality  that the good sign of $l$ is positive. Let $\gax=\gax(A,B)>0$ be the constant from Lemma \ref{lemma.drift.k}, and let $\ett=\ett(m)=0.99/s(m)$, as before.
We split the sum into the following five partial sums, each of which will be estimated separately:
$$
\begin{aligned}
\sum_{k\in \Z} \sum_{l\geq 0} |\bA^k \bB^l m|^{-r}  = &\left(  \sum_{|k|\leq |\gax m|,} \sum_{l\leq | \gax m|^{\ett }} +
 \sum_{|k| \leq |\gax m|,} \sum_{l >  | \gax m|^{\ett }} +    \right.\\
& \left.  \sum_{|k| > | \gax m|,} \sum_{l > | k|^{\ett }}+
    \sum_{|k| > |\gax m|,} \sum_{l < |\gax m|^{\ett } }   +
   \sum_{|k|> | \gax m|,} \sum_{|\gax m|^{\ett }\leq l\leq |k|^{\ett }}     \right)  |\bA^k \bB^l m|^{-r} \\
&:= \Sigma_1+\Sigma_2+\Sigma_3+\Sigma_4+\Sigma_5 .
\end{aligned}
$$

The following elementary estimate is used several times below: for any $p, r>0$, we have:
$$
 \sum_{j \geq p} j^{-r} \leq c_0(r) p^{-r+1}.
$$
 
 \noindent {\bf Estimate of $\Sigma_1$. \ }
By Lemma \ref{lemma.small.l}, there exists $c=c(A,B)$ such that  for all  
$l\leq  |\gax m|^\ett$ and for all $k\in \ZZ$
we have $ |\bA^k \bB^l m |\geq c |m |^{\ett}$. The sum $\Sigma_1$ contains $\leq 3 |\gax m |^{\ett +1}$.
Hence, 
 $$
 \Sigma_1  \leq  3  |\gax m|^{\ett + 1}  (c |m|^{\ett})^{-r}  <  c_1 |m|^{- r \ett +2}.
 $$

 \noindent {\bf Estimate of $\Sigma_2$ and $\Sigma_3$. \ }
 By Lemma  \ref{lemma.drift.l},  there exists $c=c(A,B)$ such that 
 for all $k\in \Z$ and for all $l\geq 0$ we have:   $|\bA^k \bB^l m | \geq c l$.
 In the case of $\Sigma_2$ we have $|k|\leq \gax |m|$, so 
 $$
\Sigma_2 \leq   3 \gax |m| \, c^{-r} \sum_{l> |\gax m|^{\ett}}    l^{-r}  \leq c_2 |m|^{- r\ett +2} .
 $$

 In the case of $\Sigma_3$ we have: 
 $$
 \Sigma_3 \leq   3 c^{-r} \sum_{k> \gax |m |, }  \sum_{l > k^{\ett}}    l^{-r} \leq 
 \tilde c_3 
 \sum_{k> \gax |m | } (k^{\ett})^{-r+1}  \leq c_3 |m |^{-r\ett+2}.
 $$

 \noindent {\bf Estimate of $\Sigma_4$ and $\Sigma_5$. \ }
By Lemma \ref{lemma.drift.k}, if $|k| > \gax |m |$ and $0\leq l\leq |k |^{\ett}$, then $|\bA^k \bB^l m |\geq C|k|^{\ett}$. 
Therefore, 
$$
\Sigma_4 \leq   C |\gax m |^{\ett}  \sum_{k > \gax | m |}   k ^{-\ett r} \leq c_4 |m |^{-\ett r +2},
$$ 
and  
$$
\Sigma_5 \leq    \sum_{l \geq |\gax  m |^{\ett},}    \sum_{k\geq l^{1/\ett}} k^{-r\ett} 
\leq  c_5    \sum_{l \geq | \gax m |^{\ett}}  ( l^{1/\ett} )^{-r\ett +1} \leq c_5 |m |^{-r\ett +2}.
$$
Recall that for any $m\in\ZZ^d$ we have $\ett(m)=0.99/s(m)\geq  \eta=0.99/S$, where $S$ is the step of $B$. Therefore,
$|m |^{-r\ett +2}\leq |m |^{-r\eta +2}$. Summing up the above estimates, we obtain the desired result.
\end{proof}

\section{Solution of the linearized problem. Proof of Proposition \ref{Main iteration step}
}\label{plan_of_proof}

A way of interpreting the statement of  Proposition \ref{Main iteration step} is the following. 
A perturbation $\(F, G\)$ of the action $\(a, b\)$ defines a map $\bfp: \mathbb Z^2\to \rm Vect^{\infty}(\mathbb T^d)$ 
by $\bfp(k, l):= F^kG^l-a^kb^l$, $(k, l)\in \mathbb Z^2$.  Proposition \ref{Main iteration step} 
(in fact) claims that there exists a tame map which projects $\bfp$ to the space of (twisted) 
coboundaries over $\(a, b\)$ in such a way that the complement of this projection has quadratic estimates with respect to $\bfp$. 
How is commutativity going to give us that the error we make while projecting is quadratically small? 
Commutativity relations for all action elements  tell us that certain linear operator $\bfL\bfp$ (see \S \ref{vf}) defined on $\bfp$ is bounded (roughly) by the size of the square of $\bfp$. So the core of the problem is to produce a projection of $\bfp$ to the space of (twisted) coboundaries over $\(a, b\)$ so that the complement of this projection (the error we are making) can be bounded by the size of $\bfL\bfp$. This is done in Sections \ref{Fourier} and \ref{s_end_prop}. The final \S \ref{proof_of_main_iteration_step} contains the proof of Proposition \ref{Main iteration step}. It is in this proof that we use the fact that  the commutativity assumption implies that $\bfL\bfp$ is quadratically small with respect to $\bfp$. 

This interpretation of the statement of  Proposition \ref{Main iteration step} is useful for understanding its proof. 
Namely, even though the statement of  Proposition \ref{Main iteration step} contains only $\bff$ and $\bfg$ 
(in the notations of this section it means that $\bff=\bfp(1,0)$ and $\bfg=\bfp(0,1)$), in order to produce the estimates, 
{\it we need to use the whole map} $\bfp: \mathbb Z^2\to \rm Vect^{\infty}(\mathbb T^d)$, not just $\bff$ and $\bfg$.  We explain this point  more in \S \ref{p} after the statement of Proposition \ref{prop_main_estimate} that contains the main estimates on the conjugacy and the error. 

The plan of the proof of Proposition \ref{Main iteration step} is the following. We start by constructing projections to coboundaries for  {\it function-valued} maps  $p: \mathbb Z^2\to C^{\infty}(\mathbb T^d)$. The main result that leads to Proposition \ref{Main iteration step}  is Proposition \ref{prop_main_estimate} which we state in \S \ref{p} and prove in \S \ref{Fourier}.  Proposition \ref{prop_main_estimate} contains the crucial estimates for the convergence of the iteration process.    In \S \ref{s_end_prop}
 we use Proposition \ref{prop_main_estimate} 
 to deduce the corresponding statement,
 Proposition \ref{...}, for truncations of $p: \mathbb Z^2\to C^{\infty}(\mathbb T^d)$, which we then inductively apply to obtain Proposition \ref{vector fields} for the truncated {\it vector field-valued} map $\bfp: \mathbb Z^2\to \rm Vect^{\infty}(\mathbb T^d)$.This passage from a function-valued map $p$ to a vector field-valued map $\bfp$ is quite direct due to the fact that our action $\(a, b\)$ has a {\it parabolic} linear part. 
 Similar inductive argument has been used in all the other works which use KAM method for parabolic actions (\cite{D}, \cite{DK2}, \cite{ZW}, \cite{DT}. Finally, the main result for vector field-valued maps $\bfp$ (Proposition \ref{vector fields}) is used in \S \ref{proof_of_main_iteration_step} to prove the main iterative step,  Proposition \ref{Main iteration step}.
\color{black}



\color{black}

\subsection{Approximating $p: \mathbb Z^2\to C^{\infty}(\mathbb T^d)$ by a coboundary}\label{p}

We start with a set of functions, $p:\ZZ^2 \to  C^\infty(\mathbb T^d)$. Recall the definitions of 
$
\partial_ {k,l}(h)
$
and 
$
Lp((k,l), (s,t))
$
from \S \ref{s_notations}. Here we introduce some extra notations.
For a fixed natural number $N$ we define  $\res_N$ to be  the set consisting of all the resonant pairs corresponding to the resonant vectors of norm less than $N$ (see Lemma \ref{l_tech_reson} for the bound on the norm of resonant pairs with respect to the resonant vector, and for the definition of constant $C$ which appears in the definition below). In other words,
\begin{equation}\label{R'}
\res_N:= \{(k, l)\in \mathcal R(A, B):\, \,  C(|k|+|l|)<N\}\cup \{(1,0), (0, 1)\}.
\end{equation}
For the simplicity of notations, we introduce the following norms, for any $r\ge 0$:
\begin{equation}\label{maxnorms}
\begin{aligned}
\|p\|_r &:= \max\{ \|p(1,0)\|_r,  \|p(0,1)\|_r\},\\
 \|Lp \|_{r, N}&:=\max\{  \|Lp((1,0), (k, l)\|_{r},  \|Lp((0,1), (k, l)\|_{r}:\, (k,l)\in \res_N\}.
 \end{aligned}
\end{equation}
From this point on, $\kappa$ will denote a constant which depends only on the action $\(a,b\)$ and the regularity $r$, but along the way it will absorb other constants which appear in the estimates.

The main result we prove here is: 

\begin{PProp}\label{prop_main_estimate}
Let $\(a,b\)$ be an unlocked  $(\ga,\tau)$-Diophantine step-$S$ parabolic  affine action, where $a$ is step-2, and let $r>0$.  
There exist constants $\etan=\etan (S)>0$,  $\sigman=\sigman(\(a,b\))$ and  $\kappa=\kappa(r,\(a,b\))$ such that 
 for any map $p: \ZZ^2 \mapsto C^\infty(\mathbb T^d)$ there exists 
 a $C^\infty$ function  $h$ such that 
\begin{equation}\label{hm}
\|h\|_r\leq 
\kappa \|p\|_{r+\sigman},
\end{equation}
and the map $\tp: \ZZ^2 \mapsto C^\infty(\mathbb T^d)$ defined by 
\begin{equation}\label{def_p_tilde_global}
\tilde p(s,t):= \partial_{s,t} h- p(s,t)+ave(p(s,t)), \,\,\,\, (s,t)\in \ZZ^2,
\end{equation}
 satisfies the following estimate:
\begin{equation}\label{linest}
\|\tp(s,t)\|_r\leq \|p(s,t)\|_r+  \kappa(|s|+|t|)^{dr}\|p\|_{r+\sigman}.
\end{equation}
Moreover, if $p$ is truncated up to $N$, then $\tilde p$ satisfies even the following estimate for any $(s, t)\in \res_N$: 
\begin{equation}\label{tildepst}
\|\tp(s,t)\|_r\le \kappa (|s|+|t|)^{d(\etan r+\sigman)} \|Lp\|_{\etan r+\sigman,N}.
\end{equation}

\end{PProp}

The proof of the proposition is lengthy and takes up all of the next section. Here is a short overview of the proof. 

We will define $h$ via its Fourier coefficients $h_m$, in different ways depending on $m$.  We need to apply different arguments 
in the following three cases: when the orbit of $m$ under the dual linear action $\(\bar A, \bar B\)$ is a single point,  when it is finite under one element of the action (but not under all elements), or when it is infinite. If the orbit is a single point we are in the degenerate case. The second case is when $m$ is resonant, otherwise $m$ is non-resonant. As explained in \S \ref{Diophantine conditions},  to each resonance $m$ we can attach a unique resonance pair $(k, l)$ for which  $\bar A^{k}\bar B^{l}m=m$.  

The special (degenerate) case when the $\(\bar A, \bar B\)$-orbit of $m$ is a single point, that is when  $\bar  Am=\bar  Bm=m$, is dealt with in the same way as in the original proof of Moser in \cite{M} (see \S \ref{s_case1}). 

Next, we have the situation when $\bar Am =m$ and $\bar Bm\ne m$. This is a $(1,0)$ resonance. In this case the fact that $\(A, B\)$ is unlocked implies that $\hat B^2m=0$ (so $m$ is step 2 for $\bar B$). Then we use the generator $b$ to construct $h_m$. To show that $\tp$ satisfies the needed estimate in this case, we will need the Diophantine condition on the translation vector $\al_{1,0}$. 

If  $\bar Am \ne m$, we use the generator $a$ to construct $h_m$, and use the fact that $\bar A$ is step 2 to estimate $h_m$.  To obtain the estimate for $\tilde p$ we use different strategies for resonant and non-resonant $m$.

When $m$ is resonant, we will use the corresponding resonant pair $(k,l)$ and the action element $a^kb^l$ to estimate the error. This is exactly where we need to use all the elements of the action, i.e., the map $\bfp: \mathbb Z^2\to \rm Vect^{\infty}(\mathbb T^d)$, and not just two generators $\bff$ and $\bfg$. Moreover, it is here that we will use the Diophantine assumptions on the translation parts $\alpha_{k, l}=a^kb^l-A^kB^l$. To control the number of action elements we use, we need to truncate the given data first. This is why the crucial error estimate in Proposition \ref{prop_main_estimate} is stated only for truncated maps. In our arguments, as explained in Lemma \ref{l_tech_reson}, the norm of the resonant pair will be bounded by the norm of the resonance.  Therefore, for the estimate of the $N$-truncated maps,  we only need to consider the resonant pairs for the resonances bounded by $N$. 
The treatment of all the resonant cases is done in \S \ref{s-case2}. 

Finally, if the $\(\bar A, \bar B\)$-orbit of $m$ is infinite, we use the double sums estimates. This part of the argument uses \S \ref{double-buble} and  is contained in \S \ref{s_case3}.

\subsubsection{Proof of  Proposition \ref{prop_main_estimate}  }\label{Fourier}

 Let us pass to defining and estimating the numbers $h_m$ and $(\tp(s,t))_m$. The arguments will strongly depend on $m$. 
Namely, we always set  $h_0=0$, and
for $m$ lying in each of the three subsets, $\cC_{1} $, $\cC_{2} $ and $\cC_{3} $, defined in \S \ref{s_notations}, we have to develop a special approach. A more precise statement of Proposition \ref{prop_main_estimate}  is the following.
\color{black} 
\begin{proposition}\label{prop_main_estimate_detail}
Let $\(a,b\)$ be an unlocked  $(\ga,\tau)$-Diophantine parabolic step-$S$ affine action, where $a$ is step-2, and let $r>0$.  There exist constants $\eta>0$ ($\eta=0.99/S)$), $\sigma=\sigma(\(a,b\))$ and  $\kappa=\kappa(r,\(a,b\))$ such that,  
for any  map $p: \ZZ^2 \mapsto C^\infty(\mathbb T^d)$ there exists
a set of numbers  $(h_m)$, $m\in \ZZ^d\setminus \{0\}$ such that 
\begin{equation}\label{hm_detail}
|h_m|\leq 
\kappa \max \{\|p(1, 0)\|_r,  \|p(0, 1)\|_r\} |m |^{-r+1+\tau}
\end{equation}
with the following property.  Define a new map $\tilde p$ from $\ZZ^2$ into the space of formal Fourier series 
as follows:
\begin{equation}\label{def_p_tilde}
(\tilde p(s,t))_m:=  h_{\bA^s\bB^tm} e(\bA^s\bB^tm,\a_{s,t})-h_m - (p(s,t))_m, \,\,\, (s,t)\in \ZZ^2.
\end{equation} 
Then it satisfies,  for any $ r>8/ \eta $:
\begin{equation}\label{tpst}
\begin{aligned}
&| \tp(s,t)_m|\leq \\
&{
{\scriptstyle
\begin{cases}
\kappa (|s|+|t|)^{d r}    \max \{\|Lp((s,t),(1,0))\|_r,\|Lp((s,t),(0,1))\|_r \} |m|^{-\eta r+  9+ \tau}, & m\in \cC_1\cup\cC_3\\
\kappa  (|s|+|t|)^{d r }   
\max \{ \|Lp((1, 0),(k,l))\|_r, \|Lp((0,1),(k,l))\|_r, \\
\qquad\qquad\qquad\qquad \| Lp((1,0),(s,t))\|_r , \| Lp((0,1),(s,t))\|_r  \} |m|^{-r+\tau+2},  & m \in \cC_2(k,l). \\ 
\end{cases}}
}
\end{aligned}
\end{equation}
Moreover, for $m \in \cC_2(k,l)$ we have $|k|+|l|\le C(A,B) |m|$.
\end{proposition}

\begin{proof}[\bf \bleu1 Proof of Proposition \ref{prop_main_estimate} from Proposition \ref{prop_main_estimate_detail}]

Recall that  in \eqref{maxnorms} we defined:
$$\|p\|_r:=\max \{\|p(1, 0)\|_r,  \|p(0, 1)\|_r\} .
$$ 
Then  estimate  \eqref{hm_detail} directly implies estimate \eqref{hm} for $h$ with a loss of $\sigman:=\tau +d+2$ derivatives. 

The map $\tilde p$, defined in \eqref{def_p_tilde_global}, satisfies the linear estimate \eqref{linest}, which follows from its 
definition and estimate \eqref{hm} for $h$:
$$\| \tilde p(s,t)\|_r\le  \|\partial_{s,t} h\|_r + \| p(s,t)\|_r\le \|p(s,t)\|_r+  \kappa(|s|+|t|)^{rd}\|p\|_{r+\sigman}.
$$

If $p$ is truncated up to $N$, then by taking the maximum on the right hand side of \eqref{tpst} over all resonant pairs  $(k, l)\in \res_N$ (which is a finite set), we get for any $(s,t)\in \res_N$ (see definition \eqref{maxnorms}):  
\begin{equation*}
\begin{aligned}| \tp(s,t)_m|&\le \kappa (|s|+|t|)^{d r }   \max_{(k, l)\in \res_N}\{ 
 \|Lp((1, 0),(k,l))\|_r, \|Lp((0,1),(k,l))\|_r,\\
 & \| Lp((1,0),(s,t))\|_r , \| Lp((0,1),(s,t))\|_r  \} |m|^{-\eta r+\tau+9}\\
& \le  \kappa (|s|+|t|)^{d r }   \|Lp\|_{r, N} |m|^{-\eta r+\tau+9}.
 \end{aligned}
 \end{equation*}

This implies that $$\sup | \tp(s,t)_m| |m|^{\eta r-\tau-9}\le  \kappa (|s|+|t|)^{d r }   \|Lp\|_{r, N},$$ which (by making a substitution $r:= \eta r-\tau-9$) gives:  
$$\sup_m \{| \tp(s,t)_m| |m|^{r}\}\le  \kappa (|s|+|t|)^{d (\eta^{-1} r+ \eta^{-1}(\tau +9) ) }   \|Lp\|_{\eta^{-1} r+ \eta^{-1}(\tau +9) , N}.$$
Because of the well known norm comparison: $\|\tp(s,t)\|_r\le C\sup_m \{| \tp(s,t)_m| |m|^{r+d+2}\}$, we have: 
$$
\|\tp(s,t)\|_r\le  \kappa (|s|+|t|)^{d (\eta^{-1} (r+d+2) + \eta^{-1}(\tau +9) ) }   \|Lp\|_{\eta^{-1}( r+d+2)+ \eta^{-1}(\tau +9) , N}.
$$
Now let  $\etan:=\eta^{-1}$ (recall that $\eta<1$), define the  new $\sigman:=\eta^{-1}(\tau +9+d+2)$ to obtain the final  estimate $\|\tp(s,t)\|_r\le  \kappa (|s|+|t|)^{d (\etan r+\sigman)}   \|Lp\|_{\etan r+ \sigman , N}.$

 \end{proof}

In the rest of this section we prove Proposition \ref{prop_main_estimate_detail}. We will split the proof into three subsections according to $m\in \cC_1$, $m\in \cC_2$ or $m\in \cC_3$.


\subsubsection{Proof of Proposition \ref{prop_main_estimate_detail} in the case $m\in\cC_1$}\label{s_case1}   Let $m\in \cC_1$, i.e., we have $\bA m=\bB m=m$. 
Since  the action is assumed to be $(\gamma,\tau)-$Diophantine, we have either  
$|e(m,\alpha ) -1|\ge \gamma \|m\|^{-\tau}$, or $|e(m,\be )-1|\ge \gamma \|m\|^{-\tau}$.
Let $f=p(1,0)$ and $g=p(0,1)$. 

Define $h_m$ as follows:

\begin{equation}\label{h_in_C1}
h_m: =
\begin{cases}
(e(\alpha, m)-1)^{-1}f_m, &  \text{ if } \ |1-e(\alpha, m)|\ge \gamma \|m\|^{-\tau},\\
(e(\beta, m)-1)^{-1}g_m, &    \text{ if } \ |1-e(\alpha, m)|< \gamma \|m\|^{-\tau} .
\end{cases}
\end{equation}


Then we have the following 

\begin{proposition}\label{l_Mosers_trick}
Let $\(a,b\)$ be a $(\ga,\tau)$-Diophantine parabolic affine action, and let a map $p: \ZZ^2 \mapsto C^\infty(\mathbb T^d)$, be given. 
For $m\in \cC_1$, let $h_m$ be defined as in \eqref{h_in_C1}. Then 
$$
|h_m| \le \ga \max\{\|f\|_r, \|g\|_r\} \, |m |^{-r+\tau},
$$
and for any $(s, t)\in \mathbb \Z^2$,  the number $(\tp(s,t))_m$ defined   by formula \eqref{def_p_tilde}, which in this case has the form:
$$
(\tp(s,t))_m = h_m (e(m, \alpha_{s,t})-1) -(p(s,t))_m,
$$
satisfies 
$$
 |(\tp(s,t))_m|\le c \max \{\|Lp((s,t),(1,0))\|_r, \|Lp((s,t),(0,1))\|_r \} |m|^{-r+\tau}.
 $$
 \end{proposition}
\begin{proof} 

Suppose first that $m$ is such that $|1-e(\alpha, m)|\ge \gamma \|m\|^{-\tau}$, in which case $h_m=  (e(\alpha, m)-1)^{-1}f_m$. Then 
$$
\begin{aligned}
\left( Lp((s,t), (1,0)) \right)_m =&\left( \partial_{s,t}  f\right)_m  - \left(  \partial_{1,0} p(s,t) \right)_m \\
=&(e(\a_{s,t},m)-1) \, f _m -  (e(\a,m )-1)  \left( p(s,t)\right)_m = \\ 
=&(e(\a_{s,t},m)-1)  \,  h_m(e(\alpha, m)-1) -  (e(\a,m )-1)  \left( p(s,t)\right)_m = \\
=&(e(\a,m)-1)  \,  \left(  e(\a_{s,t},m)-1)  h_m-  \left( p(s,t)\right)_m   \right) \\
= & (e(\a,m)-1) (\tp(s,t))_m . \\
 \end{aligned}
$$
Estimate $|1-e(\alpha, m)|\ge \gamma \|m\|^{-\tau}$ implies the result. The case when $h_m=  (e(\be, m)-1)^{-1}g_m$ is treated in the same way.

Directly from the definition of $h_m$ and from the SDC-condition on $\alpha$ and $\beta$ we obtain the bound for $|h_m|$:
$$
\begin{aligned}
|h_m|\le &\max \{ |e(\be, m)-1|^{-1}, |e(\alpha, m)-1|^{-1}\} \max\{ |f_m|, |g_m|\} \\
\le  &\gamma |m|^{\tau} |m|^{-r} \max\{ \|f\|_r, \|g\|_r\}.
\end{aligned}
$$
\end{proof}
 
It is straightforward that Proposition \ref{l_Mosers_trick} implies Proposition \ref{prop_main_estimate_detail} in case $m\in \cC_{1} $.


\subsubsection{Proof of Proposition \ref{prop_main_estimate_detail} in the case $m\in\cC_2$}\label{s-case2}
 Let $\(a,b\)$ be an unlocked parabolic affine action, where $a$ is step-2. 
Let $m\in \cC_2(k,l)$, i.e.,  at least one of $\bA m$ and $\bB m$ is different from $m$, and there exists 
$(k,l)\in \Z^2\setminus \{ 0\}$ such that $\bA^k\bB^lm = m$ (thus $k\hA m + l\hB m=0$).
By Lemma \ref{lemma.hA2},  $\cC_2(k,l)$ can be divided into two sub-cases: 

$\cC_2'(k,l)$: \  $\bA m \neq m$;

$\cC_2''(k,l)$: \ $\bA m = m$, while $\bB m \neq m$.  Note that in this case, by Lemma \ref{lemma.hA2}, we have $\bB^2 m =m$. 

\medskip

\noindent The sets $\cC_2'(k,l)$ and $\cC_2''(k,l)$ are invariant under the action of $\(A,B\)$ due to the commutativity of the action.
Consider  a map $p: \ZZ^2 \mapsto C^\infty(\mathbb T^d)$.
Denote
$f= p(1,0)$ and $ g= p(0,1)$.

For $m\in\cC_2'(k,l)$ we define:
\begin{equation}\label{def_h_Case2}
 h_m=
\begin{cases}
\Sigma^{+,A}_m (f), & m \in \cM(A), \\
\Sigma^{-,A}_m (f), & m\in \cN(A).
\end{cases}
\end{equation}

For $m\in\cC_2''(k,l)$
we define:
\begin{equation}\label{def_h_B}
 h_m=
\begin{cases}
\Sigma^{+,B}_m (g), & m \in \cM(B), \\
\Sigma^{-,B}_m (g), & m\in \cN(B).
\end{cases}
\end{equation}
To understand our choice for $h_m$, think of $C^{\infty}$-functions $h$ and $f$ satisfying $h\circ a-h=f$. Then the Fourier coefficients are related by 
$$
(h\circ a-h)_m ={h_{\bA m}} \la_m^{(1)}-h_m = f_m.
$$
Iterating this equality by $\bA$ either in the positive or in the negative direction while multiplying by appropriate constants, one obtains a telescopic sum equal to $h_m$ as given in formula  \eqref{def_h_Case2}.  
The following proposition is the main statement of this section. 
It is straightforward that it
implies Proposition \ref{prop_main_estimate_detail} in case $m\in \cC_{2} $.

\def \dd{{dr}}

\begin{proposition}\label{l_est_case2} Assume that $\(a,b\)$ is an unlocked $(\ga,\tau)$-Diophantine parabolic affine action, where $a$ is step-2. 
Consider  a map $p: \ZZ^2 \mapsto C^\infty(\mathbb T^d)$,
 denote
$f= p(1,0)$ and $ g= p(0,1)$. 
Given $m\in \cC_2(k,l)$ for some $(k,l) \in \Z^2 \setminus \{ 0\}$, 
define $h_m$  as above. 

Then there exists $\kappa=\kappa (\ga,\tau, A, B ) >0$ such that
$$
\begin{aligned}
&|h_m|  \leq   \kappa \max \{ \|f \|_r , \|g \|_r  \}\, |m|^{-r+1}, \\
\end{aligned}
$$
and for any $(s, t)\in \mathbb \Z^2$,  the number $(\tp(s,t))_m$  defined   by formula \eqref{def_p_tilde}, i.e.,
\begin{equation*}
\tp(s,t)_m = h_{\bA^s\bB^tm} e(\bA^s\bB^tm,\a_{s,t})-h_m - (p(s,t))_m,
\end{equation*} 
satisfies:
\begin{equation}\label{est_tp_case2}
\begin{aligned}
| (\tp(s,t))_m|\leq  
\kappa (|s|+|t|)^{\dd }   
\max \{ &\|Lp((1, 0),(k,l))\|_r, \|Lp((0,1),(k,l))\|_r, \\
&\| Lp((1,0),(s,t))\|_r , \| Lp((0,1),(s,t))\|_r  \} |m|^{-r+\tau+2}.
\end{aligned}
\end{equation}
In addition, $|k|+|l|\le c |m|$ for a certain $c=c(A,B)>0.$
\end{proposition}

\begin{proof}
Let us present the proof of Proposition \ref{l_est_case2} modulo certain lemmas, that are proved below.  
First consider $m\in \cC_2'(k,l)$; the arguments for $m\in \cC_2''(k,l)$ are similar. 

(i) Define $h_m$ by  \eqref{def_h_Case2}. The estimate for $|h_m|$ follows from Proposition \ref{c_est_easy} part {\it (3)}.

(ii) We start by proving estimate \eqref{est_tp_case2} for $(s,t)=(1, 0)$. This is done in Lemma \ref{l_formal_f_tildeC2}. 
Namely, we observe in that Lemma that the error term 
$\tf_m=\tp(1, 0)_m = h_{\bA m} \lambda_m^{(1)} - h_m - f_m$, vanishes for all $m$ except for those $m$ that are  lowest in norm on their $\bA$ orbit. In the latter case, we will derive from the commutation relation formula \eqref{eq_sum_C2} that we repeat here $$
 ( e(m, \akl)-1)\, \tf_{m}= \Sigma^{A}_{m} \left( Lp((1,0), (k,l) )\right) .
$$ 
After that, the right-hand side is bounded above by the norm of $Lp((1,0), (k,l)$ (because $m$ is lowest on its orbit), while the term $| e(m, \akl)-1|$ is bounded below by the Diophantine condition for the resonances. It is here that the Diophantine conditions on the resonances play a crucial role:   this condition, combined with formula \eqref{eq_sum_C2}, implies:
$$
{ |\tilde f_{m}|} \leq c \|Lp((1,0), (k,l) )  \|_r |m|^{-r+1+\tau}.
$$

(iii) Use step (ii) above to prove estimate \eqref{est_tp_case2} for all $(s,t)$.  Lemma  \ref{l_formal_p_tilde} 
derives the estimates on $(\tilde p(s,t))_m$ for any $(s,t)$ from those on $\tf_m=(\tilde p(1,0))_m$ (or on $\tg_m=(\tilde p(0,1))_m$, which will be relevant for $m\in \cC_2''(k,l)$).
We use Lemma  \ref{l_formal_p_tilde} with $\cK =c \|Lp((1,0), (k,l) )  \|_r$ and 
$\rho={-r+1+\tau}$ to get  \eqref{est_tp_case2} for all $(s,t)$. 

The arguments for $m\in \cC_2''(k,l)$ are similar:  define $h_m$ by \eqref{def_h_B} and estimate $|h_m|$ with the help of part {\it (3)} of Proposition  \ref{c_est_easy}; 
estimate $|\tg_m|:= |(\tp(0, 1)_m)|$  via $\|Lp((0,1), (k,l) )  \|_r$ with the help of Lemma \ref{l_formal_f_tildeC2}, and use it  instead of $|\tf_m|$ to get formula \eqref{est_tp_case2} for all $(s,t)$. 
\end{proof}

\bigskip 

The following lemma provides the proof for item (ii) above.  Below we write $m= \bm$ to say that $m$ is the lowest (in norm) point on its $\bA$ orbit, and $m\neq \bm$ otherwise. 
\begin{lemma}\label{l_formal_f_tildeC2} Assume that $\(a,b\)$ is an unlocked $(\ga,\tau)$-Diophantine parabolic affine action, where $a$ is step-2.  Let $m\in \cC_2(k,l)$.
Denote $f= p(1,0)$, $\tilde f_m=(\tp(1,0))_m$, $ g= p(0,1)$, $\tg_m=(\tp(0,1))_m$, 
and let $h_m$ be as in \eqref{def_h_Case2}, \eqref{def_h_B}. 
Then there exists a constant $c=c(r, A,B)>0$ such that 
for $m= \bm$ we have:
$$
\begin{cases} 
{| \tilde f_m|} \leq c \|Lp((1,0), (k,l) )  \|_r |m|^{-r+1+\tau} & \text{if } m\in \cC_2'(k,l), \\
{ |\tilde g_m|} \leq c \|Lp((0,1), (k,l) )  \|_r |m|^{-r+1+\tau} & \text{if } m\in \cC_2''(k,l).
\end{cases}
$$
For $m\neq \bm$ we have  $\tf_m=\tg_m=0$.  
\end{lemma}
 
\begin{proof} 
Assume that $m\in \cC_2'(k,l)$, the case $m\in \cC_2''(k,l)$ being similar. 
By the definition of $h_m$, we have:
$\tf_m = h_{\bA m} \lambda_m^{(1)} - h_m - f_m$, and formally we can express:
$$
\tf_m=
\begin{cases}
\Sigma^{A}_m (f), & m =\bm, \\
0, & \text{otherwise}.
\end{cases}
$$
Assume that  $m= \bm$.
Using the definition of $h_m$,
we get:
$$
\begin{aligned}
(Lp((1,0), (k,l) ))_m & = (\partial_{1,0} p(k,l) )_m  - (\partial_{k,l} f )_m   \\
& = (\partial_{1,0} p(k,l) )_m  -  ( e(m, \akl)-1) f_m   .
\end{aligned}
$$
Note that, by the commutativity of $a^kb^l$ and $a$, we get the relation: $A\akl=\akl$. Therefore, the term $ ( e(m, \akl)-1)$ is not changed when $m$ moves along the $\bA$-orbit.

Take the weighted sum $\Sigma^A_m$ on both sides and recall that for $m=\bm$ we have $ \Sigma^{A}_m (f) =\tf_m$:
\begin{equation}\label{eq_sum_C2}
\Sigma^{A}_m \left( Lp((1,0), (k,l) )\right) =  ( e(m, \akl)-1) \, \Sigma^{A}_m(f) =  ( e(m, \akl)-1)\, \tf_m.
\end{equation}

Since $Lp((1,0), (s,t))\in C^r$, we can use Proposition \ref{c_est_easy}, which gives us:
$$
|\Sigma^{A}_m (Lp((1,0), (k,l) ))|\leq  \|Lp((1,0), (k,l) ) \|_r |m|^{-r+1} .
$$ 
Using the Diophantine assumption on resonances, we conclude:
$$
|\tf_m|\leq \|Lp((1,0), (k,l) ) \|_r |m|^{-r+1} | e(m, \alpha)-1|^{-1} \leq  \|Lp((1,0), (k,l) ) \|_r |m|^{-r+1+\tau}. 
$$ \end{proof}

\bigskip

The following lemma derives the estimates on $(\tilde p(s,t))_m$ for any $(s,t)$ from those on $\tf_m=(\tilde p(1,0))_m$ or $\tg_m=(\tilde p(0,1))_m$, providing the details for item (iii).

\begin{lemma}\label{l_formal_p_tilde}  Assume that $\(a,b\)$ is an unlocked parabolic affine action, and $a$ is step-2. 
Consider  a map $p: \ZZ^2 \mapsto C^\infty(\mathbb T^d)$, and let $\{h_m\}_{m\in \mathcal U}$ be given, where $\mathcal U$ is some $\(\bar A, \bar B\)$-invariant set.  Define for every $(s,t)\in  \ZZ^2 $,  $\{(\tilde p(s,t))_m\}_{m\in \mathbb Z}$ via 
\begin{equation*}
(\tp(s,t))_m = h_{\bA^s\bB^tm} e(\bA^s\bB^tm,\a_{s,t})-h_m - (p(s,t))_m.
\end{equation*} 
Suppose that there exists $\cK>0$ and $0< \rho\leq r$, 
such that  we have 
$$
\text{either} \quad |(\tp(1,0))_m| \leq \cK  |m|^{-\rho} \text{ for all } m\in \mathcal U, \quad \text{or}  \quad |(\tp(0,1))_m| \leq \cK  |m|^{-\rho}\text{ for all } m\in \mathcal U.
$$
Then there exists a constant $c=c(A,B)>0$ such that for any $(s,t)\in \Z^2$ we have:
\begin{equation}\label{est_tp_C3}
 |(\tp(s,t))_m |\le c (|s|+|t|)^{d r} 
 \left( \max  \{ \|Lp((s, t),(1,0))\|_r , \|Lp((s, t),(0,1))\|_r \} +\cK \right) \, |m|^{-\rho+1}.
\end{equation}
\end{lemma}

\begin{proof} Suppose first that for all $m$ in an $\bA$-invariant set $\mathcal U$ we have $|(\tp(1,0))_m| \leq \cK  |m|^{-\rho}$. 
Denote  $f_m:= (p(1,0))_m$ and  $\tf_m:= (\tp(1,0))_m$ for brevity.

For the sequence $\{ h_m\}$ we formally have: 
$$\partial_{s,t} (\partial_{k,l}  h )_m=\partial_{k,l} (\partial_{s,t}  h)_m.
$$ 
To see this notice that
any smooth function $H$ one can verify, using only the commutativity relation $ab=ba$, that $\partial_{s,t} \partial_{k,l} H =\partial_{k,l} \partial_{s,t} H $. 
Hence, under the commutativity condition, the Fourier coefficients of a smooth function $H$ satisfy for each $m$: 
$\partial_{s,t} (\partial_{k,l} H)_m=\partial_{k,l} (\partial_{s,t} H)_m$. This implies the desired relation for the sequence of numbers $\{ h_m\}$
(this relation can be also verified directly).

By the definition of the set of numbers $(h_m)$,
$$
\begin{aligned}
(Lp((1,0), (s,t)))_m &=(\partial_{1,0} p(s,t))_m - (\partial_{s,t} f )_m  \\
&=(\partial_{1,0} p(s,t))_m - (\partial_{s,t} \partial_{1,0} h -\partial_{s,t} \tf )_m  \\
&=    ( \partial_{1,0}  \,       ( p(s,t) - \partial_{s,t}  h  )    )_m  +  (\partial_{s,t} \tf )_m \\
&= (\partial_{1,0} \, \tilde p(s,t))_m    +  (\partial_{s,t} \tf )_m      .
\end{aligned}
$$
Hence, 
$$ 
(\partial_{1,0} \, \tilde p(s,t))_m  =  (Lp((1,0), (s,t)))_m- (\partial_{s,t} \tf )_m .
$$
Since $Lp((1,0), (s,t))\in C^r$ and $\rho \leq r$, we have 
$$|(Lp((1,0), (s,t)))_m|\leq \|Lp((1,0), (s,t) ) \|_r |m|^{-\rho}.
$$
Let us estimate $(\partial_{s,t} \tf )_m= \tf_{\bA^{s}\bB^{t}  m} -\tf_m$.  
To bound $|\tf_{\bA^{s}\bB^{t}  m}|$ notice that, since the linear part of the action $\(a, b\)$ is parabolic, for some constant $c_0=c_0(A,B)$ we have:
$$
\| \bA^{-s} \bB^{-t} \|\leq c_0 (|s|+|t|)^d.
$$  
Hence, for any $(s, t)$ we  have:
$$
|m| \leq \| \bA^{-s} \bB^{-t}  \|\, | \bA^{s}\bB^{t} m |  \leq c_0  (|s|+|t|)^d |\bA^{s}\bB^{t}  m|,
$$ 
and thus $|\bA^{s}\bB^{t}  m|^{-\rho} \leq c_1 (|s|+|t|)^{dr} |m|^{-\rho}$, and therefore 
$$
\begin{aligned}
|(\partial_{s,t} \tf )_m |\leq & |\tf_{\bA^s\bB^t m}| + |\tf_{m}| \leq   \cK  ( |\bA^s\bB^t m|^{-\rho} +|m|^{-\rho}) \\
&\leq c_2 \cK   (|s|+|t|)^{d r}   |m|^{-\rho}
\end{aligned}
$$
for some  $c_1, c_2 >0$ only depending on $(A,B)$. Finally, we obtain:
$$ 
\begin{aligned}
|(\partial_{1,0} \, \tilde p(s,t))_m|  \leq &| (Lp((1,0), (s,t)))_m+ (\partial_{s,t} \tf )_m|   \\
\leq &\left(   \|Lp((1,0), (s,t) ) \|_r + c_2 \cK  (|s|+|t|)^{d r} \right)  \, |m|^{-\rho} .
\end{aligned}
$$
By Proposition \ref{c_est_easy} (2),  this implies that
$$ 
\begin{aligned}
|(\tp(s,t))_m|  \leq  &c_3   \left(   \|Lp((1,0), (s,t) ) \|_r  +c_1 \cK  (|s|+|t|)^{d r} \right) \, |m|^{-\rho+1} \\
 \leq &c  (|s|+|t|)^{d r}  \left(   \|Lp((1,0), (s,t) ) \|_r  + \cK   \right) \, |m|^{-\rho+1}
\end{aligned}
$$
for some $c=c(A,B)>0$. The case  $|(\tp(0,1))_m| \leq \cK  |m|^{-\rho}$  is similar, we just have to use $\partial_{0,1}$,  $g_m:= (p(0,1))_m$ and $\tg_m:= (\tp(0,1))_m$ instead of $\partial_{1,0}$, $f_m$ and $\tf_m$, respectively.
 \end{proof}

\bigskip

\color{black}

\subsubsection{Proof of Proposition \ref{prop_main_estimate_detail} in the case $m\in\cC_3$} \label{s_case3}



Let $m\in \cC_3$; denote $f=p(0,1)$. 
We define $ h_m$ by
\begin{equation}\label{def_h}
 h_m=
\begin{cases}
\Sigma^{+,A}_m (f), & m \in \cM(A), \\
\Sigma^{-,A}_m (f), & m\in \cN(A).
\end{cases}
\end{equation}

\bigskip

\begin{proposition}\label{prop_C3}
Assume that $\(a,b\)$ is unlocked step-$S$ parabolic affine action, where $a$ is step-2.   
Let a map $p: \ZZ^2 \mapsto C^\infty(\mathbb T^d)$ be given. For $m\in \cC_3$, define  $h_m$ by \eqref{def_h}. 

There exists a constant 
$\kappa=\kappa(\gamma , r, A,B)$ such that
\begin{equation}\label{est_h_C3}
 |h_m| \le  \kappa \|f\|_r  | m |^{-r+1}, 
\end{equation}
and, defining for each $(s,t)$ the number  $(\tp(s,t))_m$  as in formula \eqref{def_p_tilde}, i.e.,
\begin{equation*}\label{C3_def_tp}
(\tp(s,t))_m = h_{\bA^s\bB^tm} e(\bA^s\bB^tm,\a_{s,t})-h_m - (p(s,t))_m,
\end{equation*}
for $\eta=0.99/S$, for any $ r>8/ \eta $,  we have the estimate:
\begin{equation}\label{est_tp_C33}
 |\tp(s,t)_m|\le \kappa (|s|+|t|)^{d r}   \max \{  \|Lp((1,0),(0,1))\|_r , \|Lp((s, t),(1,0))\|_r \}  \, |m|^{-\eta r+ 9}.
\end{equation}
\end{proposition}


\begin{proof} The proof of this proposition is done in 3 steps similarly to that of Proposition \ref{l_est_case2}. Steps (i) and  (iii) of the proof rely on the same lemmas (applied with slightly different constants). The important difference lies in the proof of step (ii) that relies on Proposition  \ref{lemma.double.sums.case3}  on the control of the double sums along the dual orbit of a lowest point on an $\bar A$-orbit. Here are the steps:

(i) Estimate \eqref{est_h_C3}  follows from part {\it (3)} of Proposition \ref{c_est_easy}. 


(ii) Based on Proposition \ref{lemma.double.sums.case3}, we will show the following: 

\begin{lemma}\label{l_formal_f_tilde} Assume that $\(a,b\)$ is unlocked step-$S$ parabolic affine action, where $a$ is step-2, and let $m\in \cC_3$.  As before, denote $f= p(1,0)$, $\tilde f_m=(\tp(1,0))_m$,  and let $h_m$ be as in \eqref{def_h}. 
Then 
 there exists a constant $c=c(r, A,B)>0$ such that 
for $\eta=0.99/S$, if $m= \bm$ (i.e.,  $m$ is the lowest point on its $\bA$ orbit), we have:
\begin{equation}\label{est_f_tildeC3}
|\tilde f_m| \leq c  \|Lp((1,0),(0,1))\|_r  \, |m|^{-\eta r +  8};  
\end{equation}
if $m\neq \bm$, then $\tf_m=0$. 
\end{lemma}
(iii) Lemma \ref{l_formal_f_tilde}, followed by Lemma \ref{l_formal_p_tilde} with $\cK=c \|Lp((1,0),(0,1))\|_r$ and $\rho={-\eta r+ 8}$, implies \eqref{est_tp_C33} for arbitrary $(s,t)$.
\end{proof}
Thus it only remains now to show Lemma \ref{l_formal_f_tilde}.


\begin{proof}[Proof of Lemma \ref{l_formal_f_tilde}.]   By the definition of $h_m$ we have: 
\begin{equation}\label{def_f_tilde}
{ \tilde f_m} =h_{\bA m} e(\bA m,\a)-h_m - { f_m }= 
\begin{cases}
-\Sigma^A_m (f),  & m= \bm ,\\
0, & m\neq \bm. \\
\end{cases}
\end{equation}
Assume that $m= \bm$. 
 Denote $\phi:=Lp((1,0), (0,1))$ for brevity.
Note that for any  $m\in \cC_3$
the definition of $\phi$ implies directly that the  following holds in terms of the formal power series:  
\begin{equation}\label{formula_Sigma}
\Sigma_m^A (f) =
\sum_{l\geq 0}  \sum_{k\in \Z} \phi_{\bar A^{k}\bar B^l m}\la^{(k)}_{m}\mu^{(l)}_{ m} = 
-\sum_{l\leq -1}  \sum_{k\in \Z}  \phi_{\bar A^{k}\bar B^l m}\la^{(k)}_{ m}\mu^{(l)}_{ m} .
\end{equation}

Since $\phi\in C^r$, for each $m$, its $m$-th Fourier coefficient satisfies $|\phi_{m}| \leq  \|\phi \|_r |m |^{-r}$. Hence, 
$$
| \Sigma_m^A (f) | \leq
\|\phi \|_r  \sum_{l\geq 0}  \sum_{k\in \Z} |\bar A^{k}\bar B^l m |^{-r} = 
\|\phi \|_r  \sum_{l\leq -1}  \sum_{k\in \Z} |\bar A^{k}\bar B^l m |^{-r} .
$$ 
The desired estimate  \eqref{est_f_tildeC3} in this case follows directly  from the estimate of the above double sums, namely it is proved in Proposition \ref{lemma.double.sums.case3} that for $r$ sufficiently large
there exists a constant $c=c(r, A,B)>0$ such that at least one of the following holds:
$$
\sum_{k\in \Z} \sum_{l\geq 0} |\bA^k \bB^l m |^{-r} \leq c |m|^{-\eta r + 8},
\quad
\text{ } 
\quad
\sum_{k\in \Z} \sum_{l< 0} |\bA^k \bB^l m |^{-r} \leq c |m |^{-\eta r +  8}.
$$\end{proof}

\subsection{Application of Proposition \ref{prop_main_estimate} to truncated functions and vector fields}\label{s_end_prop}

 The application of Proposition \ref{prop_main_estimate}  to truncated functions is direct. Application to truncated vector fields requires an iteration process as a consequence of the fact that the linear part of the unperturbed action is parabolic. Similar iterative procedure has been used before in \cite{D}, \cite{DK2}.

\subsubsection{Truncated functions}
For a general smooth function $v$ on $\mathbb T^d$ (or a vector field) and for $N\in \mathbb N$ the truncation $T_Nv$ is obtained by cutting off the Fourier series of $v$  with index $\|n\|\ge N$. The residue operator is defined as $R_N:=Id-T_N$. 

For a map  $p: \mathbb Z^2\to C^\infty(\mathbb T^d)$, and $N\in \mathbb N$, define the truncation $T_Np$ by  $(T_Np)(s,t)= T_Np(s,t)$, and   $R_Np:= p-T_Np$.

Then the operators $T_N$ and the residue operators $R_N:=Id-T_N$ satisfy the following estimates for all $N\in \mathbb N$ and all $0<r\le r'$:

\begin{equation}\label{truncation_estimates}
\begin{aligned}
\|T_Nv\|_{r'}&\le C_{r, r'} N^{r'-r+d}\|v\|_r, \\
\|R_Nv\|_{r}&\le C_{r, r'} N^{r-r'+d}\|v\|_{r'}.
\end{aligned}
\end{equation}
Observe that  $ave(T_Nv)= ave(v)$.




\color{black}

The following statement is a direct application of our main technical result,  Proposition \ref{prop_main_estimate}, to the truncations. 

\begin{PProp}\label{...}
Let $\(a,b\)$ be an unlocked $(\ga,\tau)$-Diophantine parabolic step-$S$ affine action, where $a$ is step-2, and let $r>0$.  
There exist constants $\etan=\etan (S)>0$,  $\sigman=\sigman(\(a,b\))$ and  $\kappa=\kappa(r,\(a,b\))$ such that 
 for any $p: \ZZ^2 \mapsto C^\infty(\mathbb T^d)$ there exists $V: \mathbb Z^2\to \mathbb R$ such that for every fixed $N\in \mathbb N$ and the truncation $q=T_Np$, there exist $h\in C^\infty(\mathbb T^d)$ and  $\tilde q:\mathbb Z^2\to C^\infty(\mathbb T^d)$ satisfying  
$$q(s, t)= \partial_{s, t} h +\tilde q(s,t) + V (s, t), \,\, (s,t)\in \mathbb Z^2,$$
and the following estimates hold:
\begin{equation}\label{hqV}
 \begin{aligned}
&\|h\|_{r}\leq  \kappa \|q\|_{r+\sigman},\\
&\|\tq(s,t)\|_r\leq \|q(s,t)\|_r+  \kappa(|s|+|t|)^{rd}\|q\|_{r+\sigma},\\
&\|\tq(s,t)\|_{r} \leq  \kappa (|s|+|t|)^{d \rn} \|Lq\|_{\etan r+\sigman,N},
\\
& V(s,t)=ave(p(s, t))=ave (q(s, t)),
\end{aligned}
\end{equation}
where $\rn:= \etan r+\sigman$.
\end{PProp}

Note that in the above proposition,  $h$ and $\tilde q$ depend on $N$. 
 

\color{black}

\subsubsection{Truncated vector fields}\label{vf}
For a map  $\bfp: \mathbb Z^2\to  \rm Vect^\infty(\mathbb T^d)$ and $N\in \mathbb N$, define the truncation $T_N \bfp$ by  $(T_N \bfp)(s,t)= T_N \bfp(s,t)$, and   $R_N \bfp:= \bfp-T_N \bfp$. We have the same estimates for the operators $T_N$ and $R_N$ on vector fields as in \eqref{truncation_estimates}.

For  $\bfh\in \rm Vect^\infty(M)$  we define  
$$
D_{s,t}: \bfh\mapsto \bfh\circ \rho(s,t)-\rho_0(s,t)\bfh,
$$ 
where $(s,t)\in \mathbb Z^2$. The operator $\bfL$ is then defined by 
\begin{equation}\label{linL}
(\bfL{ \bfp})((k,l), (s,t))= D_{s, t}\bfp(k,l) -D_{k,l}\bfp{(s,t) }
\end{equation}
for any $(k, l), (s,t) \in \mathbb Z^2$.


Recall that the set $\res_N$ is defined in \eqref{R'}. For $\bfp: \ZZ^2\to\rm Vect^\infty(\mathbb T^d)$ and $r\ge 0$ let
\begin{align*}\|\bfp\|_{r}&:=\max\{  \| \bfp(1,0)\|_{r}, \bfp(0,1)\|_{r}\}, \\
 \|\bfL\bfp\|_{r, N}&:=\max\{  \|\bfL\bfp((1,0), (k, l))\|_{r}, \|\bfL\bfp((0,1), (k, l))\|_{r}:\, (k,l)\in \res_N\}.
\end{align*}


\begin{PProp}\label{vector fields}
Let $\(a,b\)$ be an unlocked step-$S$ parabolic  $(\ga,\tau)$-Diophantine  affine action, where $a$ is step-2, and let $r>0$.  There exist constants $\etan=\etan (S)>0$,  $\sigman=\sigman(\(a,b\))$ and  $\kappa=\kappa(r,\(a,b\))$ such that 
 for any $\bfp: \ZZ^2 \mapsto \rm Vect^\infty(\mathbb T^d)$ there exists $\bfV: \mathbb Z^2\to \mathbb R^d$ such that for every fixed $N\in \mathbb N$ and the truncation $q=T_Np$, there exist $\bfh\in \rm Vect^\infty(\mathbb T^d)$ and  $\tilde{\bfq}:\mathbb Z^2\to \rm Vect^\infty(\mathbb T^d)$ satisfying
$$\bfq(s,t)= D_{s,t}\bfh+\tilde{\bfq}(s,t)+ \bfV(s,t),  \,\, (s,t)\in \mathbb Z^2,$$
and the following estimates hold
\begin{equation}\label{hqVVV}
 \begin{aligned}
&\|\bfh\|_{r}\leq  \kappa \|\bfq\|_{r+\sigman},\\
&\|\tilde{\bfq}(s,t)\|_r\leq \|\bfq(s,t)\|_r+  \kappa(|s|+|t|)^{rd}\|\bfq\|_{r+\sigma},\\
&\tilde{\bfq}(s,t)\|_{r} \leq  \kappa (|s|+|t|)^{d \rn} \|\bfL\bfq\|_{\etan r+\sigman,N},
\\
&\|\bfV(1,0)\|\le \|\bfp(1,0)\|_0;\, \,  \|\bfV(0,1)\|\le \|\bfp(0,1)\|_0,
\end{aligned}
\end{equation}
where $\rn:= \etan r+\sigman$.
\end{PProp}

Note that in the above Proposition,  $\bfh$ and $\tilde{\bfq}$ depend on $N$.

\begin{proof}

Since $A$ and $B$ are commuting unipotent matrices, there exists a basis in $\mathbb R^d$ in which both of them are upper triangular. We choose this basis to represent $\bff$, $\bfg$ and $\bfh$ in coordinate form. Also, we use this basis to define the norm of an arbitrary $\bfh$: if $\bfh=(h_1, \dots, h_d)$ in the chosen coordinates, then we define $\|\bfh\|_r=\max\{\|h_i\|_r: \, i=1, \dots, d\}$. For an element $(s,t)\in \ZZ^2$, let $\mathcal A^{(s,t)}$ denote the matrix $A^sB^t$ in the chosen basis.   Notice that the matrix $\mathcal A^{(s,t)}$ is upper triangular. Moreover, it has a polynomial growth of coefficients with respect to $(s, t)$, so we can assume that any element of the matrix $\mathcal A^{(s,t)}$ has size at most $C(|s|+|t|)^S$, where $C$ is a constant depending on  $A$ and $B$,  and $S$ is the step of the action. 

Recall that in the beginning of this section,  for $(s,t)\in \ZZ^2$, we defined  the operator 
$$D_{s, t}\bfh:=\bfh\circ (a^sb^t)-A^sB^t\bfh.$$ {In the chosen basis, this operator will have the form}
 \begin{equation}\label{D} 
 D_{s, t}\bfh=\bar\partial_{s,t}\bfh + \mathcal A^{(s,t)}\bfh,
 \end{equation}
 where $\bar \partial_{s, t}$ denotes the "diagonal" operator acting on $\bfh$ coordinatewise by the operator $\partial_{s, t}$:
 $$ 
 D_{s,t}=\begin{pmatrix}
\pelta_{s,t}& a_{12}^{s,t} &a_{13}^{s,t} &\cdots& a_{1d}^{s,t}\\
0&\pelta_{s,t}&a_{23}^{s,t} &\cdots& a_{2d}^{s,t}\\
\dots&\dots&\dots&\dots&\dots\\
0&0&0&0&\pelta_{s,t}
\\
\end{pmatrix}.
$$
This upper triangular form allows us to construct $\bfh$  inductively, starting from its last coordinate, and then continuing upwards to the first coordinate. At each step of this inductive procedure, we use all the previously constructed coordinates of $\bfh$.

To keep track of the estimates during the induction, it is convenient to consider the operator $L$ in the upper triangular form as well. Thus, for any $\bfq$ we let $\bfq(s, t)=(q_1(s,t), \dots, q_d(s,t))$  be the coordinates of the vector field $\bfq(s, t)$ in the chosen basis, and introduce the following operator:\begin{equation} \label{Lop}  \bfL\bfq((k, l), (s,t))=
\begin{pmatrix}
\pelta_{k,l}& a^{k,l}_{12} &\cdots& a^{k,l}_{1d}\\
0&\pelta_{k,l} &\cdots& a^{k,l}_{2d}\\
\dots&\dots&\dots&\dots\\
0&0&0&\pelta_{k,l}
\\
\end{pmatrix} \begin{pmatrix}
q_1(s,t)\\
q_2(s,t)\\
\dots\\
q_d(s,t)
\\
\end{pmatrix}-
\begin{pmatrix}
\pelta_{s,t}& a_{12} ^{s,t}&\cdots& a_{1d}^{s,t}\\
0&\pelta_{s,t} &\cdots& a_{2d}^{s,t}\\
\dots&\dots&\dots&\dots\\
0&0&0&\pelta_{s,t}
\\
\end{pmatrix} \begin{pmatrix}
q_1(k, l)\\
q_2(k, l)\\
\dots\\
q_d(k, l)
\\
\end{pmatrix}.
\end{equation}

Given $\bfq(s, t)$, we first take out all of its averages and call the vector of averages by
$$
\bfV(s, t)=(ave(q_1(s, t)), \dots, ave(q_d(s, t))).
$$ 
Now we can work with $\bfq(s, t)$ assuming that its averages are 0. Then the last estimate of Proposition \ref{vector fields} follows directly from the definition of $\bfV(1,0)$ and $\bfV(0,1)$.

Let us proceed with the inductive construction of $\bfh$. We start with the last coordinate $q_d$. By applying Proposition \ref{...} to $q_d$ we obtain $h_d$ and $\tilde q_d$ such that:
\begin{equation}\label{hqVV}
 \begin{aligned}
 &q_d(s,t)=\partial_{s, t} h_d+\tilde{q}_d(s,t),\\
&\|h_d\|_{r}\leq  \kappa \|q\|_{r+\sigman}\le  \|\bfq\|_{r+\sigman},\\
&\|\tilde{ q_d}(s,t)\|_{r} \leq \kappa  (|s|+|t|)^{d\rn}\|{Lq_d}\|_{{\etan r+\sigman}, N}\le \kappa  (|s|+|t|)^{d\rn}\|{\bfL\bfq}\|_{{\etan r+\sigman}, N}.
\end{aligned}
\end{equation}

Now we turn to constructing $h_{d-1}$. 
First, define $h_d'(s, t):= a^{s,t}_{d-1, d}h_d$ for any $(s,t)\in \ZZ^2$ and observe that 
\begin{equation} \label{hprime} \|h_d'(s, t)\|_r\le |a^{s,t}_{d-1, d}|\|h_d\|_r\le \kappa (|s|+|t|)^d \| q_d\|_{r+\sigma}.\end{equation}

Since we have put \eqref{D} in the upper triangular form, we have to show the following.

 \begin{lemma}\label{step1} There exist $h_{d-1}$ and $\tilde{q}_{d-1}$ such that 
 \begin{align*} &(q_{d-1}-h_d')(s,t)= \partial_{s,t} h_{d-1} +\tilde{q}_{d-1}(s,t),\\
&\|h_{d-1}\|_{r}\leq  \kappa \|\bfq\|_{r+2\sigman},\\
&\|\tilde{q}_{d-1}(s,t)\|_r\le  
 \kappa (|s|+|t|)^{d\rn}  N^{d_2} \|\bfL\bfq\|_{\etan' r+\sigman',  N}, 
 \end{align*}
where $d_2:=\max\{d, d\bar r\}$, 
$\etan':=\etan^2$ and $\sigman':= \etan\sigman +\sigman$.
\end{lemma}

\begin{proof} 
By substituting the approximation of  $q_d$ of \eqref{hqVV} into the line $d-1$ of the operator of \eqref{Lop}, we get 
\begin{align*}
&\bfL\bfq((k, l), (s,t))_{d-1} =\\
&= Lq_{d-1}((k, l), (s,t)))+ a_{d-1, d}^{s,t}(\partial_{k,l} h_d+\tilde{q}_d(k,l))-a^{k,l}_{d-1, d}(\partial_{s,t} h_d+\tilde{q}_d(s,t)) \\
&=L(q_{d-1}-h_d')((k, l), (s,t)))+ a_{d-1, d}^{s,t}\tilde{q}_d(k,l)-a^{k,l}_{d-1, d}\tilde{q}_d(s,t).
\end{align*}


From this equality  and the estimates for $\tilde{q}_d$ in \eqref{hqVV},  we obtain the following estimate for all $(k, l), (s,t)\in \res_N$:
\begin{equation}\label{1.stepL}
\begin{aligned}
&\|L(q_{d-1}-h_d')((k, l), (s,t)))\|_r\le \\
&\le \|\bfL\bfq((k, l), (s,t)))_{d-1}\|_r + |a_{d-1, d}^{s,t}|\|\tilde{q}_d(k,l)\|_r+|a^{k,l}_{d-1, d}|\|\tilde{q}_d(s,t)\|_r\\
&\leq \|\bfL\bfq((k, l), (s,t)))_{d-1}\|_r+\kappa (|k|+|l|)^{d_2}(|s|+|t|)^{d_2}  \|Lq_d\|_{{\etan r+\sigman}, N},
\end{aligned} 
\end{equation}
where $d_2=\max\{d, d\rn\}$.
By taking the maximum of $\|L(q_{d-1}-h_d')((k, l), (s,t)))\|_r$ for $(k, l)\in \res_N$ and $(s,t)\in \{(1,0), (0, 1)\}$, we get from inequatity \eqref{1.stepL}: 
\begin{equation}\label{1.stepL'}
\|L(q_{d-1}-h_d')\|_{r,N}\le \|\bfL\bfq_{d-1}\|_{r,N}+ \kappa N^{d_2}\|Lq_d\|_{\etan r+\sigman, N}.
\end{equation}
Now we apply Proposition \ref{...} to $q_{d-1}-h_d'$, and find $h_{d-1}$ such that 
\begin{align*} 
&\|h_{d-1}\|_r \leq \\
&\leq   \kappa \| q_{d-1}-h_d'\|_{r+\sigma}= \kappa  \max\{\|q_{d-1}(1,0)-h_d'(1,0)\|_{r+\sigma}, \|q_{d-1}(0,1)-h_d'(0,1)\|_{r+\sigma}\}\\
&\le \kappa (\| q_{d-1}\|_{r+\sigma}+ \|h_d'\|_{r+\sigma})\le \kappa( \| q_{d-1}\|_{r+\sigma}+\|q_d\|_{r+2\sigma})\\
&\le \kappa \, \|\bfq\|_{r+2\sigma},
\end{align*}
where we used the estimate for $\|h_d'\|_{r+\sigma}$ from  \eqref{hprime}.

For the new error in coordinate $d-1$ we use the estimate obtained in Proposition \ref{...}  and \eqref{1.stepL'} to obtain 
\begin{align*}
&\|\tilde{q}_{d-1}(s,t)\|_r\le  \kappa (|s|+|t|)^{dr} \|L(q_{d-1}-h_d')\|_{{\etan r+\sigman}, N}\\
&\le  \kappa  (|s|+|t|)^{d\rn} (\|\bfL\bfq_{d-1}\|_{\etan r+\sigman,N}+ \kappa N^{d_2}\|Lq_d\|_{\etan (\etan r+\sigman)+\sigman, N})\\
&\le \kappa (|s|+|t|)^{d\rn}  N^{d_2} \|\bfL\bfq\|_{\etan' r+\sigman',  N},\\
\end{align*}
where $\etan':=\etan^2$ and $\sigman':= \etan\sigman +\sigman$.
\black\end{proof}

The rest of the inductive construction goes along the same lines. Namely, assume that $h_d$, $\tilde{q}_{d}$, $h_{d-1}$, $\tilde{q}_{d-1},\dots h_i$, $\tilde{q}_{i}$ have all been constructed, and assume that they satisfy the following estimates for every $j=i, \dots, d$ and $(s, t)\in  \ZZ^2$:
\begin{align*}
\|h_j\|_r&\le \kappa \| \bfq\|_{r+(d-j+2)\sigma},\\
\|\tilde{q}_{j}(s,t)\|_r&\le  \kappa (|s|+|t|)^{d\rn} N^{(j+1)d_2}\|\bfL\bfq\|_{\eta_{j+1}r+\sigma_{j+1}, N},
\end{align*}
where $\eta'_{j+1}:= \eta\cdot \eta_j$,  $\sigma_{j+1}:= \eta\sigma_{j}+\sigma$  and  $d^{(j+1)}:=\max\{d, d^{(j)}\}$. 

In $(i-1)$-st equation in \eqref{Lop} we substitute all the $q_j$ coordinates with $q_j(s, t)= \partial_{s,t}h_j+\tilde q_j(s, t)$.    Then the $(i-1)$-st equation in \eqref{Lop} becomes: 
\begin{equation*}\label{L_i}
\begin{aligned}
\bfL&\bfq((k, l), (s,t)))_{i-1}= Lq_{i-1}((k, l), (s,t)))+ \sum_{j=i}^da^{s,t}_{1, j}q_{j}(k, l)-\sum_{j=i}^da^{k,l}_{j, d}q_j(s,t)\\
&= Lq_{i-1}((k, l), (s,t)))+ \sum_{j=i}^da^{s,t}_{1, j}(\partial_{k,l}h_j+\tilde q_j(k, l))-\sum_{j=i}^da^{k,l}_{j, d}(\partial_{s,t}h_j+\tilde q_j(s,t))\\
&= L(q_{i-1}-h'_i)((k, l), (s,t))) + \sum_{j=i}^da^{s,t}_{1, j}\tilde q_j(k, l)-\sum_{j=i}^da^{k,l}_{j, d}\tilde q_j(s,t),
\end{aligned}
\end{equation*}
where we defined $h'_i(s, t):=\sum_{j=i}^da^{s,t}_{j, d}h_j$.

Now we do the same as in Lemma \ref{step1}: we apply Proposition \ref{...} to $q_{i-1}-h'_i$ to obtain $h_{i-1}$ and $\tilde q_{i-1}$, such that 
$$
q_{i-1}-h'_i= \partial_{s,t} h_{i-1}+ \tilde q_{i-1}.
$$ 
The same procedure as in  Lemma \ref{step1} gives:
\begin{align*}
\|h_j\|_r&\le \kappa \| \bfq\|_{r+(d-i+2)\sigman},\\
\|\tilde{q}_{i-1}(s,t)\|_r&\le  \kappa (|s|+|t|)^{d\rn}\, N^{(i+1)d_2}\|{\bf L}\bfq\|_{{\etan_{i+1}r+\sigman_{i+1}}, N},
\end{align*}
where $\eta'_{i+1}:= \etan\cdot \etan_i$,  $\sigman_{i+1}:= \etan\sigma_{i}+\sigman$. 
\color{black}

After passing through all the $d$ steps we redefine the constants $\sigman$ and  $\etan$. Namely, will have a multiplicative loss $\etan:= \etan^d$ and an additive loss  $\sigman:=\max\{\etan^d\sigman + d\cdot \sigman, (d+2)\sigman\}$ of the number of the derivatives for the newly constructed $\tilde \bfq$ with respect to $\bfL\bfq$. In other words, for every  $j=1, \dots, d$, we have with the newly constructed $\sigman$ and $\etan$ the following estimates:
\color{black}
\begin{align*}
\|h_j\|_r&\le \kappa \| \bfq\|_{r+{\sigman}},\\
\|\tilde{q}_{j}(s,t)\|_r&\le  \kappa \,(|s|+|t|)^{d\rn} N^{D_1}\|{\bf L}\bfq\|_{\etan r+\sigman, N},
\end{align*}
where  $D_1=D_1(d, r)$. \color{black}
By defining $\bfh= (h_1, \dots h_d)$ and $\tilde \bfq=(\tilde q_1, \dots, \tilde q_d)$ we obtain the estimates claimed in the Proposition \ref{vector fields} for $\bfh$ and $\tilde\bfq(s,t)$.  
 
The "linear" estimate for $\|\tilde{\bfq}(s,t)\|_r$ follows directly from the fact that by construction ${\bfq}(s,t)= D_{s,t} \bfh+\tilde{\bfq}(s,t)$, so the estimate for $\bfh$ implies:
\begin{align*}
\|\tilde{\bfq}(s,t)\|_r&\le  \|D_{s,t} \bfh\|_r+\|{\bfq}(s,t)\|_r\le (|s|+|t|)^{dr}\|\bfh\|_r+\|{\bfq}(s,t)\|_r \\
&\le \kappa (|s|+|t|)^{dr}\| \bfq\|_{r+\sigma}+\| \bfq(s,t)\|_{r}.
\end{align*}
By denoting $D=d(d+1)$, we obtain the estimates claimed in the Proposition \ref{vector fields}.
\end{proof}

\subsection{Proof of Proposition \ref{Main iteration step}}
\label{proof_of_main_iteration_step}

Recall that in our setup for a perturbation $\(F, G\)$ of the action $\(a, b\)$, we define the map $\bfp: \mathbb Z^2\to \rm Vect^{\infty}(\mathbb T^d)$ by $\bfp(k, l):= F^kG^l-a^kb^l$, $(k, l)\in \mathbb Z^2$. With  the notations $\bfp(1,0)= \bff$ and $\bfp(0,1)= \bfg$, the action $\tilde\rho$ is generated by the maps  $a+\bff$ and $b+\bfg$. 

Let $N$ be fixed, and let $\bfq_N=T_N\bfp$ be the truncation of $\bfp$. 
First we define $\bfV(k, l)= -ave (\bfp(k, l))$. (With a little abuse of notation) denote $\bfV=\bfV(1, 0):= -ave(\bff)$ and $\bfW=\bfV(0, 1)= -ave(\bfg)$. It is then clear that the last estimate in \eqref{main-est} holds. 

Now we apply Proposition \ref{vector fields} to $\bfp$ and its truncation $\bfq_N=T_N\bfp$ in order  to obtain $\bfh_N$.   Then as in Proposition \ref{vector fields} we define $\tilde \bfq_N(k, l):=\bfq_N(k, l)-D_{k, l}\bfh_N +\bfV(k,l)$, and so  
\begin{equation}\label{tildep}
\tilde \bfp_N(k,l):=\bfp(k, l)-D_{k, l}\bfh_N +\bfV(k,l)=\tilde \bfq_N+ R_N\bfp.
\end{equation}
Finally, we let  
$$
\tilde \bff_N=\tilde\bfp_N(1,0), \quad \tilde \bfg_N=\tilde \bfp_N(0,1).
$$ 
 We begin by estimating $\bfp(k,l)=(a+\bff)^k\circ (b+\bfg)^l - a^kb^l$. 

\begin{lemma} \label{p(k, l)} For $(k, l)\in \ZZ^2$,  
\begin{equation} \label{pkl}
\| \bfp(k,l)\|_{r}\leq C_r (|k|+|l|)^{3d} \max\, \{ \| \bff\|_{r}, \| \bfg\|_{r}  \}.
\end{equation}

\end{lemma}
\begin{proof}

Developing the expression for $\bfp(k,l)$ as $\bfp(k,l)=(a+\bff)^k\circ (b+\bfg)^l-a^kb^l= a(a+\bff)^{k-1}\circ (b+\bfg)^l +\bff\circ (a+\bff)^{k-1}\circ (b+\bfg)^l -a^kb^l$ and continuing inductively, one gets 
$$
\begin{aligned}
\bfp(k,l)&=\sum_{j=0}^{k-1} s_j \, \bff\circ (a+\bff)^{j} +\sum_{j=0}^{l-1} t_j \, \bfg\circ (a+\bff)^{k} \circ (b+\bfg)^{j}\\
&=\sum_{j=0}^{k-1} s_j \, f\circ (a^j+\xi_j ) +\sum_{j=0}^{l-1} t_j \, g\circ (a^{k}b^j+\eta_j). 
\end{aligned}
$$
Here $s_j$, $t_j$ are appropriate compositions of type $A^iB^j$ with $i\leq k$, $j\leq l$.  
The sums contain $|k|+|l|$ terms, each coefficient can be estimated by $\|A^kB^l\|\leq C(|k|+|l|)^S$. 
The $r$-norm of each of these terms can be estimated by 
$c  \|A^kB^l \|^r \max\, \{ \| \bff\|_{r}, \| \bfg\|_{r}  \} \leq cr(|k|+|l|)^S \max\, \{ \| \bff\|_{r}, \| \bfg\|_{r}  \}$.  Hence, 
$\| \bfp(k,l)\|_{r}\leq C_r (|k|+|l|)(|k|+|l|)^{2S} \max\, \{ \| \bff\|_{r}, \| \bfg\|_{r}  \} \leq C_r (|k|+|l|)^{3S} \max\, \{ \| \bff\|_{r}, \| \bfg\|_{r}  \}$.  Since $S\le d$, we can then bound this from above by  $C_r (|k|+|l|)^{3d} \max\, \{ \| \bff\|_{r}, \| \bfg\|_{r}  \}$.
\end{proof}

 Let us now move to proving that the main estimates  
 \eqref{main-est} in of Proposition \ref{Main iteration step} hold for $\bfh_N$, $\tilde \bff_N$ and $\tilde \bfg_N$. 
 We start with $\bfh_N$. Directly from the first estimate in  \eqref{hqVVV} of Proposition \ref{vector fields}, and truncation estimate  \eqref{truncation_estimates}   it follows that:
\begin{equation*}
\begin{aligned}
\|\bfh_N\|_{r}&\leq  \kappa \|\bfq_N\|_{r+\sigma} \le  \kappa \, N^\sigma \|\bfp\|_{r}.\\
\end{aligned} \label{hhN}
\end{equation*}

The "linear" estimate for $\tilde \bff_N$ and $\tilde \bfg_N$ (the third estimate in Proposition \ref{Main iteration step}) follows directly from the estimate for $\bfh_N$:
$$\| \tilde \bff_N \|_r= \|\tilde \bfp_N(1,0)\|_r\le \|\bfp_N(1,0)-D_{1,0}\bfh_N +\bfV\|_r\le C_rN^{D'} \Delta_r.$$
Estimate for $\tilde \bfg_N= \tilde \bfp_N(0,1)$ follows exactly in the same way.

Now we will use the fact that $a^kb^l+ \bfp(k, l)$ is a commutative action, in order to obtain the quadratic estimate for the error 
$\tilde\bfq$. This is done in two steps: the first one is to to show that $\|\bfL\bfp\|_{r,N}$ is quadratic by using the fact that 
$a^kb^l+ \bfp(k, l)$ is a commutative action. 
The second one is to compare $\|\bfL\bfq \|_{r,N}$ to $\|\bfL\bfp \|_{r,N}$. 



Recall that the operator $\bfL$ acts on $\bfp$ by the formula 
$\bfL{ \bfp}((k,l), (s,t))= D_{s, t}\bfp(k,l) -D_{k,l}\bfp{(s,t) }$
for any $(k, l), (s,t) \in \mathbb Z^2$.

 
\begin{lemma}\label{bfp} For any $N\in \mathbb N$ and $r\ge 0$ the following holds: 
$$
\|\bfL\bfp \|_{r,N}\le C_r N^{4d} \Delta_{r+1}\Delta_0.
$$ 
\end{lemma}
\begin{proof}
First, we notice that the commutativity of the action $a^kb^l+\bfp(k, l)$ implies that 
$$
(a^kb^l+ \bfp(k, l))\circ (a^sb^t+\bfp(s,t))=(a^sb^t+\bfp(s,t)))\circ (a^kb^l+ \bfp(k, l)).
$$ 
Therefore, the operator $\bfL$, besides its linear form, has also a non-linear expression (on the right below), in particular: 
\begin{equation}
\begin{aligned}
&D_{(s,t)}\bfp(k, l)-D_{k, l}\bfp(s,t)=\bfL\bfp ((k,l), (s,t))\\
&= \bfp(k, l)(a^sb^t+\bfp(s,t))- \bfp(k, l)(s,t) +  \bfp(s,t)(a^kb^l+\bfp(k,l))- \bfp(s,t)(a^kb^l).\\
\end{aligned}
\end{equation}

From the non-linear expression for $\bfL\bfp ((k,l), (s,t))$ and the classical estimates for compositions (\cite{Hormander} Theorem A.8) we obtain the "quadratic" estimate: 
 \begin{equation*}\label{quadLest}
\begin{aligned}
\|\bfL\bfp ((k,l), (s,t))\|_r&=\|\bfp(k, l)(a^sb^t+\bfp(s,t))- \bfp(k, l)(a^sb^t)\|_r  \\
&+ \|\bfp(s,t)(a^kb^l+\bfp(k,l))- \bfp(s,t)(a^kb^l)\|_r\\
&\le C_r N^d (\|\bfp(k, l)\|_{r+1}\|\bfp(s,t)\|_0+\|\bfp(k, l)\|_{0}\|\bfp(s,t)\|_{r+1})\\
&\le C_r N^{d} (|k|+|l|)^{3d} (|s|+|t|)^{3d}\Delta_{r+1}\Delta_0.\
\end{aligned}
\end{equation*}
Recall that  $\|\bfL\bfp\|_{r, N}:=\max\{  \|\bfL\bfp((1,0), (k, l))\|_{r}, \|\bfL\bfp((0,1), (k, l)\|_{r}:\, (k,l)\in \res_N\}$ and that for $ (k,l)\in \res_N$ we have that  $C(|k|+|l|)\le N$. 
From the above inequality, by taking  maximum over $(k,l)\in \res_N$, we get the required estimate for  $\|\bfL\bfp\|_{r, N}$.
\color{black} 
\end{proof}

Next we will use the following basic estimate of the norm of the operator $D_{k,l}$:  

\begin{lemma}\label{dkl} For  $ (k,l)\in \res_N$ and $r\ge 0$ 
$$
\|D_{k,l}\bfp(s,t)\|_r\le C_r N^{dr} \|\bfp(s,t)\|_r .
$$
\end{lemma}
\begin{proof}
To show this bound we only need to observe that for any affine map $a: x\to Ax + \alpha$ and any function $\phi$ the norm $\|\phi \circ  a\|_r$ is bounded above by $\|A\|^r\|\phi\|_r$. Since in the operator $D_{k,l}$ we compose with $a^kb^l$, which has linear part $A^kB^l$, and the norms of matrices $A^kB^l$ are (up to a positive constant) bounded by $(|k|+|l|)^d$, then $ (k,l)\in \res_N$ directly implies the required estimate, because for $ (k,l)\in \res_N$, $C(|k|+|l|)\le N$ by  Lemma \ref{l_tech_reson}.
\end{proof}


\begin{lemma}\label{bfq} For any $N\in \mathbb N$ and $r'\ge r\ge 0$ the following holds: 

$$\|\bfL\bfq\|_{r,N}\le  C_{r} N^{r+5d} \Delta_{1}\Delta_0 + C_{r, r'} N^{r-r'+2d_2} \Delta_{r'},$$ 
where $d_2=d_2(r):=max\{d, dr\}$. 

\end{lemma}

\begin{proof}


Because $\bfL$ is a linear operator we have the following: 
\begin{equation}\label{linLL}
\bfL\bfp ((k,l), (s,t))=\bfL (T_N\bfp) ((k,l), (s,t)) + \bfL (R_N\bfp) ((k,l), (s,t)).
\end{equation}
The term $ \bfL (R_N\bfp) ((k,l), (s,t))$ for $ (k,l)\in \res_N$ and $(s,t)\in \{(1,0), (0,1)\}$ is estimated by using the estimates on the truncation operators, \eqref{truncation_estimates},  the linear form \eqref{linLL} of $\bfL$ and Lemmas \ref{dkl} and \ref{p(k, l)} (again, recall that for $ (k,l)\in \res_N$ we have that $C(|k|+|l|)\le N$):
\begin{equation*}
\begin{aligned}
 \|\bfL (R_N\bfp) ((k,l), (s,t))\|_r&\le \|D_{s,t} R_N\bfp(k,l)-D_{k,l} R_N\bfp(s,t)\|_r\\
 &\le C_r(|s|+|t|)^{dr}\|R_N\bfp(k, l)\|_r+C_r(|k|+|l|)^{dr}\|R_N\bfp(s,t)\|_r \\
 &\le C_{r,r'}(N^{r-r'}\|\bfp(k, l)\|_{r'}+ N^{dr} N^{r-r'}\|\bfp(s,t)\|_r) \\
&\le C_{r,r'}(N^{r-r'+3d}\|\bfp\|_{r'}+ N^{r-r'+dr}\|\bfp\|_{r'})\\
&\le C_{r,r'}N^{r-r'+d_3}\|\bfp\|_{r'},
\end{aligned}
\end{equation*}
where $d_3=d_3(r):=\max\{3d, dr)\}$.

Using this and \eqref{linLL}, as well as \eqref{truncation_estimates} and Lemma \ref{bfp}, we obtain for $0\le r''\le r\le r'$ we have:
\begin{equation*}
\begin{aligned}
\| \bfL\bfq ((k,l), (s,t)) \|_r&\le \|\bfL\bfp ((k,l), (s,t))\|_r+ \|\bfL(R_N\bfp) ((k,l), (s,t))\|_r\\
&\le \|T_N\bfL\bfp ((k,l), (s,t))\|_r+ \|R_N\bfL\bfp ((k,l), (s,t))\|_r\\
&\,\, \,\quad  + \|\bfL(R_N\bfp) ((k,l), (s,t))\|_r\\
&\le C_{r, r''} N^{r-r''+d} \|\bfL\bfp ((k,l), (s,t))\|_{r''}+ C_{r, r'}N^{r-r'+d} \|\bfL\bfp ((k,l), (s,t))\|_{r'} \\
&\,\, \,\quad +C_{r,r'}N^{r-r'+d_3} \Delta_{r'}\\
&\le  C_{r, r''} N^{r-r''+5d} \Delta_{r''+1}\Delta_0 + C_{r,r'}N^{r-r'+d_3} \Delta_{r'}, \\
\end{aligned}
\end{equation*}
where we applied the estimate   $C(|k|+|l|)\le N$ that is valid for all $ (k,l)\in \res_N$. 
Now by setting $r''=0$, and by taking the maximum over $ (k,l)\in \res_N$ and $(s,t)\{(1,0), (0,1)\}$, we obtain the required estimate for $\|\bfL\bfq\|_{r, N}$. 
\end{proof}


 From the third estimate in  Proposition \ref{vector fields} and Lemma \ref{bfq}, for some fixed constant $D>0$, for any $r'>0$ and for a fixed $d_2=d_2(\sigma)$, we have: 
\begin{equation*}
\begin{aligned}
\|\tilde \bfq_N(1,0)\|_0&\le \kappa N^D \|{\bf Lq}_N\|_{\sigma, N}\\
&\le  \kappa N^{D+5d+\sigma} \Delta_{1}\Delta_0+C_{r'}N^{{D}-r'+\sigma+ d_3} \Delta_{r'}\\
& \le  \kappa N^{D'} \Delta_{1}\Delta_0+C_{r'}N^{-r'+D'} \Delta_{r'},
\end{aligned}
\end{equation*}
where we define $D'=D+5d+d_3+\sigma$. Notice that $d_3=d_3(\sigma)$, and consequently $D'$, depends only on the action $\(a,b\)$ and $d$.
Recall that in \eqref{tildep} we defined $\tilde \bfp_N=\tilde \bfq_N+ R_N\bfp$ for every $N$.  
By combining the above estimate and the truncation estimates \eqref{truncation_estimates} 
for the operator $R_N$, we get a similar bound, with possibly new constants $\kappa$ and $C_{r'}$:
\begin{equation*}
\|\tilde \bff_N\|_0=\|\tilde \bfp_N(1,0)\|_0\le \kappa N^{D'} \Delta_{1}\Delta_0+C_{r'}N^{-r'+D'} \Delta_{r'}.
\end{equation*}
Finally, we declare the new $D$ to be the $D'$. 
Estimates for $\tilde \bfg_N= \tilde \bfp_N(0,1)$ are proved in exactly the same way. This completes the proof of all the estimates claimed in Proposition \ref{Main iteration step}. \hfill $\Box$

\bigskip

\noindent {\sc \bleu1 Acknowledgment.} The authors are grateful to Livio Flaminio and Giovanni Forni for stimulating discussions on the subject. The first author was supported by the Swedish Research Council grant VR 2019-04641. The second author was supported by the NSF grant DMS-2101464.

\end{document}